\documentclass[12pt]{article}

\usepackage{amsmath,amscd,amsfonts,amsthm, amssymb, graphicx, mathdots, xypic}
\input{xy}
\usepackage{color}
\usepackage{cleveref}

\newcommand{\shrinkmargins}[1]{
  \addtolength{\textheight}{#1\topmargin}
  \addtolength{\textheight}{#1\topmargin}
  \addtolength{\textwidth}{#1\oddsidemargin}
  \addtolength{\textwidth}{#1\evensidemargin}
  \addtolength{\topmargin}{-#1\topmargin}
  \addtolength{\oddsidemargin}{-#1\oddsidemargin}
  \addtolength{\evensidemargin}{-#1\evensidemargin}
  }

\shrinkmargins{.7}

\DeclareMathOperator{\rk}{rk}
\DeclareMathOperator{\et}{et}

\DeclareMathOperator{\Conf}{Conf}

\DeclareMathOperator{\CHur}{CHur}

\DeclareMathOperator{\Hur}{Hur}

\DeclareMathOperator{\Hom}{Hom}

\DeclareMathOperator{\SL}{SL}
\DeclareMathOperator{\Sub}{Sub}

\DeclareMathOperator{\GL}{GL}
\DeclareMathOperator{\PSL}{PSL}

\DeclareMathOperator{\Spec}{Spec}

\DeclareMathOperator{\Sym}{Sym}

\DeclareMathOperator{\Gal}{Gal}
\DeclareMathOperator{\Frob}{Frob}

\DeclareMathOperator{\Aut}{Aut}

\DeclareMathOperator{\gr}{gr}
\DeclareMathOperator{\br}{br}
\DeclareMathOperator{\tr}{tr}
\DeclareMathOperator{\cok}{cok}
\DeclareMathOperator{\ind}{ind}

\DeclareMathOperator{\Tor}{Tor}

\DeclareMathOperator{\Ext}{Ext}

\DeclareMathOperator{\Ni}{Ni}

\newcommand{\UU} {\mathcal{U}}

\newcommand{\Hn}{\mathsf{Hn}}
\newcommand{\co}{co\,}

\newcommand{\field}[1]{\mathbb{#1}}
\newcommand{\eps}{\epsilon}

\newcommand{\Q}{\field{Q}}

\newcommand{\Zhat}{\hat{\field{Z}}}
\newcommand{\Z}{\field{Z}}
\newcommand{\A}{\field{A}}

\newcommand{\F}{\field{F}}

\newcommand{\Fqbar}{\bar{\field{F}}_q}
\newcommand{\R}{\field{R}}
\newcommand{\C}{\field{C}}

\newcommand{\YD}{\mathcal{YD}}

\newcommand{\Ahat}{\hat{A}}

\newcommand{\HHH}{\mathbb{H}}
\renewcommand{\P}{\field{P}}

\newcommand{\KK} {\mathcal{K}}
\newcommand{\LL} {\mathcal{L}}
\newcommand{\XX} {\mathcal{X}}
\newcommand{\ra}{\rightarrow}

\newcommand{\OO}{\mathcal{O}}

\newcommand{\mA}{\mathfrak{A}}
\newcommand{\grA}{\mathfrak{A}^{\gr}}
\newcommand{\mB}{\mathfrak{B}}

\newcommand{\mC}{\mathfrak{C}}

\newcommand{\mD}{\mathfrak{D}}
\newcommand{\mE}{\mathfrak{E}}

\newcommand{\mS}{\mathfrak{S}}

\newcommand{\id}{\mbox{id}}
\newcommand{\im}{{\rm im \,}}

\newcommand{\mat}[4]{\left[\begin{array}{cc}#1 & #2 \\
                                         #3 & #4\end{array}\right]}

\newcommand{\inj}{\hookrightarrow}
\newcommand{\tensor} {\otimes}

\newcommand{\bs}{\backslash}

\newcommand{\tautilde}{\widetilde{\tau}}

\newcommand{\beq}{\begin{displaymath}}
\newcommand{\eeq}{\end{displaymath}}
\newcommand{\beqn}{\begin{equation}}
\newcommand{\eeqn}{\end{equation}}

\newcommand{\nth}{\textsuperscript{th}}
\newcommand{\st}{\textsuperscript{st}}

\theoremstyle{plain}
\newtheorem{thm}{Theorem}[section]
\newtheorem{prop}[thm]{Proposition}
\newtheorem{cor}[thm]{Corollary}
\newtheorem{lem}[thm]{Lemma}

\theoremstyle{definition}
\newtheorem{defn}[thm]{Definition}
\newtheorem{conj}[thm]{Conjecture}
\newtheorem{hyp}[thm]{Hypothesis}
\newtheorem{exmp}[thm]{Example}

\theoremstyle{remark}
\newtheorem{rem}[thm]{Remark}

\title{Fox-Neuwirth-Fuks cells, quantum shuffle algebras, and Malle's conjecture for function fields}
\author{Jordan S. Ellenberg, TriThang Tran and Craig Westerland}

\begin{document}
\bibliographystyle{amsalpha}

\maketitle

\begin{abstract}

The purpose of this paper is to prove the upper bound in the weak Malle conjecture on the distribution of finite extensions of $\F_q(t)$ with specified Galois group.  As in \cite{evw}, our result is based upon computations of the homology of braid groups with certain coefficients.  However, the approach in this paper is new, relying on a connection between the cohomology of Hurwitz spaces and the cohomology of quantum shuffle algebras and Nichols algebras.
\end{abstract}

\section{Introduction}
\label{intro}

The fundamental objects of algebraic number theory are global fields, which fall into two types:  number fields (finite algebraic extensions of $\Q)$ and function fields (finite algebraic extensions of the field of rational functions $\F_q(t)$ over some finite field $\F_q$).  If $K$ is a global field and $L/K$ a finite extension of $L$ of degree $m$, the two basic invariants of $L/K$ are the {\em Galois group} $\Gal(L/K)$, a transitive subgroup of $S_m$, and the {\em discriminant} $\Delta_{L/K}$, a real number measuring the extent to which $L/K$ is ramified.  An old result of Hermite guarantees that the number of isomorphism classes of degree-$m$ extensions of $K$ with Galois group $G$ and having $|\Delta_{L/K}| < X$ is finite; it is thus natural to ask how {\em many} isomorphism classes of such extensions there are.  We denote this number by $N_G(K,X)$. 

In 2002, Malle organized much of the previous work on this problem into a single governing conjecture.  If $g$ is an element of $S_m$, we define the \emph{index} of $g$ to be 
$$\ind(g) = m - \# (\{1, \dots, m\} / \langle g \rangle).$$
We define $a(G)$ to be $[\min_{G \setminus \{1 \}} \ind(g)]^{-1}$.  For instance, if $G = S_m$, the minimal index is $1$, realized by transpositions, and so $a(S_m) = 1$.

\begin{conj}[\cite{malle02}]  Let $K$ be a number field.  For any $\eps > 0$, there exist constants $c_1(K,G)$ and $c_2(K,G,\eps)$ such that
\beq
c_1(K,G) X^{a(G)} \leq N_G(K,X) \leq c_2(K,G,\eps) X^{a(G) + \eps}.
\eeq
\label{co:weakmalle}
\end{conj}

Later, in \cite{malle}, Malle proposed a more refined conjecture, that
\beq
N_G(K,X) \sim c(K,G) X^{a(G)} \log X^{b(K,G)-1}
\eeq
for a specified constant $b(K,G) \in \Z_{\geq 1}$.  We refer to Conjecture~\ref{co:weakmalle} as the {\em weak Malle conjecture} and the more refined version as the {\em strong Malle conjecture}.

Malle stated his conjecture only for extensions of a number field $K$, but it makes perfect sense when $K$ is a general global field, and nowadays ``Malle's conjecture" is commonly understood to have this scope.  We will assume, {\em here and for the rest of this paper}, that $K$ is a field whose characteristic is prime to $\#G$.   Malle's conjecture still makes sense without this tameness restriction, but we have somewhat less certainty about its correctness.  

Malle's conjecture has been proved in many special cases.  The strong Malle conjecture holds for abelian Galois groups \cite{wright}; for transitive permutation groups of degree at most $4$, it is either proven or strongly supported by numerical evidence in each case~\cite{cohen}.  The strong conjecture (over number fields or function fields) is known to be not quite correct as formulated here with respect to the power of $\log(X)$; Kl\"{u}ners gave a counterexample in \cite{klunerscounter}, and T\"{u}rkelli proposed a modified conjecture with a slightly different value of $b(K,G)$ in \cite{turkelli:malle}; note that these counterexamples do not contradict the weak Malle conjecture.  In \cite{evw}, strong Malle was shown to hold up to multiplicative constants over $\F_q(t)$ when $G$ is a generalized dihedral group of order prime to $q$; this is the case that's relevant to the Cohen-Lenstra conjectures over function fields.\footnote{More precisely, \cite{evw} addresses the case where the local monodromy consists entirely of involutions.}  The theorem of Davenport-Heilbronn \cite{dh} shows that strong Malle holds for $S_3$-extensions of $\Q$, and the later work of Datskovsky and Wright~\cite{datskovsky} extends this not only to $\F_q(t)$ but to an arbitrary global field of characteristic prime to $6$.  The work of Bhargava and his collaborators \cite{bhargava-4, bhargava-5,bhargava15} show that strong Malle holds for the groups $S_4$ and $S_5$ for arbitrary global fields of characteristic other than $2$, and even explicitly computes the constant $c(K,G)$.  There is recent work of Koymans and Pagano in the case where $G$ is nilpotent~\cite{koymanspagano21}, and of Wang and collaborators where $G$ is a direct product of a small finite group with an abelian group~\cite{jiuyacompositio},\cite{mttw}.  See \cite[\S 10.4,10.5]{wood:survey} for a fuller description of what is currently known about Malle's conjecture and its variants.  For general $G$, however, very little is known, beyond some upper and lower bounds for the case $G=S_m$ which are very far from Malle's prediction~(\cite{evcounting},\cite{bhargavalowerbounds},\cite{rlothorne}.) 

The goal of the present paper is to prove that the upper bound in the weak Malle conjecture over $\F_q(t)$ holds for all choices of $G$ and all $q$ sufficiently large relative to $G$.

\begin{thm} \label{malle_intro0}  For each integer $m$ and each transitive $G \leq S_m$, there are constants $C(G), Q(G)$, and $e(G)$ such that, for all $q>Q(G)$ coprime to $\#G$ and all $X>0$, 
$$N_G(\F_q(t),X) \leq C(G) X^{a(G)} \log(X)^{e(G)}$$
\end{thm}

More generally, we prove in Theorem~\ref{th:mainmalle} an upper bound for the number of extensions of $\F_q(t)$ with a given Galois group and conditions on the local monodromy at the ramified places of $\F_q(t)$.  We note that any nontrivial {\em lower} bound for $N_G(\F_q(t),X)$ for general $G$ would give a positive answer to the inverse Galois problem for $\F_q(t)$, in a strong quantitative sense; this result is beyond the reach of our techniques for now, apart from a few cases outlined in \S 8 where circumstances are especially favorable.  

The exponent $e(G)$ in Theorem~\ref{malle_intro0} may be taken to be $d-1$, where $d$ is the Gelfand-Kirillov dimension of a certain graded ring $R$: this ring may be presented either as a ring of components of the family of Hurwitz spaces described below, or as a free braided commutative algebra on a braided vector space $V$ determined by $G$.  The value of $e(G)$ provided by our theorem is obtainable by a finite combinatorial computation on $G$; in any event, it is bounded above (typically very non-sharply) by $|G|-1$.

The proof of Theorem~\ref{malle_intro0} involves a fairly broad range of techniques.  On the one hand, we address the counting questions at hand by relating them to counting points on certain moduli spaces over finite fields, which involves the arithmetic geometry of moduli spaces and their etale cohomology.  In turn, we understand this etale cohomology by relating it to the singular cohomology of analogous topological spaces that can be studied by the machinery of homotopy theory.  So far, this combination of arithmetic and topological methods is very much in the spirit of (and in some respects identical to) what happens in the first and third author's earlier work with Venkatesh in \cite{evw}, concerning problems of Cohen-Lenstra type over function fields.  In \cite{evw} our approach to the topological problems was to prove homological stability theorems in the style of Quillen and Harer.  In a very recent preprint~\cite{bianchimiller}, Bianchi and Miller prove a homological stability result for Hurwitz spaces under conditions much less restrictive than those required in \cite{evw}; however, the range of stability is not sufficiently large to obtain arithmetic applications.

The new ingredient in the present work, which makes up the bulk of the paper's length, is to instead return from topology to algebra, viewing the cohomology of spaces relevant to the question as Ext algebras of quantum shuffle algebras and Nichols algebras.  We now explain this in a little more depth, separating the sketches into the portion concerning arithmetic geometry and the portion concerning topology and quantum algebra. 

\subsection{Sketch of proof:  arithmetic geometry}

When $K$ is the function field of a smooth curve $Y/\F_q$, a degree-$m$ extension $L/K$ is the function field of a connected curve $\Sigma$ over $\F_q$ which is presented as a degree $m$ branched cover $f: \Sigma \to Y$.  The discriminant of the cover $\Sigma \ra Y$ is $q^r$, where $r$ is the degree of ramification of $f$.  There is a slight and unavoidable mismatch here between the conventional function field and number field notation, so we pause to establish the notation we will use. The discriminant of a cover of curves takes into account ramification at all places of the base field, while the discriminant of a number field extension measures ramification only at the non-archimedean places. In the present paper, we will be studying covers of $K = \F_q(t)$, which is the function field of $Y = \P^1$.  By analogy with the number field case, we define the discriminant of an extension $L/K$ to be the discriminant of the extension of rings $\OO_L/\F_q[t]$, where $\OO_L$ denotes the integral closure of $\F_q[t]$ in $L$.  Equivalently, the discriminant is $q^r$, where $r$ is the degree of ramification of $L/K$ considered as a cover of the {\em affine} line $\A^1 = \P^1 - \infty$.

In this geometric setting, by contrast with the original arithmetic conjecture, the number $N_G(K;X)$ may be identified with the cardinality of the set of $\F_q$-points of a moduli scheme; namely, a {\em Hurwitz scheme} parametrizing branched covers of $\A^1$.  (See Section~\ref{hur_rack_section} for definitions and references.)  We let $\Hn_{G,\Delta = r}$ denote the Hurwitz moduli stack (defined over $\Z[\frac{1}{\#G}]$) of geometrically connected\footnote{The case of covers that are connected but not geometrically connected is an annoying issue that turns out to be harmless in this context; see the last paragraph of the proof of Theorem \ref{th:mainmalle}.} branched covers of $\A^1$ with ramification degree $r$ and Galois group $G$.  Theorem~\ref{malle_intro0} then amounts to an upper bound for the cardinality of $\Hn_{G,\Delta = r}(\F_q)$ for $r \leq \log_q X$. 

One way to guess how this cardinality behaves (see e.g. \cite{ev}) is to adopt the heuristic that every $\F_q$-rational component of $\Hn_{G,\Delta = r}(\F_q)$ of dimension $D$ has $q^D$ elements.  Under this heuristic assumption, one can show that 

\begin{equation}
\label{evasymp}
N_G(\F_q(t);X)  \asymp X^{a(G)} \log(X)^{b(G) - 1};
\end{equation}
that is, $N_{G}(\F_q(t); X)$ is asymptotically bounded above and below by constant multiples of $X^{a(G)} \log(X)^{b(G)-1}$, where $a(G)$ and $b(G)$ are explicit constants which agree with those in the strong Malle conjecture.   (We recall that, as always, $G$ is understood to be a subgroup of $S_m$, and the constants $a(G)$ and $b(G)$ depend on this permutation representation of $G$, not only on $G$ as abstract group.)

We can refine the picture by placing conditions on local monodromy.  Every ramified point $x$ in $Y$ gives rise to a local monodromy element in $S_m$ (and indeed in $G$), defined up to conjugacy.  By analogy with topology, one thinks of this element as the permutation of the $m$ sheets of the cover induced by a small loop around the branch point.  In general, there is a purely algebraic description:  the monodromy is the image of a generator of tame inertia in the decomposition group of $\Gal(L/K)$ at $x$.  (This description applies in the number field case as well, as long as we avoid primes of characteristic dividing $\#G$.)  If $c$ is a conjugacy-invariant set of nontrivial elements of $G$, we denote by $N^c_G(K,X)$ the number of degree-$m$ extensions of $K$ with discriminant at most $X$ and all local monodromy elements contained in $c$.  For example, when $G = S_m$ and $c$ is the class of transpositions, $N^c_G(\F_q(t),X)$ counts the number of degree-$m$ covers of $\A^1$ with simple branch points -- or, in algebraic terms, with squarefree discriminant divisor.  When $c = G \setminus \{1\}$, the local monodromy condition is no constraint\footnote{We insist that $c$ not contain 1; otherwise, this amounts to allowing the cover to not actually ramify at the branch locus.} at all; i.e. $N^c_G(K,X) = N_G(K,X)$.  Our actual goal, accomplished in Theorem~\ref{th:mainmalle}, will be to bound $N^c_G(\F_q(t),X)$ for an arbitrary conjugacy-invariant subset $c$.

To this end, we replace $\Hn_{G,\Delta = r}$, with the closely related Hurwitz scheme $\Hn_{G,n}^c$, which parametrizes branched covers with $n$ branch points, Galois group $G$, and local monodromy in $c \subseteq G$; it turns out to be enough for our purposes to control the number of $\F_q$-rational points on this more classical Hurwitz space.  The Grothendieck-Lefschetz fixed point theorem allows us to enumerate $\Hn_{G,n}^c(\F_q)$ as the alternating sum of traces of Frobenius on the compactly supported \'{e}tale cohomology of $\Hn_{G,n}^c$.   Indeed, the heuristic \eqref{evasymp} records the contribution of the $H^0$, and so an argument that the heuristic is correct (or approximately correct) comes down to finding some kind of bound on the traces of Frobenius on the higher cohomology.  Using Deligne's bounds on the eigenvalues of Frobenius, it suffices to bound the ranks of these cohomology groups.  A comparison theorem, along with Poincar\'{e} duality, identifies these ranks with those of the singular homology $H_j(\Hn_{G,n}^c(\C); \Q)$.  Theorem \ref{malle_intro0} is obtained by bounding these ranks exponentially in $j$, and polynomially in $n$.

To put this another way: the Weil bounds show that the heuristic applied in \cite{ev}, that a $D$-dimensional geometrically irreducible variety over $\F_q$ has $q^D$ points, is asymptotically correct as $q$ goes to $\infty$; it follows that that Malle's conjecture for $N_G(\F_q(t);q^r)$ is correct if $q$ is allowed to go to $\infty$ with $r$ fixed.  Proving something as $r \ra \infty$ with $q$ fixed, as we do here, is substantially harder.  The effect of letting $q$ go to $\infty$ is to make the contribution of all higher \'{e}tale cohomology groups of $\Hn_{G,n}^c$ negligible, so that it suffices to understand $H^0$; in the present work, by contrast, we have to give bounds for {\em all} the cohomology groups.

\subsection{Sketch of proof:  topology and quantum algebra}

The topological Hurwitz space $\Hur_{G,n}^c$ is a stack fibering over the configuration space $\Conf_n(\C)$ of $n$ points in the plane via the forgetful map which carries a branched covering of $\C$ to its branch locus.   Such a covering corresponds to a finite set endowed with an action of  $\pi_1(\Conf_n(\C))$, which is the Artin braid group $B_n$.  In fact, this $B_n$-set is very concretely describable; the finite set is $c^{\times n}$, and the action of the braid group is the {\em Hurwitz action} in which the standard generator $\sigma_i$ of the braid group acts by the rule
\beq
\sigma_i \cdot (g_1, \ldots, g_n) = (g_1, \ldots, g_{i-1}, g_i g_{i+1} g_i^{-1}, g_i, g_{i+2}, \ldots, g_n).
\eeq
This is not exactly the same thing as the complex manifold $\Hn_{G,n}^c(\C)$ considered in the previous section; rather, as we will see, $\Hn_{G,n}^c(\C)$ is an open and closed submanifold of a quotient of $\Hur_{G,n}^c$ by a natural action of $G$; but bounding the homology of $\Hur_{G,n}^c$ will suffice for our purposes.

The homology of this larger space is then precisely the homology of the braid group with coefficients in the Hurwitz representation:
$$H_*(\Hur_{G,n}^c; \Q) = H_*(B_n, \Q [c^{\times n}]).$$ 
The bulk of this paper is devoted to developing new techniques for these and related computations.

A \emph{braided vector space} over a field $k$ is a finite rank vector space $V$ equipped with an automorphism $\sigma: V \otimes V \to V \otimes V$ which satisfies the braid equation on $V^{\otimes 3}$.  This data yields an action of $B_n$ on $V^{\otimes n}$ in which $\sigma_i$ acts as $\sigma$ on the $i\nth$ and $i+1\st$ factors.  The relevant example for Malle's conjecture is $V = kc$; here $\sigma(g \otimes h) = h \otimes g^h$ for $g, h \in c$.  The resulting $B_n$ representation on $V^{\otimes n} = k [c^{\times n}]$ is precisely the Hurwitz representation described above.

For a braided vector space $V$, the \emph{quantum shuffle algebra} $\mA(V)$ is a braided, graded Hopf algebra whose underlying coalgebra is the cofree (tensor) coalgebra on $V$.  Its multiplication is given by a shuffle product which incorporates the braiding $\sigma$; see section \ref{qsa_section} for details.  Our first main result is that the homology in all degrees of all the Hurwitz spaces is packaged in the homological algebra of a quantum shuffle algebra.

\begin{thm} \label{ext_intro_thm}

There is an isomorphism $H_j(B_n; V^{\otimes n}) \cong \Ext^{n-j, n}_{\mA(V^*_\epsilon)}(k, k)$.

\end{thm}

Here $V^*_\epsilon$ is $V^*$ with braiding dual to that of $V$, and twisted by a sign.  In the bigrading on $\Ext$, the former index is the homological degree, and the latter the internal degree.  This result is Corollary \ref{main_cor}, below; this extends work of Fuks, Va{\u\i}n{\v{s}}te{\u\i}n, Markaryan, and Callegaro \cite{fuks, vainstein, markaryan, callegaro}.  The connection between cohomology of sheaves on configuration spaces and shuffle algebras has also recently been established by Kapranov-Schechtman in \cite{ks}.  Theorem \ref{ext_intro_thm} is proven using the Fox-Neuwirth/Fuks cellular stratification of $\Conf_n(\C)$, whose cellular cochain complex with coefficients in $V^{\otimes n}$ can be identified with the internal degree $n$ part of the bar complex for $\mA(V_\epsilon^*)$.

The quantum shuffle algebra is not an extremely well-studied object (with some notable exceptions such as \cite{rosso, dkkt, lebed}), and very little is known about its cohomology.  However, it contains an important subalgebra called the \emph{Nichols algebra}, $\mB(V_\epsilon^*)$, which plays a central role in the classification of pointed Hopf algebras, e.g., \cite{nichols, woronowicz, andruskiewitsch-schneider, as_annals}.  The Nichols algebra may be defined as the subalgebra of $\mA(V^*_\epsilon)$ generated by its degree 1 part (that is, $V^*_\epsilon$).  

These algebras also play a role in combinatorics and representation theory.  For instance, in the case that $V = kc$ is the braided vector space associated to the conjugacy class $c$ of transpositions in $S_d$, the quadratic cover of $\mB(V^*_\epsilon)$ is the Fomin-Kirillov algebra $\mathcal{E}_d$ whose structure constants encode a noncommutative extension of Schubert calculus on flag varieties \cite{fomin-kirillov}.  Alternatively, if the braiding on $V$ is given by exponentiating a Cartan matrix, then $\mB(V)$ is the Borel part of the enveloping algebra of the associated quantum group \cite{rosso}.

The Nichols algebra turns out to be much easier to work with than the full quantum shuffle algebra; and it is ``close enough" to the quantum shuffle algebra that it suffices for our purposes to control extensions over $\mB(V^*_\epsilon)$ rather than over $\mA(V^*_\epsilon).$  More precisely, we show in Theorem \ref{alg_bound_thm} that if we can control the growth of $\Ext^{n-j, n}_{\mB(V^*_\epsilon)}(k, k)$ in $j$ and $n$, we get bounds of similar quality for $\Ext^{n-j, n}_{\mA(V^*_\epsilon)}$.

Thus, we have reduced the problem of obtaining the bounds on the growth of $\# \Hn_{G,n}^c(\F_q)$ necessary to yield Theorem \ref{malle_intro0} to the purely algebraic problem of providing such a bound on the growth of the cohomology of $\mB(V^*_\epsilon)$.  This is achieved in sections \ref{nichols_coh_section} and \ref{bound_section}, which are the technical heart of the paper: we show that the cohomology of $\mB(V^*_\epsilon)$ has graded pieces whose dimensions are bounded above by those of a tensor algebra over the ring $R = \sum H_0(B_n, V^{\otimes n})$.  The generators of this tensor algebra are computed via a Koszul complex for $\mB(V_{\epsilon})$.  Generally, we can inductively compute this Koszul complex via a spectral sequence arising from the subgroup lattice of $G$.  We show that the Koszul complex for each subgroup vanishes above a certain degree using a variant on the Conway-Parker-Fried-V\"{o}lklein theorem \cite{fried-volklein}.  This is sufficient to control the growth of $\Ext_{\mB(V^*_\epsilon)}(k, k)$.

\subsection{Further questions}

Malle's conjecture is believed to hold for global fields in general, or at least those whose characteristic is prime to $|G|$.  In particular, it should apply to $G$-extensions of higher-genus function fields over $\F_q$, not just $\F_q(t)$, the case considered here.  We believe the methods of the present paper will be useful for treating the more general situation, but some new ideas would be required.  (Indeed, we would say the same about the results of \cite{evw}.)

Another interesting avenue for future research concerns the cohomology of the Nichols algebra.  In the present paper, it appears as a kind of expedient, helping us compute the thing we are actually after, the cohomology of the quantum shuffle algebra (or equivalently, as we show, the cohomology of Hurwitz spaces).

However, as we learned from the work of Kapranov and Schechtman~\cite{ks}, the cohomology of the Nichols algebra also has a direct geometric interpretation, which is likely of interest in its own right.  The Hurwitz space is a finite cover of $\Conf_n \A^1$, which gives rise to a permutation representation of the braid group, which we think of as a local system $V^{\tensor n}$ on $\Conf_n \A^1$.  Our goal in the present paper is to study $H^*(\Conf_n \A^1, V^{\tensor n})$, which is what eventually allows us to bound $\F_q$-rational points on Hurwitz spaces.  These are the cohomology groups packaged in $\Ext_{\mA(V_\eps)}(k,k)$.  

There is a natural open inclusion $j: \Conf_n \A^1 \ra \Sym_n \A^1 \cong \A^n$, given by  the inclusion of the space of monic squarefree degree $n$ polynomials into the space of all monic degree $n$ polynomials.  It turns out that what is packaged by $\Ext_{\mB(V^*_\epsilon)}(k, k)$, the cohomology of the Nichols algebra, are the cohomology groups
\beq
H^*(\Sym_n \A^1, j_{!*} V^{\tensor n})
\eeq
where $j_{!*} V^{\tensor n}$ denotes the {\em middle extension}, a perverse sheaf on $\Sym_n \A^1$.  This family of perverse sheaves and their cohomology groups (in both the topological and the \'{e}tale context) seems to us to be a natural object of study, especially given that there are cases (for instance, the case where $G=S_3$ and $c$ is the class of transpositions \cite{stefan-vay}) where the cohomology of the Nichols algebra has been computed in full, and where the groups can be shown to satisfy much better upper bounds on their growth than we expect for those of the full quantum shuffle algebra.

In the \'{e}tale context, the alternating trace of Frobenius on these cohomology groups should be expressible as a sum of some function over monic degree-$n$ polynomials
\beq
\sum_{P \in \Sym_n \A^1(\F_q)}  F(P)
\eeq
where, for {\em squarefree} polynomials $P$, the local term $F(P)$ is the number of $G$-extensions with monodromy type $c$ and discriminant $P$.  It seems very worthwhile to understand what value $F$ takes on non-squarefree polynomials, and even, though here we get speculative, whether there is an analogue of $F$ over $\Spec \Z$.  In the case of $G=S_3$ and $c$ transpositions, for instance, this would be a function $\Phi$ on $\Z$ whose value at a squarefree integer $m$ is the number of cubic extensions with discriminant $m$, but whose value at a non-squarefree $m$ was something different, arrived at by analogy with the function $F$ arising from the middle extension.  Would the Dirichlet series attached to $\Phi$ have good analytic properties?  Would $\Phi$ count things we are interested in counting?

\subsection{Acknowledgements}

We are indebted to Dev Sinha for encouraging us to approach the study of Hurwitz spaces using the Fox-Neuwirth stratification of configuration spaces.  We thank Mikhail Kapranov and Vadim Schechtman for sharing their preprint \cite{ks} with us, and Joel Specter for a lively discussion of related material, as well as alerting us to the work of Markaryan \cite{markaryan}.  Similar thanks go to S{\o}ren Galatius for pointing us to Callegaro's work \cite{callegaro}, and to Marshall Smith for allowing us to reproduce his related, unpublished computations here.  

Many thanks are due to Marcelo Aguiar, Nicol\'{a}s Andruskiewitsch, Victor Ostrik and Sarah Witherspoon for a number of very helpful conversations about Nichols algebras and their cohomology, to Jiuya Wang for conversations about Malle's conjecture, and to Will Sawin for conversations about perverse sheaves in arithmetic statistics.  

We owe special thanks to Oscar Randal-Williams for finding a flaw in our original argument for bounding the growth of the cohomology of Nichols algebras; we learned a lot from repairing the proof. 

The first author has been supported on this project by NSF Grant DMS-2001200, NSF Grant DMS-1700885, a Guggenheim Fellowship, and a Simons Fellowship. The second author was partly supported by ARC Grant DE130100650. The third author was supported by NSF Grant DMS-1406162.

\section{Algebra from braid representations}

Let $k$ be a field; unless otherwise indicated, all tensor products will be over $k$.  Many of the definitions below carry through for arbitrary commutative rings $k$.  However, at various points, we will need to split surjective maps of $k$-modules by one-sided inverses; the blanket assumption that $k$ is a field certainly ensures this.

Besides the first section below, the bulk of this material is an introduction to Yetter-Drinfeld modules and braided Hopf algebras.  None of these results are original, although we include some proofs for the unfamiliar reader, and refer them to sources such as \cite{nichols, andruskiewitsch-schneider} as well as sections 2 and 3 of \cite{bazlov} for more details.

\subsection{Monoidal braid representations and braided vector spaces}

Recall that the $n^{\rm th}$ \emph{braid group} $B_n$ is presented as 
\[ B_n := \langle \sigma_1, \dots, \sigma_{n-1} \; | \; \sigma_i \sigma_{i+1} \sigma_i =  \sigma_{i+1} \sigma_i \sigma_{i+1}, \; \sigma_i \sigma_j =\sigma_j \sigma_i, \mbox{if $|i-j|>1$} \rangle. \]
Consider the groupoid $B = \sqcup_{n \geq 0} [*/B_n]$ of all braid groups.  We may equip $B$ with the structure of a (braided) monoidal category; the tensor product $B \times B \to B$ is induced by the homomorphism $B_n \times B_m \to B_{n+m}$ that places braids side-by-side.

\begin{defn} A \emph{monoidal braid representation} is a (strictly) monoidal functor $\Phi$ from $B$ to the category of finite rank vector spaces over $k$.  \end{defn}

The family of monoidal braid representations forms a category whose morphisms are monoidal natural transformations.  Note that $\Phi$ carries the $n^{\rm th}$ object of $B$ to a vector space $\Phi(n)$.  Since $\Phi$ is monoidal, there is an isomorphism of $k$-modules $\Phi(n) \cong \Phi(1)^{\otimes n}$.

\begin{defn} A \emph{braided vector space} $V$ is a vector space over $k$ of finite rank, equipped with a \emph{braiding} $\sigma: V \otimes V \to V \otimes V$ which is invertible and satisfies the braid equation
$$(\sigma \otimes \id) \circ (\id \otimes \sigma) \circ (\sigma \otimes \id) =  (\id \otimes \sigma) \circ (\sigma \otimes \id) \circ (\id \otimes \sigma)$$
as a self-map of $V^{\otimes 3}$.

\end{defn}

Braided vector spaces form a category where morphisms between $(V, \sigma)$ and $(W, \tau)$ are $k$-linear maps $f: V \to W$ with $(f \otimes f) \circ \sigma = \tau \circ (f \otimes f)$.  Note that $B_n$ acts on $V^{\otimes n}$ via $\sigma_i \mapsto \id^{\otimes i-1} \otimes \sigma \otimes \id^{\otimes n-i-1}$.  The following is straightforward:

\begin{prop} \label{monoidal_equiv_prop} There is a pair of inverse equivalences between the categories 
$$\mbox{monoidal braid representations} \leftrightarrow \mbox{braided vector spaces}$$
carrying $\Phi$ to $\Phi(1)$ and $V$ to the braid representation on $V^{\otimes n}$ described above. \end{prop}

\subsection{Yetter-Drinfeld modules and braided Hopf algebras} \label{YD_section}

One way of obtaining a braided vector space is as a Yetter-Drinfeld module:

\begin{defn} For a group\footnote{One may define this category for any Hopf algebra; group algebras will suffice for our purposes.} $G$, a (right, right) \emph{Yetter-Drinfeld module} $M$ is a $G$-graded right $kG$-module with the property that $M_{b} \cdot a \leq M_{b^a}$ for $a, b \in G$ and $j \in \Z$; here $b^a = a^{-1} b a \in G$.  A map of Yetter-Drinfeld modules is a $kG$-module map preserving this grading.  

A \emph{graded Yetter-Drinfeld module} is a $\Z$-graded collection of these; maps between these are collections of Yetter-Drinfeld module maps.\end{defn}

Depending upon the context, for $g \in G$ and $m \in M$, we may either write $m \cdot g$ or $m^g$ for the action of $g$ on $m$.  Further, if $n \in M_{g}$, we will write $m^n := m^g$ for the action on $m$ of the index $g \in G$ labelling $n$.

The category $\YD_G^G$ of Yetter-Drinfeld modules forms a braided monoidal category, where the braiding $\sigma:M \otimes N \to N \otimes M$ is given by the formula
$$\sigma(m \otimes n) := n \otimes m^n = n \otimes m^g$$
where $n \in N_{g}$.  Taking $N=M$ gives a braided vector space $(M, \sigma)$.

If $M$ is a Yetter-Drinfeld module, the ($G$-graded) dual vector space $M^* = \Hom_k(M, k)$ becomes a Yetter-Drinfeld module via the obvious decomposition
$$M^* \cong \bigoplus_{g \in G} (M_g)^*$$
and the right action by inverses of duals: for $\phi \in M^*$ and $g \in G$,
$$\langle \phi^g, x \rangle = \langle \phi, x^{g^{-1}} \rangle.$$
Here $\langle \phi, x \rangle := \phi(x)$ is our somewhat awkward notation for evaluation of a homomorphism on an element of $M$.

If $A$ and $B$ are Yetter-Drinfeld algebras (i.e., monoid objects in $\YD_G^G$), then their tensor product $A \otimes B$ becomes a Yetter-Drinfeld algebra using the braiding, by defining a multiplication $\star$:
$$(a \otimes b) \star (a' \otimes b') := (a \star a') \otimes (b^{a'} \star b') = (a \star a') \otimes (b^g \star b') $$ 
where $a' \in A_{g}$.

A \emph{braided, graded Hopf algebra} is a graded Hopf algebra object $A$ in a braided monoidal category; we will mostly focus on those occurring in the category $\YD^G_G$.  That is, $A$ is simultaneously a unital associative $k$-algebra (with multiplication $\star$ and unit $\eta$), and a counital coassociative coalgebra (with coproduct $\Delta$ and counit $\epsilon$), where all structure maps are maps of Yetter-Drinfeld modules.  Further, $\Delta: A \to A \otimes A$ is an algebra map (using the twisted multiplication described above).  Lastly, the antipode $\chi: A \to A$ is a Yetter-Drinfeld map with
$$\eta \epsilon(a) = \sum_i a_i' \star \chi(a_i'') = \sum_i \chi(a_i') \star a_i''$$
for $\Delta(a) = \sum a_i' \otimes a_i''$.

We will use the notation 
$$P(A) = \{a \in A \; | \; \Delta(a) = a \otimes 1 + 1 \otimes a\}$$
for the \emph{primitives} of $A$, and 
$$Q(A) =I/I^2, \; \mbox{ where } \; I = \ker(\epsilon)$$
for its \emph{indecomposables}.  A graded algebra $A$ is \emph{connected}\footnote{More carefully, $A$ is \emph{connected as a graded algebra}.  A coalgebra is said to be connected if its coradical is one dimensional.  These are closely related definitions for graded Hopf algebras (see \cite{bgz}); we will always use the notion for algebras in this paper.} if
$$A_0 = k \{ 1 \} \; \mbox{ and } \; A_n = 0 \mbox{ if } n<0.$$  

\subsection{Braided partial derivatives} \label{derivative_section}

\begin{defn} \label{deriv_defn} Let $A$ be a graded algebra and $A^*$ its graded dual.  For $w \in A$, define \emph{braided partial derivative operators} $\partial^L_w: A^* \to A^*$ and $\partial^R_w: A^* \to A^*$
by 
$$\langle \partial^L_w \phi, x \rangle = \langle \phi, xw \rangle \; \mbox{ and } \; \langle \partial^R_w \phi, x \rangle = \langle \phi, wx \rangle$$
\end{defn}

We note that these are degree-shifting operators, decreasing the degree of $\phi$ by the degree of $w$.

\begin{prop}

The operators $\partial^L_w$ define a left action of the algebra $A$ on $A^*$.  Further, if $A$ is a braided Hopf algebra in $\YD_G^G$, the action of primitives is through braided derivations: when $v \in P(A)$,
$$\partial^L_v(\phi \psi) = \partial^L_v(\phi) \psi^{v^{-1}} + \phi \partial^L_v(\psi).$$
Finally, this operation interacts with the action of $G$ in the following way:  
$$(\partial^L_v \phi)^g = \partial^L_{v^g} \phi^g.$$

\end{prop}

\begin{proof}

It is easy to see that $\partial^L_{vw} = \partial^L_v \partial^L_w$, so indeed this defines an action of $A$ on $A^*$.  To see that the action of the primitives is through derivations, we use the notation $\Delta(x) = \sum_i x_i' \otimes x_i''$ for the diagonal:
\begin{eqnarray*} 
\langle\partial^L_v(\phi \psi), x\rangle & = & \langle \phi \psi, xv \rangle \\
 & = & \langle \phi \otimes \psi, \Delta(xv)\rangle \\
 & = & \langle \phi \otimes \psi, \sum_i (x_i' \otimes x_i'') \cdot(v \otimes 1 + 1 \otimes v)\rangle \\
 & = & \langle \phi \otimes \psi, \sum x_i'v \otimes (x_i'')^v + \sum x_i' \otimes x_i''v \rangle \\
 & = & \sum \langle \phi, x_i' v \rangle \langle \psi^{v^{-1}}, x_i'' \rangle + \sum \langle \phi, x_i' \rangle \langle \psi, x_i'' v \rangle \\
 & = & \langle \partial^L_v(\phi) \psi^{v^{-1}} + \phi \partial^L_v(\psi), x \rangle.
\end{eqnarray*}

The equivariance is a similar computation:
\begin{eqnarray*} \langle (\partial^L_v \phi)^g, x \rangle & = &  \langle \partial^L_v \phi, x^{g^{-1}} \rangle  \\
 & = & \langle \phi, x^{g^{-1}}v \rangle \\
 & = & \langle \phi^{g}, x v^g \rangle \\
 & = & \langle \partial^L_{v^g} \phi^g, x \rangle.
\end{eqnarray*}
The third equality uses the fact that the $G$ action on $A$ is through algebra homomorphisms.

\end{proof}

There is a similar result for the right derivatives $\partial_w^R$: this gives a right action of $A$ on $A^*$.  While the equivariance of the construction is as for $\partial_w^L$, the Leibniz rule takes the form:
$$\partial^R_v(\phi \psi) = \partial^R_v(\phi) \psi + \phi \partial^R_{v^{\phi}}(\psi).$$
Finally, the following is a simple application of the definition of these operators:

\begin{lem} \label{diagonal_lem}

Let $A$ be a locally finite graded algebra, so that $A^*$ is a graded coalgebra.  Let $\{e_i\}$ be a basis for $A$, with dual basis $\{e_i^*\}$ for $A^*$.  Then if $\phi \in A^*$, 
$$\Delta(\phi) = \sum e_i^* \otimes \partial_{e_i}^R(\phi) = \sum \partial_{e_i}^L(\phi) \otimes e_i^*.$$

\end{lem}

\subsection{Braided vector spaces of rack type} \label{rack_section}

We recall that a \emph{rack} is a set $R$ with a binary operation $(a, b) \mapsto b^a$ with the properties that
\begin{enumerate}
\item $(c^a)^{b^a} = (c^b)^a$, and
\item For each $a$, $b \in R$, there exists a unique $c \in R$ with $c^a = b$.  We will notate $c$ as $c = {}^{a}b$.
\end{enumerate}
Further, $R$ is a \emph{quandle} if $a^a = a$ for each $a \in R$.  

\begin{exmp}

The main example of a quandle is obtained by taking $R$ to be a group, and the operation to be conjugation: $b^a = a^{-1} b a$.  The first axiom is straightforward to verify; the second is gotten from taking 
$$c = {}^{a}b = b^{(a^{-1})} = a b a^{-1}.$$
Consequently $R$ may also be taken to be a union of conjugacy classes.  Alternatively, one may define the \emph{structure group} $G(R)$ of a rack $R$: it is generated by $R$, subject to the relations $b \cdot a = a \cdot (b^a)$.

\end{exmp}

\begin{defn}

A \emph{rack 2-cocycle} $x$ with values in $k^{\times}$ is a function $x: R \times R \to k^{\times}$ with the property that
\beqn \label{cocycle_eqn} x_{ab} x_{a^b c} = x_{ac} x_{a^c b^c}.\eeqn

\end{defn}

It is indeed the case that these objects form cocycles for a cohomology theory for racks; see \cite{etingof-grana, szymik}.  We may associate a braided vector space to this data.  The following is a special case of the construction of \cite{andruskiewitsch-grana}, where this idea is far more thoroughly explored; see also \cite{heckenberger-lochmann-vendramin} for a number of very interesting examples.

\begin{defn}

If $R$ is a rack and $x$ is a 2-cocycle for $R$ with values in $k^{\times}$, define a braided vector space $V(R, x):= kR$ to be the vector space spanned by $R$, with braiding
$$\sigma(a \otimes b) = x_{ab} (b \otimes a^b)$$
for $a, b \in R$.

\end{defn}

The map $\sigma$ satisfies the braid equation: this follows from the axioms of a rack and the cocycle property (\ref{cocycle_eqn}).
An obvious choice of cocycle is to take $x_{ab} = 1$ for all $a, b$.  A slightly less obvious cocycle is given by $x_{ab} = -1$; the resulting braided vector space $V_\epsilon := V(R, -1)$ is one of the main subjects of this paper.  

We note that there is a right action of the structure group $G(R)$ on $V$ by conjugation of basis elements; for $a \in R$ and $b \in G(R)$,
$$a \cdot b = a^b.$$
Further, we may make $V$ into a Yetter-Drinfeld module for $G(R)$ by equipping $V$ with the $G(R)$-grading where $a \in V_{a}$ for $a \in R \subseteq G(R)$.

\subsection{Quantum shuffle algebras} \label{qsa_section}

Recall that a \emph{shuffle} $\tau: \{1, \dots, m\} \sqcup \{1, \dots, n\} \to \{1, \dots, m+n\}$ is a bijection which does not change the order of either $\{1, \dots, m\}$ or $\{1, \dots, n\}$.  For any such $\tau$, there is a natural choice of a lift $\tautilde \in B_{m+n}$ which is given by the braid that shuffles the endpoints according to $\tau$ by moving the right $n$ strands in front of the left $m$ strands (See Figure \ref{shuffle to braid}).
\begin{figure}
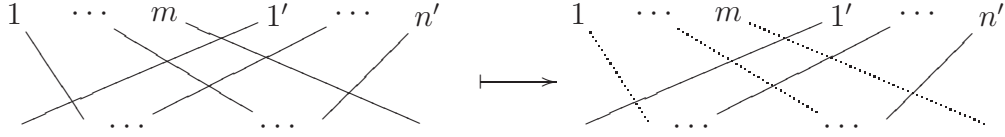

\[ \vcenter{	
	\xy
		{\ar@{-} (0,0)*+{}; (35,15)*+{1'} };
		{\ar@{-} (15,0)*+{\cdots}; (45,15)*+{\cdots} }; 
		{\ar@{-} (40,0)*+{}; (55,15)*+{n'} }; 
		{\ar@{-} (10,0)*+{}; (0,15)*+{1} }; 
		{\ar@{-} (35,0)*+{\cdots}; (10,15)*+{\cdots} }; 
		{\ar@{-} (55,0)*+{}; (20,15)*+{m} }; 
		{\ar@{|->} (62,6)*{}; (72,6)*{}};
		{\ar@{-} (75,0)*+{}; (110,15)*+{1'} };
		{\ar@{-} (90,0)*+{\cdots}; (120,15)*+{\cdots} }; 
		{\ar@{-} (115,0)*+{}; (130,15)*+{n'} }; 
		{\ar@{..} (85,0)*+{}; (75,15)*+{1} }; 
		{\ar@{..} (110,0)*+{\cdots}; (85,15)*+{\cdots} }; 
		{\ar@{..} (130,0)*+{}; (95,15)*+{m} }; 
	\endxy
	}
\]
	\caption{Lifting an $(m,n)$-shuffle to a braid.\label{shuffle to braid}}
\end{figure}

Let $(V, \sigma)$ be a braided vector space.  We will write elements of $V^{\otimes n}$ using bar complex notation; i.e., $[a_1| \dots | a_n]$.  

\begin{defn}

The \emph{quantum shuffle algebra} $\mA(V)$ is a braided, graded bialgebra: its underlying coalgebra is the tensor coalgebra $T^{co}(V)$ (i.e., with the deconcatenation coproduct $\Delta$) and has multiplication defined by the quantum shuffle product:
$$[a_1| \dots | a_m] \star [b_1 | \dots | b_n] = \sum_{\tau} \tautilde [a_1| \dots | a_m | b_1 | \dots | b_n]$$
where the sum is over shuffles $\tau: \{1, \dots, m\} \sqcup \{1, \dots, n\} \to \{1, \dots, m+n\}$.

\end{defn}

The quantum shuffle multiplication on $\mA(V)$ is associative, though this is not obvious; see, e.g., \cite{rosso, lebed, kapranov-schiffmann-vasserot}.  This makes $\mA(V)$ into a braided, graded, connected bialgebra.  It is not generally commutative and is only cocommutative when $\rk_k(V) = 1$.  In fact, $\mA(V)$ is a braided Hopf algebra, as it admits an antipode $\chi: \mA(V) \to \mA(V)$.  Following \cite{milnor_steenrod}, this is given by $\chi(1) = 1$, and in positive degrees as the unique solution to the equation
$$\sum_{j=0}^{n} [a_1| \dots | a_j] \star \chi([a_{j+1}| \dots | a_n]) = 0.$$
As noted in the definition of a braided Hopf algebra, the coproduct is only a map of algebras when $\mA(V) \otimes \mA(V)$ is given the product structure twisted by the braiding.  In general, we may describe that as follows: for each $m$ and $n$, we define a map 
$$\sigma_{m, n}: V^{\otimes m} \otimes V^{\otimes n} \to V^{\otimes n} \otimes V^{\otimes m}$$
via the action of $B_{m+n}$ on $V^{\otimes m+n}$, using the braid that moves the first $m$ strands behind the last $n$ strands (see Figure \ref{braiding}).  As $m$ and $n$ vary, the collection of these maps define the braiding on $\mA(V) \otimes \mA(V)$ which twists the multiplication.
\begin{figure}
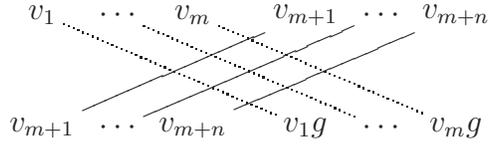

\[ \vcenter{	
	\xy
		{\ar@{-} (0,0)*+{v_{m+1}}; (35,15)*+{v_{m+1}} };
		{\ar@{-} (10,0)*+{\cdots}; (45,15)*+{\cdots} }; 
		{\ar@{-} (20,0)*+{v_{m+n}}; (55,15)*+{v_{m+n}} }; 
		{\ar@{..} (35,0)*+{v_1g}; (0,15)*+{v_1} }; 
		{\ar@{..} (45,0)*+{\cdots}; (10,15)*+{\cdots} }; 
		{\ar@{..} (55,0)*+{v_mg}; (20,15)*+{v_m} }; 
	\endxy
	}
\]
\caption{The braiding on $\mA(V) \otimes \mA(V)$ in the case when $(V,\sigma)$ is a Yetter-Drinfeld module for a group $G$. \label{braiding}} 
\end{figure}
In the case that $(V, \sigma)$ is a Yetter-Drinfeld module for a group $G$, a formula for $\sigma_{m, n}$ is straightforward:
\beqn \label{twist_eqn} \sigma_{m, n}([a_1| \dots|a_m] \otimes [b_1|\dots| b_n]) = [b_1|\dots| b_n] \otimes [a_1^g| \dots| a_m^g] \eeqn
where each $b_i \in V_{g_i}$, and $g = g_1 \dots g_n$.  We note that $\mA(V)$ is a Yetter-Drinfeld module where the grading on tensor products is by multiplication; e.g., in the above, $[b_1 | \dots |b_n] \in V_g$.

\subsection{The Nichols algebra} \label{nichols_section}

A central object of study in this paper is the Nichols algebra, to be defined below.  In Theorem \ref{ext_intro_thm}, we relate the homology of braid groups to the homology of the quantum shuffle algebras (defined in the previous section).  Except in special cases, their cohomology is quite inscrutable.  More approachable (though still somewhat mysterious) is the cohomology of a subalgebra, the Nichols algebra.  Further, we will see in Theorem \ref{alg_bound_thm} that in a certain sense, these subalgebras contribute a sort of ``dominant term" in the cohomology of the quantum shuffle algebra. 

\begin{defn}  The \emph{Nichols algebra} or \emph{quantum symmetric algebra} $\mB(V) \leq \mA(V)$ is the $k$-subalgebra generated by $V$ under the quantum shuffle product. \end{defn}

Since $V \leq \mA(V)$ is primitive, the Nichols algebra is also the sub-braided Hopf algebra of $\mA(V)$ generated by $V$.  We may also present it as a quotient of the tensor algebra $T(V)$ on $V$.  Here $T(V)$ is made into a braided Hopf algebra by declaring $V$ to be primitive.

\begin{prop} \label{nichols_characterization_prop}

The Nichols algebra may alternatively be characterized as:

\begin{enumerate}

\item The image of $T(V)$ in $\mA(V)$ under the unique Hopf algebra homomorphism which is the identity on $V$.

\item The unique braided Hopf algebra $\mB(V)$ where the natural maps
$$V \to P(\mB(V)) \to Q(\mB(V))$$
are isomorphisms.

\item The braided Hopf algebra which in degree $n$ is the quotient of $T(V)_n = V^{\otimes n}$ by the kernel of the \emph{quantum symmetrizer} $\mS_n: V^{\otimes n} \to V^{\otimes n}$: 
$$\mS_n [v_1 | \dots | v_n] = \sum_{\tau \in S_n} \tautilde [v_1 | \dots | v_n];$$
here $\tautilde \in B_n$ is the Matsumoto lift\footnote{This is the (set-theoretic) section $S_n \to B_n$ to the natural projection which expresses a permutation as a minimal length word in the transpositions $(i, i+1)$, and replaces each of these with the braid $\sigma_i$.} of $\tau \in S_n$.

\end{enumerate}

\end{prop}

\begin{proof} 

There is a map of algebras
$$\mS: T(V) \to \mA(V)$$
induced by the identity of $V$; here we use the fact that $T(V)$ is the free algebra generated by $V$.  This is also a map of coalgebras, since $\mA(V) = T^{co}(V)$ is the cofree graded coalgebra (primitively) cogenerated by $V$; since both Hopf algebras are connected, the map also preserves the antipode, so this is in fact a map of Hopf algebras.  Since $\mB(V)$ is, by definition, the image of the map, the first claim follows.

The second claim follows from the first: since the indecomposables of $T(V)$ are clearly $V$, so too are the indecomposables of the quotient algebra $\mB(V)$.  Similarly, the primitives of $\mA(V)$ are definitionally $V$, so the same holds for the sub-coalgebra $\mB(V)$.  Further, if $A$ is any other braided Hopf algebra satisfying $V \cong P(A) \cong Q(A)$, then there are natural Hopf maps
$$T(V) \to A \to \mA(V).$$
The first is surjective, since $A$ is generated by $V$; dually, the second is injective.  Consequently the second carries $A$ isomorphically onto $\mB(V) \leq \mA(V)$.

From the first claim, we note that we can present $\mB(V)$ as the quotient of $T(V)$ by the kernel of the map $\mS$.  It was shown in \cite{schauenburg} that the degree $n$ component of $\mS$ is in fact given by the quantum symmetrizer.  This recovers the definition of \cite{nichols, woronowicz} that $\mB(V)_n$ is presented as the quotient of $T(V)_n = V^{\otimes n}$ by the kernel of $\mS_n$.

\end{proof}
 
There is a pairing between Nichols algebras for dual braided vector spaces that we will need in section \ref{koszul_nichols_section} in order to describe the differential on a Koszul complex for $\mB(V)$.  The dual $V^*$ of $V$ may be made into a braided vector space whose braiding is the dual of that of $V$.  

\begin{prop}

The pairing between $V$ and $V^*$ induces a nondegenerate pairing of Hopf algebras
$$\langle \cdot , \cdot \rangle: \mB(V) \otimes \mB(V^*) \to k.$$
which identifies $\mB(V^*)$ as the dual Hopf algebra of $\mB(V)$.

\end{prop}

\begin{proof}

To begin, the pairing between $V$ and $V^*$ extends to a Hopf algebra pairing between $T(V)$ and $T(V^*)$; that is, it is the unique pairing for which the product in $T(V)$ is adjoint to the coproduct in $T(V^*)$.  Explicitly, if $(-,-)$ is the standard termwise pairing between $V^{\otimes n}$ and $(V^*)^{\otimes n}$, then the pairing of tensor algebras $\langle -, -\rangle$ is zero on tensors of different length, and in degree $n$, 
$$\langle [v_1| \cdots | v_n], [\phi_1| \cdots | \phi_n] \rangle = \sum_{\tau \in S_n}  (\tautilde [v_1| \cdots | v_n], [\phi_1| \cdots | \phi_n]).$$

This may be degenerate; however, it is apparent from this formula that the kernel of the quantum symmetrizer is in fact the kernel of the pairing.  Thus it descends to a nondegenerate pairing
$$\langle \cdot , \cdot \rangle: \mB(V) \otimes \mB(V^*) \to k.$$

\end{proof}

\section{Topology of configuration spaces}

The purpose of this section is to prove Theorem \ref{ext_intro_thm}, identifying 
$$H_j(B_n, V^{\otimes n}) \cong \Ext_{\mA(V^*_\epsilon)}^{n-j, n}(k, k)$$
and explore this result through several examples.  Throughout, $V$ may be taken to be an arbitrary braided vector space.

\subsection{Fox-Neuwirth/Fuks cells}

We recall from \cite{fox-neuwirth} and \cite{fuks} (see also \cite{vassiliev, giusti-sinha}) a stratification of $\Conf_n(\C)$ by Euclidean spaces.  This does not give a cell-decomposition of $\Conf_n(\C)$, as not all of the boundary of these Euclidean spaces is attached to lower-dimensional cells.  Instead, this gives the 1-point compactification of $\Conf_n(\C)$ the structure of a CW complex -- the unattached boundaries are now glued to the point at infinity.

An \emph{ordered} partition (or \emph{composition}) $\lambda = (\lambda_1, \dots, \lambda_k)$ of $n$ has $\sum \lambda_i = n$.  We will write $k=:\ell(\lambda)$ for the number of \emph{parts} (or \emph{length}) of $\lambda$.  The $n^{\rm th}$ symmetric product $\Sym_n(\R)$ of the real line has a stratification by these partitions:
$$\Sym_n(\R) = \coprod_{\lambda \vdash n} \Sym_\lambda(\R)$$
where elements of $\Sym_{\lambda}(\R)$ consist of an unordered subset of $\ell(\lambda)$ distinct points $x_1, \dots x_{\ell(\lambda)}$, the $i^{\rm th}$ of which has multiplicity $\lambda_i$.  Further, since $\R$ is ordered, we insist that $x_1 < \dots < x_{\ell(\lambda)}$.  This space is evidently homeomorphic to $\Conf_{\ell(\lambda)}(\R)$, which is in turn homeomorphic to $\R^{\ell(\lambda)}$.

Define a map $\pi: \Conf_n(\C) \to \Sym_n(\R)$ by taking real parts:
$$\pi(z_1, \dots, z_n) = (\Re(z_1), \dots, \Re(z_n))$$
and let $\Conf_\lambda(\C)$ denote the preimage of $\Sym_\lambda(\R)$ under $\pi$ (see Figure \ref{conf picture}).  This subspace is homeomorphic to
$$\Sym_{\lambda}(\R) \times \prod_{i=1}^{\ell(\lambda)} \Conf_{\lambda_i}(\R),$$
where the configuration factors record the imaginary part of the configuration.  We again employ the fact that $\Conf_k(\R) \cong \R^k$ and conclude that $\Conf_\lambda(\C) \cong \R^{n+\ell(\lambda)}$. 
\begin{figure}
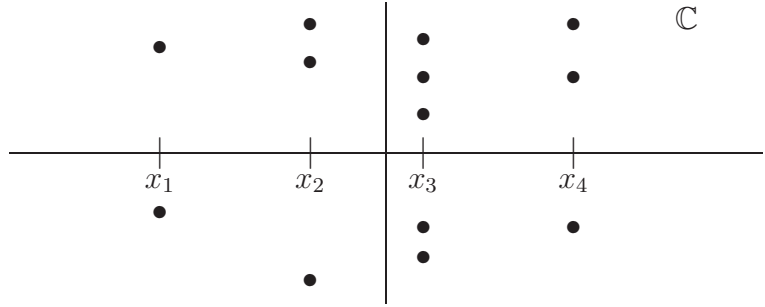

\[ \vcenter{	
	\xy
		(-30,14)*{\bullet}; (-30,-8)*{\bullet}; (-30,-4)*{x_1}; (-30,0)*{|};
		(-10,12)*{\bullet}; (-10,-17)*{\bullet}; (-10,17)*{\bullet}; (-10,-4)*{x_2}; (-10,0)*+{|};
		(5,10)*{\bullet}; (5,-10)*{\bullet}; (5,15)*{\bullet}; (5,-14)*{\bullet}; (5,5)*{\bullet}; (5,-4)*{x_3}; (5,0)*+{|};
		(25,10)*{\bullet}; (25,-10)*{\bullet}; (25,17)*{\bullet}; (25,-4)*{x_4}; (25,0)*+{|};
		{\ar@{-} (-50,0)*{}; (50,0)*{}};
		{\ar@{-} (0,-20)*{}; (0,20)*{}};
		(40,18)*{ \mathbb{C}};
	\endxy
	}
\]
	\caption{A configuration in $\Conf_{(2,3,5,3)}(\mathbb{C}) \subset \Conf_{13}(\mathbb{C})$. The configuration is mapped by $\pi$ to $(x_1,x_1, x_2, x_2,x_2, x_3, x_3,x_3,x_3,x_3, x_4,x_4,x_4) \in \Sym_{(2,3,5,3)}(\mathbb{R}) \subset \Sym_{13}(\mathbb{R})$. \label{conf picture}}
\end{figure}

These spaces then form a cellular decomposition of $\Conf_n(\C) \cup \{\infty\}$, as was established by Fox-Neuwirth and Fuks \cite{fox-neuwirth, fuks} (the latter in language closer to what we use here).  Loosely speaking, the boundaries of the cell described above occur in two ways.  First, points in a configuration may approach each other or infinity along vertical lines, (in which case their boundary is the point at infinity, as there is no point in $\Conf_n(\C)$ for them to approach).  Secondly, the $i^{\rm th}$ and $i+1^{\rm st}$ vertical columns of configurations may approach each other horizontally, in which case the associated component of the boundary is given in terms of a cell $\Conf_{\rho}(\C)$, where $\rho$ is obtained from $\lambda$ by summing $\lambda_i$ and $\lambda_{i+1}$.  In summary:

\begin{prop}

The space $\Conf_n(\C) \cup \{\infty\}$ may be presented as a CW complex whose positive dimension cells $\Conf_\lambda(\C)$ (of dimension $n+\ell(\lambda)$) are indexed by ordered partitions of $n$.  The boundary of $\Conf_\lambda(\C)$ is the union of $\Conf_\rho(\C)$ over all $\rho$ such that $\lambda$ is a refinement of $\rho$. 

\end{prop}

From this we can write down an explicit cellular chain complex for $\Conf_n(\C) \cup \{\infty\}$:

\begin{defn}[The Fox-Neuwirth/Fuks complex]  \label{conf_defn}

For integers $i$ and $j$, let 
$$c_{i, j} = \sum_{\tau} (-1)^{|\tau|}$$
be the sum of the signs\footnote{Here, the sign of a shuffle is that of the representative element of $S_{i+j}$.  This sum is actually a binomial coefficient; see the proof of Proposition \ref{simplest_ex_prop}.} of all shuffles $\tau: \{1, \dots, i\} \sqcup \{1, \dots, j\} \to \{1, \dots, i+j\}$.  Let $C(n)_*$ denote the chain complex which in degree $q$ is generated over $\Z$ by the set of ordered partitions $\lambda = (\lambda_1, \dots, \lambda_{q-n})$ of $n$ with $q-n$ parts.  The differential $d: C(n)_q \to C(n)_{q-1}$ is given by the formula 
$$d(\lambda_1, \dots, \lambda_{q-n}) = \sum_{i=1}^{q-n-1} (-1)^{i-1} c_{\lambda_i, \lambda_{i+1}}(\lambda_1, \dots, \lambda_{i-1}, \lambda_i+\lambda_{i+1}, \dots, \lambda_{q-n})$$

\end{defn}

The complex $C(n)_*$ is isomorphic to the relative cellular chain complex of $\Conf_n(\C) \cup \{\infty\}$, relative to the point $\infty$.  With coefficients modulo 2, this was originally obtained by Fuks in \cite{fuks}, and integrally by Va{\u\i}n{\v{s}}te{\u\i}n in \cite{vainstein}; see also \cite{vassiliev}.  Thus,
$$H_*(\Conf_n(\C) \cup \{ \infty \}, \{ \infty \}) \cong H_*(C(n)_*).$$
We note the shift of $q$ by $n$ in Definition \ref{conf_defn}; this reflects the fact $H_*(C(n)_*)$ is linearly dual to the compactly supported cohomology of $\Conf_n(\C)$.  The latter is supported in degrees $q=n+1, \dots, 2n$ since $B_n$ has homological dimension $n-1$.

An explanation of the formula for the differential: the constant $c_{\lambda_i, \lambda_{i+1}}$ in the definition of $C(n)_*$ is recording the fact that the part of the boundary of the cell $\Conf_\lambda(\C)$ which combines the columns of $\lambda_i$ and $\lambda_{i+1}$ points in a vertical line does so by shuffling those points together into a single vertical column.  The signs in the formula arise from the induced orientations on the boundary strata, and a general scheme for explaining them is given in \cite{giusti-sinha}.  

\subsection{The cellular chain complex with local coefficients} \label{cell_section}

Let $L$ be a representation of $B_n$, and $\LL$ the associated local system over $\Conf_n(\C)$.  
Since $\LL$ trivializes on the open cells of the Fox-Neuwirth/Fuks stratification, it follows from the previous section that the cellular chain complex with local coefficients in $\LL$ is in fact isomorphic (as a graded group) to $C(n)_* \otimes L$.  To understand the differential, however, we must incorporate the braid action on $L$.

Specifically, we define a differential $d$ on $C(n)_* \otimes L$ by
$$d[(\lambda_1, \dots, \lambda_{q-n}) \otimes \ell] = \sum_{i=1}^{q-n-1} (-1)^{i-1} \left[(\lambda_1, \dots, \lambda_{i-1}, \lambda_i+\lambda_{i+1}, \dots, \lambda_{q-n}) \otimes \sum_{\tau} (-1)^{|\tau|}\tautilde(\ell)\right]$$
where $\tau$ is drawn from the shuffles of $\lambda_i$ with $\lambda_{i+1}$, and $\tautilde$ is its lift (as described in section \ref{qsa_section}) to the copy of $B_{\lambda_i+ \lambda_{i+1}} \leq B_n$ consisting of the braids that are only nontrivial on the $\lambda_i+ \lambda_{i+1}$ strands starting with the $\lambda_1 + \dots + \lambda_{i-1}+1\st$.

\begin{thm} \label{complex_thm}

There is an isomorphism $H_*(\Conf_n(\C) \cup \{ \infty \}, \{ \infty \}; \LL) \cong H_*(C(n)_* \otimes L)$.

\end{thm}

\begin{proof}

Given a cell complex structure on a space $X$, one can lift the cells to the universal cover $\tilde{X}$ to obtain a cell decomposition of $\tilde{X}$. Given a $\pi_1(X)$ representation $L$, the homology of $X$ with coefficients in $L$ can be computed as the homology of the cellular complex $C^{CW}_*(\tilde{X})\otimes_{\mathbb{Z}\pi_1(X)} L$.  We may describe this complex explicitly.  The set of cells of $\tilde{X}$ admits the structure of a right torsor under $\pi_1(X)$ with quotient given by the set of cells of $X$. Noncanonically, we may identify this with the set of pairs consisting of a cell of $X$ and an element of $\pi_1(X)$. 

In our case, let $\widetilde{C(n)_*}$ be the cellular chain complex of the universal cover on $\Conf_n(\C)$ obtained by lifting the Fox-Neuwirth/Fuks cells. We may noncanonically identify
	\[ \widetilde{C(n)_q} \cong \Z \{ ( (\lambda_1, \ldots, \lambda_{q-n}), b) | b \in B_n \} \]
as right $B_n$ representations (where $B_n$ acts by right multiplication on itself on the right side). We describe one choice of such an identification which gives the desired description of the differentials.

The top dimensional cells of $\widetilde{C(n)_q}$ occur when $q=2n$ and are of the form $((1, \ldots, 1), b)$ where $b \in B_n$.  Consider the codimension 1 faces of the cell $((1, \ldots, 1), b)$ obtained when the $i\nth$ and $i+1\st$ points in $\mathbb{C}$ lie on a single vertical line. There are two ways for this to happen. Either the $i\nth$ point is below the $i+1\st$ point, or it is above the $i+1\st$ point. We identify the former with the cell labelled $((1, \ldots,1, 2, 1, \ldots, 1),b)$ (where the $2$ is in the $i\nth$ position), while the later is the cell labelled by $((1, \ldots, 1,2,1, \ldots, 1), b'b)$, where $b'$ is the braid swapping the $i\nth$ strand under the $i+1^\st$ strand. Note that this choice of labelling is consistent with the right action of $B_n$.

More generally, the cell $((\lambda_1, \ldots, \lambda_{q-n}), b)$ corresponds to the (high codimension) face of $((1, \ldots, 1), b)$ obtained by iteratively taking the first of the faces described above (leaving the element $b \in B_n$ unchanged).  Specifically, for each $i=1, \dots, q-n$, take the face of $((1, \ldots, 1), b)$ where the $\lambda_i$ points in the $i\nth$ column come together so that the points to the leftmost always lie below the points coming from the right.  If, instead, one arranged the face so that the points on the right come from nontrivial shuffles with the points on the left, then the element of $B_n$ that this cell should be labelled by is multiplied by the (lift to the braid group) of the shuffle, as in the definition of the quantum shuffle algebra.

Under this identification of the fibres of $(\lambda_1, \ldots, \lambda_{q-n})$ with $B_n$, it becomes clear that the $i\nth$ face map -- which collides the $i\nth$ and $i+1\st$ columns of points -- is given by
	\[ \partial_i((\lambda_1, \ldots, \lambda_{q-n}), \sigma) = \sum_{\tau} (-1)^{|\tau|}((\lambda_1, \ldots, \lambda_i + \lambda_{i+1}, \ldots, \lambda_{q-n}), \tilde{\tau} \sigma). \]
Here, the sum is over shuffles $\tau$ of sets of cardinality $\lambda_i$ and $\lambda_{i+1}$ and $\tilde{\tau}$ is the usual lift of a shuffle to $B_{\lambda_i+ \lambda_{i+1}} \leq B_n$.  The signs in this formula come from comparing the orientations of these cells.  The differential on this complex is given by the signed sum of face maps: $d = \sum_{i=1}^{q-n-1} (-1)^{i-1} \partial_i$. 

Given a $B_n$ representation $L$, we want to give a description of the chain complex $\widetilde{C(n)_*} \otimes_{\mathbb{Z}B_n} L$ and its differential.  First, we identify $(\underline{\lambda}, \sigma) \otimes \ell$ with $(\underline{\lambda}, 1) \otimes \sigma(\ell)$. Then, 
\begin{align*} 
	d[((\lambda_1, \dots, \lambda_{q-n}),\sigma) \otimes \ell] & = \sum_{i=1}^{q-n-1} \sum_{\tau} (-1)^{|\tau|+i-1} ((\lambda_1, \dots, \lambda_i+\lambda_{i+1}, \dots, \lambda_{q-n}), \tilde{\tau}\sigma) \otimes \ell  \\
	&= \sum_{i=1}^{q-n-1}  \sum_{\tau} (-1)^{|\tau|+i-1} ((\lambda_1, \dots, \lambda_i+\lambda_{i+1}, \dots, \lambda_{q-n}), 1) \otimes \tilde{\tau}\sigma(\ell). 
\end{align*}
Identifying $\widetilde{C(n)_*} \otimes_{\mathbb{Z}B_n} L \cong C(n)_* \otimes L$ via $(\underline{\lambda}, \sigma)\otimes \ell \mapsto (\underline{\lambda}, 1) \otimes \sigma(\ell)$ then gives the desired description of the chain complex and differential.
\end{proof}

\subsection{A homological algebraic interpretation} 

Let $A$ be a connected, graded algebra, and write $I:= A_{>0}$ for the augmentation ideal consisting of elements of positive degree.  For a right $A$-module $M$ and left module $N$, we consider the \emph{normalized} (or \emph{reduced}) \emph{two-sided bar complex} $B_*(M, A, N)$ with
$$B_q(M, A, N) = M \otimes I^{\otimes q} \otimes N.$$
Recall that the differential on the bar complex is given by
$$d(a_0 \otimes \dots \otimes a_{q+1}) := \sum_{i=0}^q (-1)^i  a_0 \otimes \dots \otimes a_{i-1} \otimes a_i a_{i+1} \otimes a_{i+2} \otimes \dots \otimes a_{q+1}.$$
where $a_0 \in M$, $a_{q+1} \in N$, and $a_i \in I$ if $0<i<q+1$.   The bar complex computes $\Tor$:
$$H_q B_*(M, A, N) \cong \Tor^{A}_q(M, N).$$

When $M = N = I$, $B_q(I, A, I) = I^{\otimes q+2}$ may be extended to one which is nonzero for $q \geq -2$ by the same formula.  That is, we define:
$$B^e_q(I, A, I) = I^{\otimes q+2}, \; \mbox{for $q \geq -1$.  If $q=-2$, set} \; B^e_{-2}(I, A, I) = k.$$
Equip $B^e_q(I, A, I)$ with the same differential $d: B^e_q \to B^e_{q-1}$ for $q \geq 0$.  That is, the bar differential above is still well-defined for $q \geq 0$, as well as for $q=-1$, if we choose to interpret it as $d=0$ (which we do).  Furthermore, it is still the case that $d^2 = 0$ since $A$ is associative.

There is a short exact sequence of $A$-modules
\beqn \label{aug_ses_eqn} 0 \to I \to A \to k \to 0 \eeqn
where $k = A/I$ is equipped with the trivial $A$-module structure. 

\begin{prop}  \label{shift_by_2_prop}

The connecting homomorphisms associated with (\ref{aug_ses_eqn}) induce a quasi-isomorphism $B^e_*(I, A, I) \simeq B_{*+2}(k, A, k)$.

\end{prop}

\begin{proof}

The natural chain projection $B_*^e(I, A, I) \to B_*(I, A, I)$ is an isomorphism in non-negative degrees, so the map $H_p B_*^e(I, A, I) \to \Tor_p^{A}(I, I)$ it induces in homology is an isomorphism if $p>0$.  

Further, the sequence (\ref{aug_ses_eqn}) yields a long exact sequence
$$\cdots \to \Tor_{p+1}^A(A, I) \to \Tor_{p+1}^A(k, I) \to \Tor_p^A(I, I) \to \Tor_p^A(A, I) \to \cdots$$
For $p>0$, the outer terms are 0, and so $\Tor_{p+1}^A(k, I) \cong \Tor_p^A(I, I)$.  We may apply the same argument on the second module variable in $\Tor$ (but with $k$ in the first module) and obtain a chain of isomorphisms
\beqn \label{H_shift_eqn} \Tor_p^A(I, I) \cong \Tor_{p+1}^A(k, I) \cong \Tor_{p+2}^A(k, k) \mbox{ if $p>0$,} \eeqn
and so the connecting maps give a zigzag of chain complexes
\beqn B^e_{*, n}(I, A, I) \leftarrow \cdots \rightarrow B_{*+2, n}(k, A, k) \label{shift_eqn} \eeqn 
which is an isomorphism in homology in degrees $*>0$.

In the lowest degree, it is clear that $\Tor_{0}^A(k, k) = k = H_{-2} B^e_{*}(I, A, I)$.  Similarly, it is apparent from examining low degree terms in the bar complex that both $\Tor_1^{A}(k, k)$ and  $H_{-1}B^e_{*}(I, A, I)$ are isomorphic to $I/ I^2 = Q(A)$.

In $\Tor$ degree 2, we may use one application of (\ref{H_shift_eqn}) to obtain an isomorphism $\Tor_{2}^A(k, k) \cong \Tor_1^A(k, I)$, and then
$$0 \to \Tor_{1}^A(k, I) \to I \otimes_{A} I \to I \to k \otimes_{A} I.$$
That is, $\Tor_{2}^A(k, k) \cong \Tor_1^A(k, I)$ is the kernel of the map $I \otimes_{A} I \to I$ given by multiplication.  It is easy to identify this with $H_{0} B^e_{*}(I, A, I)$, so we conclude that (\ref{shift_eqn}) may be extended to all degrees.

\end{proof}

For an element $a = a_0 \otimes \dots \otimes a_{q+1} \in B^e_q(I, A, I)$ with $a_i$ homogenous elements of $I$ of degree $\deg(a_i)$, we may define the degree of $a$ to be $\deg(a) = \sum \deg(a_i)$.  The differential in $B^e_*(I, A, I)$ strictly preserves the degree of elements; we will write $B^e_{*, n}(I, A, I)$ for the (split) subcomplex generated by homogeneous elements of degree precisely $n$.  

\begin{defn}

Write $\epsilon$ for the braided $k$-module $\epsilon = k$ with braiding on $\epsilon^{\otimes 2} \cong k$ given by multiplication by $-1$.   For a general braided $k$-module $(V, \sigma)$, write $V_{\epsilon} = V \otimes \epsilon$ with braiding twisted by the sign on $\epsilon$.

\end{defn}

If $V = V(R, x)$ is the braided $k$-module associated to a rack $R$ and cocycle $x$, then $V_\epsilon = V(R, -x)$ is obtained by twisting the cocycle by the sign representation of the rack, as defined in Section \ref{rack_section}. 

\begin{prop} \label{alg_top_prop}

There is an isomorphism of chain complexes
$$B^e_{*, n}(I, \mA(V_\epsilon), I) \cong C(n)_{n+2+*} \otimes V^{\otimes n}.$$

\end{prop}

\begin{proof}

Note that $B^e_{q, n}(I, \mA(V_\epsilon), I)$ is the summand of $I \otimes I^{\otimes q} \otimes I$ spanned by all spaces of the form
$$V_{\epsilon}^{\otimes m_0} \otimes (V_{\epsilon}^{\otimes m_1} \otimes \dots \otimes V_{\epsilon}^{\otimes m_q}) \otimes V_{\epsilon}^{\otimes m_{q+1}}$$
where $\sum m_i = n$.  Of course, such a space is isomorphic to $V^{\otimes n}$, but it also carries the information of an ordered partition of $n$ with $q+2$ parts.  So there is an isomorphism of $k$-modules between $B^e_{q, n}(I, \mA(V_\epsilon), I)$ and $C(n)_{n+2+q} \otimes V^{\otimes n}$.  Further, it is apparent from the definition in section \ref{cell_section} that the differential on the latter is precisely the differential in the bar complex for $\mA(V_\epsilon)$ -- the signs coming from $\epsilon$ encode the boundary orientations on cells in the Fox-Neuwirth/Fuks model.

\end{proof}

\begin{cor} \label{main_cor}

There are isomorphisms\footnote{Note that dual braided vector spaces $V_\epsilon$ and $V^*_{\epsilon}$ are used in the $\Tor$ and $\Ext$ calculations.}
$$H^q(B_n; V^{\otimes n}) \cong \Tor^{n-q, n}_{\mA(V_\epsilon)}(k, k) \; \mbox{ and } \; H_q(B_n; V^{\otimes n}) \cong \Ext^{n-q, n}_{\mA(V^*_\epsilon)}(k, k).$$
Furthermore, the natural multiplication on the braid homology is carried to the Yoneda product on $\Ext$; that is, 
$$\bigoplus_{n=0}^\infty H_*(B_n; V^{\otimes n}) \cong \bigoplus_n \Ext^{n-*, n}_{\mA(V^*_\epsilon)}(k, k)$$
is an isomorphism of bigraded rings. 

\end{cor}

In the case that $V=k$ is the sign representation of $B_n$, this result was established by Markaryan \cite{markaryan}; Cohen also computed these homology groups using different techniques in \cite{clm}.  Callegaro adapted this argument to the setting where the generating braids operate on $V = k$ by multiplication by $q \in k^{\times}$ \cite{callegaro}.  Both then computed these cohomologies; we revisit these computations in section \ref{q_section}.  In \cite{ks}, Kapranov and Schechtman also establish this isomorphism.

\begin{proof}[Proof of Corollary \ref{main_cor}]

Let $\LL$ be a local system over $\Conf_n(\C)$.  The compactly supported cohomology of $\LL$ is dual to a relative homology group of the one-point compactification, but one must be careful to dualize the local system:
$$H^p_c(\Conf_n(\C), \LL) = H^p(\Conf_n(\C) \cup \{ \infty \}, \{ \infty \}, \LL) \cong H_p(\Conf_n(\C) \cup \{ \infty \}, \{ \infty \}, \LL^*)^*.$$

Now take $\LL$ to be the local system corresponding to $V^{\otimes n}$, so that $\LL^*$ corresponds to $(V^*)^{\otimes n}$.  Using Theorem \ref{complex_thm}, we have an isomorphism
$$H_*(\Conf_n(\C) \cup \{ \infty \}, \{ \infty \}; \LL^*) \cong H_*(C(n)_* \otimes (V^*)^{\otimes n}).$$
Then for the second isomorphism, we employ Poincar\'{e} duality with local coefficients for the (orientable) non-compact $2n$-manifold $\Conf_n(\C)$:
$$H_q(B_n; V^{\otimes n}) = H_q(\Conf_n(\C); \LL) \cong H^{2n-q}_c(\Conf_n(\C); \LL) \cong H_{2n-q}(C(n)_* \otimes (V^*)^{\otimes n})^*.$$
The Propositions \ref{shift_by_2_prop} and \ref{alg_top_prop} identifiy the latter cochain complex:
$$(C(n)_{n+*} \otimes (V^*)^{\otimes n})^* \cong B_{*-2, n}^e(I, \mA(V^*_{\epsilon}), I)^* \simeq B_{*, n}(k, \mA(V^*_{\epsilon}), k)^*$$
Take cohomology in degree $*=n-q$ to obtain
$$H_q(B_n; V^{\otimes n}) \cong \Ext_{\mA(V_\epsilon^*)}^{n-q, n}(k, k),$$
using the fact that the dual of the bar complex $B_*(k, \mA(V_\epsilon^*), k)$ computes $\Ext_{\mA(V_\epsilon^*)}^*(k, k)$.   The dual statement for cohomology using $\Tor$ proceeds similarly, but does not require dualization of $\LL$ (and hence $V_{\epsilon}$).

We address the ring structure.  Multiplication in the homology of configuration spaces is gotten by the little disks multiplication which places configurations side by side.  This extends to homology with coefficients in the local system associated to $V^{\otimes n}$ by virtue of Proposition \ref{monoidal_equiv_prop}: these are precisely monoidal representations for this multiplication.  On the other hand, the multiplication in the cobar complex $\Omega(A^*) = \Hom(B_{*}(k, A, k), k)$ is gotten by juxtaposition of tensors.  Tracing this through the isomorphisms above, this is precisely the cellular operation corresponding to the little disks multiplication. 

\end{proof}

\subsection{An example: the configuration space of the plane} \label{conf_section}

In the case of the trivial coefficients $V = k$, the previous construction recovers the homologies\footnote{The results of this and the next section are not needed for the proof of our main result, Theorem \ref{th:mainmalle}, but are instructive on how to approach these computations.  They will hopefully also provide some context for the reader familiar with classical computations of the cohomology of configuration space with trivial (or nearly trivial) coefficients.} of $\Conf_n(\C)$ as $n$ ranges over $\Z_{\geq 0}$.  In this case, the quantum shuffle algebra $\mA(V_\epsilon)$ is generated (as a $k$-module) by the classes $x_n = [1|1|\dots |1]$, where there are $n$ occurrences of $1$.

\begin{prop} \label{simplest_ex_prop}

The algebra $\mA(V^*_\epsilon)$ is isomorphic to $\Lambda[x_1] \otimes \Gamma[x_2]$.

\end{prop}

Here, $\Gamma[x_2]$ is the divided power algebra generated by $x_2$.  It is generated additively by the classes $x_{2n}$, subject to the relations
\beqn x_{2m} \star x_{2n} = {m+n \choose n} x_{2m+2n} \label{div_power_eqn}. \eeqn

\begin{proof}

As in Definition \ref{conf_defn}, $x_m \star x_n = c_{m, n} x_{m+n}$, where 
$$c_{m, n} = \sum_{\tau} (-1)^{|\tau|}$$
where the sum is over shuffles $\tau: \{1, \dots, m\} \sqcup \{1, \dots, n\} \to \{1, \dots, m+n\}$.  This is evidently commutative, and it is easy to see that 
$$c_{1, m} = c_{m,1} = \left\{ \begin{array}{ll} 1, & \mbox{ $m$ is even} \\ 0, & \mbox{ $m$ is odd.} \end{array} \right.$$
Thus $x_1^2 = 0$ and $x_{2n+1} = x_1 \star x_{2n}$.  The result then follows from the fact that $c_{2m, 2n} = {m+n \choose n}$, which can be shown via a straightforward recursion, or using Proposition \ref{stanley_prop}.

\end{proof}

The following is then immediate:

\begin{cor}

For $\mA= \Lambda[x_1] \otimes \Gamma[x_2]$, there is an isomorphism of $k$-algebras
\beqn \label{conf_eqn}H_*(\coprod_n\Conf_{n}(\C); k) \cong \Ext^{*}_\mA(k, k).\eeqn

\end{cor}

In this and the next section, we will frequently encounter the cohomology of a divided power algebra $\Gamma[x_m]$ on a generator in degree $m$.  We recall that in characteristic zero, $\Gamma[x_m] \cong k[x_m]$, so
\beqn \Ext_{\Gamma[x_m]}(k, k) = \Lambda[z_m] \label{char0_div_ext_eqn} \eeqn
is an exterior algebra on a generator $z_m \in \Ext^{1, m}$.  

Taking $k$ to be a field of characteristic $p$, it is well-known that there is an isomorphism
$$\Gamma[x_m] = \bigotimes_{i=0}^\infty P_p[x_{mp^i}]$$
where $P_p[z] = k[z]/z^p$ is the truncated polynomial algebra on a single generator.  If $p=2$, this is an exterior algebra, so we obtain $\Gamma[x_m] \cong \Lambda[x_m, x_{2m}, \dots, x_{m2^i}, \dots]$.  Then the cohomology is a polynomial algebra
\beqn \Ext^{*}_{\Gamma[x_m]}(k, k)\cong k[y_{m}, y_{2m}, \dots, y_{m2^i}, \dots], \label{char2_div_ext_eqn} \eeqn
where $y_{m2^i} \in \Ext^{1, m2^i}$.  Finally, in odd characteristic, there is an isomorphism
\beqn \Ext_{\Gamma[x_m]}^{*, *}(k, k) = \bigotimes_{i=0}^\infty \Lambda[z_{mp^i}] \otimes k[y_{mp^{i+1}}], \label{charp_div_ext_eqn} \eeqn
where $y_{mp^{i+1}} \in \Ext^{2, mp^{i+1}}$ and $z_{mp^i} \in \Ext^{1, mp^{i}}$ come from the factor $\Ext_{P_p[x_{mp^i}]}(k,k)$.

We use these computations and the previous corollary to understand the cohomology of configuration spaces.  Throughout, the tensor factor of $\Lambda[x_1]$ in $\mA= \Lambda[x_1] \otimes \Gamma[x_2]$ will produce a tensor factor of $k[y_1]$ in $\Ext$, where $y_1 \in \Ext^{1, 1}$.  In characteristic zero, we then have
$$H_*(\coprod_n\Conf_{n}(\C); k) \cong k[y_1] \otimes \Lambda[z_2].$$
Topologically, $y_1 \in \Ext^{1, 1}$ is a 0-dimensional homology class coming from $\Conf_1(\C)$, and $z_2 \in \Ext^{1, 2}$ is a 1-dimensional class coming from $\Conf_2(\C)$.  

Then $H_*(\Conf_n(\C); k)$ is generated by $y_1^n$ (of dimension 0) and $y_1^{n-2} z_2$ (of dimension 1).  This reflects the known result that $\Conf_n(\C)$ is contractible when $n=0, 1$, and has the rational homology of a circle when $n>1$.  We note that the stabilisation map $\Conf_n(\C) \to \Conf_{n+1}(\C)$ which adds a point near infinity is given by multiplication by $y_1$.  From the above computations, this is an isomorphism in $H_i$ with $i \neq 1$ for every $n$.  In $H_1$, it is evidently an isomorphism as long as $n>1$.

In characteristic $p=2$, we similarly have
$$\Ext^{*}_\mA(k, k)\cong k[y_1, y_2, \dots, y_{2^i}, \dots],$$ 
Topologically, $y_{2^i}$ is a class $y_{2^i} \in H_{2^i-1}(\Conf_{2^i}(\C); k)$.

The approximation theorem gives us a map $\coprod_n\Conf_{n}(\C) \to \Omega^2 S^2$, and the group completion theorem implies that the homology of the codomain is obtained from that of the domain by inverting $\pi_0$.  However, a generator of $\pi_0 \cong \Z_{>0}$ is given by $y_1$, so
$$H_*(\Omega^2 S^2; k) \cong k[y_1, y_2, \dots, y_{2^i}, \dots][y_1^{-1}].$$
This is a familiar computation (see, e.g., \cite{clm}); one may identify $y_{2^i} = Q_1(y_{2^{i-1}})$ in terms of iterates of the first (and only) Kudo-Araki-Dyer-Lashof operation applied to $y_1$.  

One may obtain a proof of the (well-known) homological stability in this setting, too.  We note that a monomial
$$y_1^{e_1} y_2^{e_2} \cdots y_m^{e_m} \in H_q(\Conf_n(\C); k) = \Ext^{n-q, n}_\mA(k,k)$$
precisely when $\sum e_i = n-q$, and $\sum e_i 2^i = n$.  It is apparent from this description that stabilisation (multiplication by $y_1$) is always injective.  Let us show that stabilisation $H_{q}(\Conf_{n-1}(\C); k) \to H_q(\Conf_n(\C); k)$ is surjective when $q< \frac{n}{2}$.  In this case, $\sum e_i> \frac{n}{2}$.  However, if $e_1 = 0$, then $\sum e_i 2^i \geq \sum e_i 2 > n$, a contradiction.  Thus $e_1 > 0$, and so multiplication by $y_1$ is surjective.  A similar analysis in odd characteristics is left to the reader.

Finally, it is worth mentioning that the ring $H^*(\Omega S^2; k)$ is isomorphic to the quantum shuffle algebra $\mA$ for any base field $k$.  The dual Eilenberg-Moore spectral sequence
$$\Ext_{H^*(\Omega S^2; k)}^*(k, k) \implies H_*(\Omega^2 S^2; k)$$
does not converge, as $\Omega^2 S^2$ is not connected.  Rather, equation (\ref{conf_eqn}) tells us that the (collapsing) spectral sequence computes the homology of the disjoint union of the configuration spaces.  This result, along with the group completion theorem provides a sort of replacement for the Eilenberg-Moore spectral sequence. 

\subsection{A second example: The quantum divided power algebra} \label{q_section}

As in the previous section, we will take the braided vector space $V=k$ to be rank one, but equip $V$ with a nontrivial braiding. Specifically, for $q \in k^{\times}$, we equip $V \otimes V = k$ with the braiding that is given by multiplication by $q$.  Then for each $n>0$ and $1 \leq i <n$, 
$$\sigma_i\mapsto q\;  \mbox{ in the representation }\;  B_n \to \GL(V^{\otimes n}) = k^{\times}.$$
In unpublished work, Marshall Smith computed the structure of the quantum shuffle algebra for this braided vector space, as well as its cohomology.  The same results had been obtained previously by Callegaro in slightly different language \cite{callegaro}.  We summarize these results here.

\begin{defn}

The \emph{quantum divided power algebra} $\Gamma_q[x]$ associated to $q \in k^{\times}$ is additively generated by elements $x_n$ in degree $n$, equipped with the product
$$x_n \star x_m := {n+m \choose m}_q x_{n+m}$$
where the quantum binomial coefficient is defined by
$${a \choose b}_q = \frac{[a]_q [a-1]_q \cdots [a-b+1]_q}{[b]_q [b-1]_q \cdots [1]_q}; \; \mbox{ here } \; [r]_q = \frac{1-q^r}{1-q} = 1+q+q^2 + \dots + q^{r-1}$$

\end{defn}

\begin{prop} \label{stanley_prop} There is an isomorphism of graded rings $\Gamma_{q}[x] \to \mA(V)$ which carries the class $x_n$ to $[1|1|\dots |1]$, where there are $n$ occurrences of $1$.
\end{prop}

\begin{proof}

This amounts to the well-known fact (see, e.g., Prop. 1.7.1 of \cite{stanley}) that one may obtain ${n+m \choose m}_q$ via the weighted sum over all $(m, n)$-shuffles $\tau$ of the expression $q^{cr(\tau)}$, where $cr(\tau)$ is the number of crossings in $\tau$.

\end{proof} 

The following is essentially identical to Lemma 3.4 of \cite{callegaro}:

\begin{prop} If $q$ is not a root of unity in $k$, then there is an isomorphism $\Gamma_q[x] \cong k[x_1]$.  If $q$ is a primitive $m\nth$ root of unity, then 
$$\Gamma_q[x] = k[x_1]/x_1^{m} \otimes \Gamma[x_m].$$

\end{prop}

If we include the sign twist $\epsilon$, we have $\mA(V^*_\epsilon) = \Gamma_{-q}[x]$, so
$$H_j(B_n, V^{\otimes n}) \cong \Ext_{\Gamma_{-q}[x]}^{n-j, n}(k, k).$$
We may compute this as follows:
\begin{enumerate}
\item \label{item1} If $-q$ is not a root of unity in $k$, 
$$\Ext_{\Gamma_{-q}[x]}^{*, *}(k, k) = \Lambda[z_1]$$
is an exterior algebra on a generator $z_1$ of topological degree $0$, in filtration $1$.  That is,
$$H_j(B_n, V^{\otimes n}) = \left\{\begin{array}{ll} k, & j=0, \; \mbox{and} \; n=0, 1 \\ 0, & \mbox{otherwise.} \end{array} \right.$$

\item If $q=1$, this is the homology of the braid groups with trivial coefficients, as in the previous section. 

\item If $q=-1$, $\Gamma_{-q}[x] = \Gamma[x_1]$ is the divided power algebra on a generator in degree 1.  If $k$ has characteristic $0$, this is isomorphic to $k[x_1]$, with the same cohomology as in the first case.  In positive characteristic $p$, the cohomology of this algebra was computed in (\ref{char2_div_ext_eqn}) and (\ref{charp_div_ext_eqn}) with $m=1$.

\item If $-q$ is a primitive $m^\nth$ root of unity in $k$, where $m > 2$, then
\beqn \Ext_{\Gamma_{-q}[x]}^{*, *}(k, k) = \Lambda[z_1] \otimes k[y_m] \otimes \Ext_{\Gamma[x_m]}(k, k) \label{ext_div_eqn} \eeqn
where $z_1$ has topological degree 0, filtration 1, and $y_m$ has topological degree $m-2$, filtration $m$. The cohomology of $\Gamma[x_m]$ is described in the previous section.

\end{enumerate}

One important subtlety which is hidden in the presentation of (\ref{ext_div_eqn}) as a tensor product is that the tensor factors do not actually commute as usual in cohomology, but rather in a braided sense.  More specifically, it was shown in \cite{mpsw} that the cohomology of a braided Hopf algebra is a braided commutative algebra: cohomology classes $\alpha$ and $\beta$ satisfy
$$\alpha \beta = \mu(\sigma(\alpha \otimes \beta)),$$
where $\mu$ is multiplication, and $\sigma$ the braiding in our category.  In the case at hand, this braiding is given by multiplication by powers of $\pm q$, and depends upon cobar complex representatives of $\alpha$ and $\beta$.  However, we will see in this example that the only departure from graded commutativity is via altered signs.

We note that the cobar complex $\Omega^*(A^*) = \Hom(B_*(A), k)$ which computes $\Ext_A(k, k)$ can be presented as the tensor algebra $T(\Sigma I)$ on the shift of the augmentation ideal $I \subseteq A^*$.  The braiding on $A^*$ is dual to that on $A$, and the braiding on the shift $\Sigma I$ the negative of the braiding on $I$.  This is extended multiplicatively to the tensor algebra\footnote{See section \ref{index_section} for more details on this.}.

For instance, when $-q$ is a primitive $m^\nth$ root of unity, cobar complex representatives for $z_1$ and $y_m$ are dual to $[x_1]$ and $[x_1^{m-1}|x_1]$, respectively.  Incorporating the shifts, we have 
$$\sigma([x_1] \otimes [x_1^{m-1}|x_1]) = (-1)^{1 \cdot 2} (-q)^{1 \cdot m}[x_1^{m-1}|x_1] \otimes [x_1] = [x_1^{m-1}|x_1] \otimes [x_1]$$
so $z_1$ and $y_m$ do in fact commute.  However, there is also a class $z_m \in \Ext_{\Gamma[x_m]}(k, k)$  of topological degree $m-1$ and filtration $m$; it is represented by a class dual to $[x_m]$.  Then $z_1$ and $z_m$ anticommute: 
$$\sigma([x_1]\otimes[x_m]) = (-1)^{1\cdot1} (-1)^{1 \cdot m} [x_m] \otimes [x_1] = - [x_m]\otimes[x_1]$$
which may be surprising, given that the topological dimension of $z_1$ is $0$.

\section{Algebraic properties of quantum shuffle algebras} 

In this section we work with a fixed braided vector space $V$ over a field $k$ given by a Yetter-Drinfeld module for a group $G$.  We will write $\mA = \mA(V)$ and $\mB = \mB(V)$.  The goal of this section is to decompose the associated graded algebra $\mA^{\gr}$ of the filtration on $\mA$ by powers of its augmentation ideal into a (twisted) tensor product, and begin to compute its cohomology via Cartan-Eilenberg spectral sequences.

\subsection{Twisted tensor product decompositions}

Let $\mC$ and $\mD$ be graded, braided Hopf algebras in $\YD^G_G$, and let $\rho: \mC \to \mD$ be a map of Yetter-Drinfeld coalgebras.  This makes $\mC$ into a right $\mD$-comodule.  Recall that the cotensor product with $k$ is
$$\mC \square_{\mD} k = \{ x \in \mC \; | \; (\id \otimes \rho) \circ \Delta(x) = x \otimes 1 \}.$$
This is sometimes written $\mC^{\co \mD}$.  If $\rho$ is actually a map of bialgebras, then it is easy to see that $\mC \square_{\mD} k$ is a subalgebra of $\mC$.  The following is a special case of Theorem 3.2 of \cite{milinski-schneider}, which in turn is a case of the braided formulation of the fundamental theorem of Hopf modules.

\begin{thm} \label{ms_thm} Assume that $\rho: \mC \to \mD$ is a map of Yetter-Drinfeld bialgebras, and $i: \mD \to \mC$ is a map of Yetter-Drinfeld bialgebras with $\rho \circ i = \id_{\mD}$.  Then the multiplication map
$$\mu: (\mC \square_{\mD} k) \otimes \mD \to \mC$$
which carries $x \otimes y$ to $x \star i(y)$ is an isomorphism of left $\mC \square_{\mD} k$-modules, and right $\mD$-comodules.

\end{thm}

One can say a little more.  Let $\alpha: \mD \otimes (\mC \square_{\mD} k) \to (\mC \square_{\mD} k) \otimes \mD$ be the map carrying $d \otimes c$ to $\mu^{-1}(d \star c)$.  Define the \emph{twisted tensor product} $(\mC \square_{\mD} k) \otimes_{\alpha} \mD$ to be $(\mC \square_{\mD} k) \otimes \mD$, equipped with the multiplication
$$(c_1 \otimes d_1) \cdot (c_2 \otimes d_2) := \sum (c_1 c_2^i) \otimes (d_1^i  d_2)$$
where $\alpha(d_1 \otimes c_2) = \sum_i c_2^i \otimes d_1^i$.  Given Theorem \ref{ms_thm}, the following is nearly definitional:

\begin{cor} \label{twisted_tensor_cor} The map $\mu: (\mC \square_{\mD} k) \otimes_{\alpha} \mD \to \mC$ is an isomorphism of Yetter-Drinfeld algebras. \end{cor}

\subsection{The word-length filtration on $\mA$}

In this section, we examine the filtration of $\mA$ by powers of its augmentation ideal, and decompose the associated graded algebra $\mA^{\gr}$ into a twisted tensor product. 

Write $\mA_{>0}$ for the augmentation ideal of the quantum shuffle algebra, and let $Q(\mA) = (\mA_{>0})/(\mA_{>0})^2$ denote the graded module of indecomposables in $\mA$.  
Define a decreasing filtration of $\mA$ by setting $F_k(\mA) = (\mA_{>0})^k$ to be the $k\nth$ power of the augmentation ideal.  Then $F_k(\mA)$ is generated by words in the indecomposables in $\mA$ of length at least $k$.  This is obviously a filtration by ideals; in fact, multiplication in $\mA$ carries $F_p(\mA) \otimes F_{p'}(\mA)$ to $F_{p+p'}(\mA)$.  

\begin{defn} Let $\grA$ be the bigraded $k$-algebra associated to this filtration; $\grA_{p, q} = F_p(\mA_q)/F_{p+1}(\mA_q)$. \end{defn}

Notice that since $\mA$ is an algebra in $\YD^G_G$, the ideal $\mA_{>0}$ and all of its powers are Yetter-Drinfeld modules.  Therefore $\grA$ is a bigraded Yetter-Drinfeld algebra.  We note in particular that $Q(A) = \grA_{1, *}$ is a Yetter-Drinfeld module.

If we filter the tensor product $\mA \otimes \mA$ by 
$$F_k(\mA \otimes \mA) = \sum_{p+p'=k} F_p(\mA) \otimes F_{p'}(\mA),$$
then the coproduct on $\mA$ respects the filtration; $\Delta: F_k(\mA) \to F_k(\mA \otimes \mA)$.  This may be seen as follows: for an element $a \in F_1(\mA) = \mA_{>0}$, it is vacuously true that 
$$\Delta(a) \in F_1(\mA) \otimes F_0(\mA) + F_0(\mA) \otimes F_1(\mA) = \mA_{>0} \otimes \mA + \mA \otimes \mA_{>0},$$
since the diagonal is graded, and the right side is $(\mA \otimes \mA)_{>0}$.  For a generator $b = a_1 \cdots a_{k} \in F_k(\mA)$, $\Delta(b)$ is the product (with respect to the twisted multiplication in $\mA \otimes \mA$) 
$$\Delta(b) = \star_{i=1}^k \Delta(a_i) \in \star_{i=1}^k (F_1(\mA) \otimes F_0(\mA) + F_0(\mA) \otimes F_1(\mA)) \subseteq \sum_{p+q = k} F_p(\mA) \otimes F_q(\mA).$$
In the last containment, we have used the fact that the definition of multiplication in $\mA \otimes \mA$ uses the $G$-module structure as in equation (\ref{twist_eqn}), and that $F_k(\mA)$ is invariant under the $G$ action.  Therefore, $\Delta$ descends to a coassociative coproduct on $\grA$ of the form
$$\Delta: \grA_{p, q} \to \bigoplus_{a+b = p} \bigoplus_{c+d = q} = \grA_{a, c} \otimes \grA_{b, d}$$
Here we use the fact that $F_p(\mA \otimes \mA)/F_{p+1}(\mA \otimes \mA) = \oplus_{a+b = p} \grA_{a, *} \otimes \grA_{b, *}$.

Additionally, the antipode on $\mA$ preserves the filtration, so we may conclude:

\begin{prop} The $k$-algebra $\grA$ is a braided, bigraded Hopf algebra in the Yetter-Drinfeld category $\YD_G^G$. \end{prop}

Notice that since every indecomposable of $\mA$ has degree at least $1$, $F_p(\mA_q) = 0$ for $p>q$.  Also, the natural composite $i: \mB_p \to \mA_p \to \grA_{p, p} = F_p(\mA_p)$ is an isomorphism of vector spaces, since $\mB_p$ is precisely the subspace of $\mA_p$ generated by words of length $p$ (in the indecomposables $V$ of degree 1).   In fact, the diagonal subspace $\sum_{p} \grA_{p, p}$ is a Hopf subalgebra of $\grA$, and $i$ is a Hopf algebra isomorphism of $\mB$ onto the diagonal subspace of $\grA$.  If we define $L := \sum_{p<q} \grA_{p, q}$, then $L$ is a Hopf ideal in $\grA$.  The projection $\rho$ onto the diagonal subspace yields an isomorphism of Hopf algebras 
$$\grA / L \cong \sum_{p} \grA_{p, p} \cong \mB.$$

\begin{defn} Let $\mE := \grA \square_{\mB} k$ be the cotensor product of $\grA$ over $\mB$.  \end{defn}

Then $\mE$ is a subalgebra of $\grA$, and Corollary \ref{twisted_tensor_cor} (with $\mC = \grA$, and $\mD = \mB$) says: 

\begin{prop} \label{free_prop_3}

Multiplication gives an algebra isomorphism $\mu: \mE \otimes_\alpha \mB \to \grA$ of the twisted tensor product of $\mE$ and $\mB$ with $\grA$. 

\end{prop}

The structure of the algebra $\mE$ is far from clear.  However, since $\mA^{\gr}$ and $\mB$ agree in degree 1, $\mE$ is generated as an algebra in degrees 2 and higher.

\subsection{Change of rings and May spectral sequences} \label{ss_sec}

Our goal in this section is to develop two spectral sequences to assist in the computation of the cohomology of quantum shuffle algebras.  This cohomology is bigraded: recall that for a graded algebra $A$ and graded $A$-modules $M$ and $N$, $\Ext_A^{p, q}(M, N)$ is the component of $R\Hom_A^p(M, N)$ of degree $q$.  That is, if $R_\bullet \to M$ is a graded, projective resolution of $M$, then $\Ext_A^{p, q}(M, N)$ consist of those cohomology classes represented by $f:R_p \to N$ which decrease degree by $q$.

Our first spectral sequence computes the cohomology of associated graded algebra for the word-length filtration on quantum shuffle algebras:

\begin{prop} \label{ss_prop}

There is a first-quadrant spectral sequence of rings:
$$\Ext^{p, r}_{\mB}(k, \Ext^{q, s}_{\mE}(k, k)) \implies \Ext^{p+q, r+s}_{\grA}(k, k).$$

\end{prop}

\begin{proof}

This is a consequence of a standard Cartan-Eilenberg change of rings spectral sequence associated to the quotient ring map $\pi: \grA \to \mB$:
$$\Ext^{p, r}_{\mB}(k, \Ext^{q, s}_{\grA}(\mB, k)) \implies \Ext^{p+q, r+s}_{\grA}(k, k).$$
To get the result as stated, we note: if $R_\bullet \to k$ is a free $\mE$-resolution of $k$, then $R_{\bullet} \otimes_{\mE} \grA$ is a free $\grA$-resolution of $k \otimes_{\mE} \grA = \mB$.  Specifically, $R_{p} \otimes_{\mE} \grA$ is $\grA$-free, since it is induced from a free $\mE$-module.  The exactness of the complex follows from Proposition \ref{free_prop_3}: $\grA$ is free over $\mE$, hence flat.

Therefore, $\Ext_{\grA}(\mB, k)$ and $\Ext_{\mE}(k, k)$ are both computed as the cohomology of the complex 
$$\Hom_{\grA}(R_{\bullet} \otimes_{\mE} \grA, k) = \Hom_{\mE}(R_{\bullet}, k)$$
and so are equal.

\end{proof}

Additionally, we note that the since the filtration of $\mA$ by powers of the augmentation ideal is multiplicative, we immediately have a spectral sequence (see, e.g., \cite{may_restricted})
$$\Ext_{\grA}(k, k) \implies \Ext_{\mA}(k, k).$$

\section{Cohomological growth}

The purpose of this section is to explore the growth rate of various (bi-)graded modules associated a braided vector space $V$.  While some of these arguments work for arbitrary braided vector spaces, our focus will be on Yetter-Drinfeld modules for a finite group $G$.  This assumption allows us to show in Proposition \ref{sym_prop} that the associated ring $R$ of components has polynomial growth.  In section \ref{hyp_section} we will formulate several hypotheses on the growth of the cohomology of the quantum shuffle and Nichols algebras associated to $V$, and show how these are related; this will allow us, in subsequent sections, to focus on bounding the cohomology of the Nichols algebra.

\subsection{The ring $R$} \label{ring_comp_section} 

We continue to let $\mA(V_\epsilon^*)$ be the quantum shuffle algebra associated to the sign twist $V_{\epsilon}^*$ of the dual of a braided vector space $V$.  Then define $R$ to be the diagonal subalgebra
$$R := \bigoplus_{n=0}^\infty \Ext_{\mA(V^*_\epsilon)}^{n, n}(k, k) $$
of the cohomology.  Notice that we may make $k$ into an $R$-module via the augmentation $R \to k = \Ext_{\mA(V^*_\epsilon)}^{0, 0}(k, k)$.
We will refer to $R$ as the \emph{ring of coinvariants} or \emph{ring of components} associated to $V$, courtesy of the isomorphism
$$\Ext_{\mA(V^*_\epsilon)}^{n, n}(k, k) \cong H_0(B_n, V^{\otimes n})$$
of Corollary \ref{main_cor}.

\begin{defn}

For $n \geq 0$, let $r(n) := \rk (R_n)$.  

\end{defn}

We will see in Proposition \ref{sym_prop} that when $V$ is a Yetter-Drinfeld module, $r(n)$ grows polynomially as a function of $n$, and that its degree is bounded by $\# G \cdot \rk(V) -1$.  This bound is almost certainly not optimal, as can be seen in a number of cases.

$R$ may also be constructed as the diagonal cohomology of a number of other algebras:

\begin{lem} \label{components_lemma}

For any graded algebra $A$ with $A_1 = V^*_\epsilon$ whose relations amongst these generators in degree 2 are the same as those in $\mA(V^*_\epsilon)$, $\Ext_{A}^{n, n}(k, k) \cong R_n$. 

\end{lem}

Notice that the Nichols algebra $\mB(V^*_\epsilon) \leq \mA(V^*_\epsilon)$ is a minimal algebra with such generators and relations in degrees 1 and 2.

\begin{proof}

In the cobar complex $\Omega (\mA(V^*_\epsilon)^*)$ which computes $\Ext$, the only contribution to the diagonal bidegree $(n, n)$ must come from classes of the form $a_1 \otimes \dots \otimes a_n$ for $a_i \in V_\epsilon = \mA_1^*$.  Further, $\Ext^{n, n}_{\mA(V^*_\epsilon)}(k, k)$ can be written as the cokernel of a differential of the form 
\beqn \label{cobar_diff_eqn} \bigoplus_{i=0}^{n-2} (\mA^*_1)^{\otimes i} \otimes \mA^*_2 \otimes (\mA^*_1)^{\otimes n-2-i} \to (\mA_1^*)^{\otimes n}.\eeqn
Thus the cohomology in this bidegree is entirely computed on any algebra $A$ of the form indicated.

\end{proof}

\subsection{Growth of braided commutative algebras} 

An algebra $A$ in $\YD^G_G$ is \emph{braided commutative} if $\mu = \mu \circ \sigma$, where $\mu$ is its multiplication and $\sigma$ the braiding in $\YD^G_G$.  In this section we examine the growth of free braided commutative algebras.

\begin{defn}

Let $V$ be a Yetter-Drinfeld module for a group $G$.  The \emph{free braided commutative algebra} generated by $V$ is
$$\Sym_{\br}(V) := \bigoplus_{n=0}^\infty (V^{\otimes n})_{B_n}$$
where $(V^{\otimes n})_{B_n}$ indicates the coinvariants for the $B_n$-action.

\end{defn}

It is straightforward to verify that this does indeed have the requisite universal property: if $A$ is a braided commutative algebra, then there is a natural bijection between maps $V \to A$ in $\YD^G_G$, and algebra maps $\Sym_{\br} (V )\to A$.

Notice that the terms in the free braided commutative algebra may be identified in terms of either braid group homology or the cohomology of the quantum shuffle algebra $\mA(V^*_\epsilon)$:
$$(V^{\otimes n})_{B_n} = V^{\otimes n} \otimes_{k[B_n]} k = H_0(B_n, V^{\otimes n}) = \Ext^{n, n}_{\mA(V^*_{\epsilon})}(k, k)$$
so we may conclude:

\begin{prop} \label{sym_br_prop} There is a graded ring isomorphism $R \cong \Sym_{br}(V)$. \end{prop}

While the notion of braided commutativity is generally distinct from commutativity in the usual sense, many familiar facts do carry over in this setting.  For instance, if $I$ is a homogenous\footnote{This does not quite hold for inhomogenous ideals: For instance, if $i_1$ and $i_2$ are homogenous of degrees $g_1 \neq g_2$, then the left ideal generated by $i_1+i_2$ need not be a right ideal, since
$$(i_1 + i_2) \cdot r = i_1 \cdot r + i_2 \cdot r= {}^{g_1} r \cdot i_1 + {}^{g_2} r \cdot i_2$$
which need not be a left multiple of $i_1+i_2$.} left ideal in $R$, then it is also a right ideal: for a homogenous element $i \in I_g$ and $r \in R$, then
\beqn i \cdot r = {}^g r  \cdot i \in I. \label{left_right_eqn} \eeqn
Similarly, free braided commutative algebras have polynomial growth:

\begin{prop} \label{sym_prop}

Let $V$ be a finite dimensional Yetter-Drinfeld module for a finite group $G$, and let $d$ be the size of a $G$-invariant generating set $c \subseteq V$.  Then the rank $r(n) = \dim(\Sym_{\br}(V)_n)$ is bounded by ${n+d-1 \choose n}$.  In particular, $r(n)$ grows no faster than a polynomial in $n$ of degree $d-1$.

\end{prop}

We note that if $V$ is a permutation representation of $G$, then $c$ may be taken to be a basis, in which case $d = \rk(V)$.  In general $d \leq \#G \cdot \rk(V)$: if $b$ is an arbitrary basis of $V$, then $G\cdot b$ is a $G$-invariant generating set for $V$.

\begin{proof}

We let $c \subseteq V$ be a generating set which is invariant under the action of $G$.  Then $V^{\otimes n}$ is generated by $c^{\times n}$.  Notice that $B_n$ acts on $c^{\times n}$ via the Yetter-Drinfeld structure:
$$\sigma_i(x_1, \dots, x_n) = (x_1, \dots, x_{i-1}, x_{i+1}, x_i^{x_{i+1}}, x_{i+2}, \dots, x_n)$$
This remains in $c^{\times n}$ since $x_i^{x_{i+1}} \in G\cdot x_i \subseteq c$.  The orbits $c^{\times n} / B_n$ then descend to a generating set of $\Sym_{\br}(V)_n = (V^{\otimes n})_{B_n}$.

Pick a total ordering $<$ on $c$, entirely arbitrarily.  We claim that a generating set for $\Sym_{\br}(V)_n$ is given by monomials of the form $v_1  v_2  \dots  v_n$ where $v_i \in c$ and $v_i \geq v_{i+1}$.  This then gives the claimed bound in rank by comparison to the rank of the space of polynomials in $d$ variables of degree $n$.

Certainly $\Sym_{\br}(V)_n$ is generated by monomials of the indicated form without the ordering requirement.  However, we may inductively replace an unordered monomial with one in which the $v_i$ are ordered using braided commutativity.  Specifically, if $v_1 \cdots v_n$ is a monomial of length $n$ in $c$, let $\OO$ be the orbit of $(v_1, \dots, v_n)$ in $c^{\times n}$ under $B_n$.  This is a finite set; let $S \subseteq c$ be the subset of all of the generators which appear at some point in an element of $\OO$.  Now $S$ has a maximum element, $W$, which appears in a word $(w_1, \dots, w_{i-1}, W, w_{i+1}, \dots w_n) \in \OO$.  Then
$$v_1   \cdots  v_n = w_1  \cdots  w_{i-1}  W  w_{i+1}  \cdots  w_n = W  w_1^W  \cdots  w^W_{i-1}   w_{i+1}  \cdots  w_n.$$
Now induct; by construction, no element in the $B_{n-1}$-orbit of $(w_1^W, \dots, w^W_{i-1}, w_{i+1}, \dots, w_n)$ contains an element larger than $W$.

\end{proof}

We will need a slight variation on $r(n)$: 

\begin{defn}

Let $p(n) = \max_{m \leq n} r(m)$.

\end{defn}

We note that $p(n)$ grows at the same rate as $r(n)$, as a polynomial of bounded degree.  It has the technical advantage of being nondecreasing.

\begin{defn} \cite{kkm} Let $A$ be a graded $k$-algebra which is generated in degree 1.  Then $A$ has (finite) \emph{Gelfand-Kirillov dimension $d$} if the limit
$$\lim_{j \to \infty} \frac{\log(\rk_k(\sum_{n=0}^j A_n))}{\log(j)}$$
exists, and is equal to $d$. 

\end{defn}

If $V$ is a finite dimensional Yetter-Drinfeld module, it follows from Proposition \ref{sym_prop} that $R = \Sym_{br}(V)$ has Gelfand-Kirillov dimension 
$$\dim_{GK}(\Sym_{br}(V)) \leq d$$
where $d$ is the size of a minimal $G$-invariant generating set for $V$.

\subsection{Growth of tensor algebras} \label{growth_section} 

We will later need to control the {\em superdiagonal subalgebra} of a tensor algebra, which we now define.

\begin{defn} Let $k$ be a field, let $M$ be a graded $k$-vector space supported in nonnegative degrees, and $T M$ the tensor algebra of $M$ over $k$.  For each $m \geq 0$, write $(M^{\tensor m})_{\geq m}$ for the submodule of $M^{\tensor m}$ consisting of elements of degree at least $m$.  Then the {\em superdiagonal subalgebra} of $T M$ is $\oplus_m (M^{\tensor m})_{\geq m}$.
\label{def:superdiagonal}
\end{defn}

We note that if $M_*$ were supported in {\em positive} degree, $T M$ would be its own superdiagonal subalgebra, and this would make it easy to write down the upper bounds for the dimensions of graded pieces that we will later need.     In fact, it is not too hard to get such bounds even if $M$ has some degree-$0$ part, as we now show.

\begin{prop} \label{exp_growth_tensor_prop} Let $k$ be a field, let $M_*$ be a graded $k$-vector space supported in nonnegative degrees, and suppose there are constants $D,C \geq 1$ such that $\dim M_n < DC^n$.  Let $A$ be the superdiagonal subalgebra of $T M$.  Then the degree-$n$ part of $A$ has dimension at most $(4DC)^n.$
\end{prop}

\begin{proof}  By the definition of the superdiagonal subalgebra,  the degree-$n$ part of $A$ is the degree-$n$ part of $\oplus_{q = 0}^n M^{\tensor q}$.  Write $r_n$ for $\dim M_n$ and $R(t)$ for the generating function $\sum r_n t^n$.  Then the dimension of the degree-$n$ part of $M^{\tensor q}$  is the coefficient of $t^n$ in $R(t)^q$.  Our bounds on $r_n$ show that the coefficients of $R(t)$ are, term by term, less than or equal to those of $D(1-Ct)^{-1}$; so $\dim M^{\tensor q}_n$ is at most the coefficient of $t^n$ in $D^q (1-Ct)^{-q}$, which is easily seen to be
\beq
D^q C^n {n+q-1 \choose q-1} \leq D^q C^n 2^{n+q-1}.
\eeq
The sum of this quantity as $q$ ranges from $0$ to $n$ is bounded above by
\beq
D^n C^n 2^{2n-1}(1 + (2D)^{-1} + (2D)^{-2} + \ldots) \leq (4DC)^n
\eeq
which is the desired result.
\end{proof}

\subsection{Domination of objects in abelian categories} 

Finally, we introduce a notion that will be technically useful as we go.  If $C$ is an abelian category, we define a partial order on objects of $C$ called {\em domination} as follows. 

\begin{defn} If $M',M$ are objects of $C$, we say $M$ {\em dominates} $M'$ if $M'$ has a filtration whose associated graded is isomorphic to $M$, or if $M'$ is isomorphic to a subquotient of $M$, and we extend by transitivity. 
\label{def:dominate}
\end{defn}

In particular, if there is a spectral sequence any of whose pages is $M$ and which converges to $M'$, then $M$ dominates $M'$.  Our main interest is in the following setting:

\begin{exmp}

When $C$ is the category of graded (or multigraded) vector spaces over a field $k$, and $M$ and $M'$ are objects of $C$, both having finite-dimensional graded pieces, with $M$ dominating $M'$, then the dimension of a graded piece of $M$ is greater than equal to the dimension of the corresponding graded piece of $M'$.  

\end{exmp}

This notion will be useful to us in the argument because we will be chaining many spectral sequences together, and because our goal in the end will be to bound the dimensions of graded pieces of certain Ext modules at the end of these chains of spectral sequences.

\begin{rem} Of course, if $C$ is the category of graded $k$-vector spaces, the associated graded of a filtered object is isomorphic to the object itself, so to say $M$ dominates $M'$ is merely to say $M'$ is a subquotient of $M$.  We keep track of the more general notion because, in the context of the present paper, we will be proving domination relations between modules $M,M'$ for a ring $R$, and only at the end forgetfully thinking of these modules as graded $k$-vector spaces.  It seems possible that the fact that $M$ dominates $M'$ as an $R$-module might be useful for some other purpose, even though here we are only using it to get inequalities on $k$-dimensions of graded pieces, so we have chosen not to throw that information away.
\end{rem}

\subsection{Hypotheses on cohomological growth} \label{hyp_section}

In this section, we present two hypotheses on the growth rate of the cohomology of quantum shuffle and Nichols algebras, as well as a theorem about the relation between these hypotheses.  We also give examples where these hypotheses hold.  The results of section \ref{malle_section} will relate these hypotheses to Malle's conjecture in the function field setting.

Let us write $\Ext = \Ext_{\mA}^{*, *}(k, k)$, and
$$\Ext^j := \bigoplus_{n=j}^\infty \Ext_{\mA}^{n-j, n}(k, k)$$
The Yoneda multiplication on $\Ext$ makes $\Ext^j$ a (left) $R$-module.  If it is finitely generated over $R$, then its growth rate (with $n$) is asymptotic to a constant multiple of $p(n)$.  More generally, we may ask for that its growth rate be bounded by such a function of $n$.  As a function of $j$, we may ask that this bound be at worst exponential in $j$:

\begin{hyp}[Exponential growth] \label{exp_hyp} There exists a constant $C$ such that 
$$\dim_k \Ext_{\mA}^{n-j, n}(k, k) \leq p(n) C^j.$$

\end{hyp}

Using the spectral sequences in Section \ref{ss_sec}, we may get at the cohomology of the quantum shuffle algebra if we understand that of the Nichols algebra.  Restating the previous hypothesis for $\mB$ in place of $\mA$: 

\begin{hyp} \label{e_o_conj} There exists a constant $C_{\mB}$ such that
$$\dim_k \Ext_{\mB}^{n-j, n}(k, k) \leq p(n) C_{\mB}^j.$$
\end{hyp}

We believe that this hypothesis is valid for a very large family of Nichols algebras $\mB$, and will show in Corollary \ref{we_win_cor} that it holds in the cases relevant to Hurwitz spaces.  A conjecture of Etingof-Ostrik \cite{etingof-ostrik} states that the cohomology of a finite dimensional (braided) Hopf algebra is finitely generated.  If true, this would imply Hypothesis~\ref{e_o_conj} (with $C_{\mB} = 1+\epsilon$ for any $\epsilon>0$) in the case that $\mB$ is finite dimensional, since finitely generated braided-commutative algebras have polynomial growth.  

It is known that $\mB$ is finite dimensional for all Nichols algebras of Cartan type whose associated Cartan matrix is of finite type \cite{as_qg_cartan}.  Etingof-Ostrik's conjecture was recently verified for these Nichols algebras in \cite{mpsw}.  Additionally, in the rack setting, $\mB$ has been shown to be finite dimensional in several instances (e.g., the Fomin-Kirillov algebra $\mathcal{E}_n$ for $n<6$).  The cohomology of $\mathcal{E}_3$ has been explicitly computed in \cite{stefan-vay}; it is indeed finitely generated.

\begin{thm} \label{alg_bound_thm} Hypothesis \ref{e_o_conj} implies Hypothesis \ref{exp_hyp} with $C  = 4\max\{C_{\mB}, \dim V\}^2$. 

\end{thm}

\begin{proof}

Theorem 4 of \cite{may_restricted} gives a (tri-graded) spectral sequence in $\Ext$ for a filtered algebra:
$$\bigoplus_p \Ext_{\grA}^{n, p, q}(k, k) \implies \Ext^{n, q}_{\mA}(k, k)$$
We will suppress the filtration grading $p$, and write $\Ext_{\grA}^{n, q}(k, k)$ for the $E_2$ term of this spectral sequence, summing over $p$.  Proposition \ref{ss_prop} gives a change of rings spectral sequence 
$$\Ext_{\mB}^{n, q}(k, \Ext_{\mE}^{n', q'}(k, k)) \implies \Ext^{n+n', q+q'}_{\grA}(k, k).$$
Since $\mB$ is a connected algebra, we may obtain a spectral sequence for $\Ext_{\mB}(k, M)$ (for a $\mB$-module $M$) by filtering by degree.  In the case that $M = \Ext_{\mE}^{n', q'}(k, k)$, this is of the form
$$\Ext_{\mB}^{n, q}(k, k) \otimes \Ext_{\mE}^{n', q'}(k, k) \implies \Ext_{\mB}^{n, q}(k, \Ext_{\mE}^{n', q'}(k, k)) $$

In summary, these three spectral sequences imply that $\Ext^{*, *}_{\mA}(k, k)$ is dominated by $\Ext_{\mB}^{*, *}(k, k) \otimes \Ext_{\mE}^{*, *}(k, k)$.  In particular $\Ext_{\mA}^{n-j, n}(k, k)$ has dimension at most that of
$$\bigoplus_{a+b=n-j,} \bigoplus_{c+d=n} \Ext_{\mB}^{a, c}(k, k) \otimes \Ext_{\mE}^{b, d}(k, k).$$

Notice that $\grA$ and $\mB$ agree in degree $\leq 1$; thus $\mE$ is generated in degree 2 and higher.  Examining the bar complex for $\mE$, we see that $\Ext_{\mE}^{b, d}(k, k)$ can only be nonzero when $d\geq 2b$.  For $\mB$, we note additionally that only when $c \geq a$ is $\Ext_{\mB}^{a, c}(k, k)$ nonzero.  Therefore the only contribution to the sum above occurs when $b \leq j$:
$$b = n-j-a = (c+d)-j-a = (c-a) +d-j \geq d-j \geq 2b-j,$$
which gives $b \leq j$.  So our sum is in fact
$$\bigoplus_{b=0}^j \bigoplus_{c+d=n} \Ext_{\mB}^{n-j-b, c}(k, k) \otimes \Ext_{\mE}^{b, d}(k, k)$$

Since $c \geq a = n-j-b$, we can bound $d$ as well:
$$d = n-c \leq n-a= n-(n-j-b) = j+b \leq 2j.$$
This shows that $\Ext^{n-j, n}$ has dimension at most that of
$$\bigoplus_{b=0}^j \bigoplus_{d=0}^{2j} \Ext_{\mB}^{n-j-b, n-d}(k, k) \otimes \Ext_{\mE}^{b, d}(k, k)$$
Hypothesis \ref{e_o_conj} implies that the first tensor factor is bounded in rank by 
$$p(n-d) C_{\mB}^{j+b-d} \leq p(n) C_{\mB}^{j+b-d}.$$  

The second tensor factor is the finite dimensional $k$-vector space $\Ext_{\mE}^{b, d}(k, k)$, whose rank is bounded by that of the bar complex for $\mE$ and thus that of $\mA$.  In bidegree $(b, d)$, that rank is 
$$\dim B_b(k, \mA, k)[d] = \#\{\mbox{compositions of $d$ into $b$ parts} \} \cdot (\dim V)^d  = {d-1 \choose b-1} \cdot (\dim V)^d$$
Summing over $b\in [0, j]$ and $d \in [0, 2j]$, this gives
\begin{eqnarray*}
\dim \Ext_{\mA}^{n-j, n}(k,k) & \leq & \sum_{b=0}^j \sum_{d=0}^{2j} p(n) C_{\mB}^{j+b-d} {d-1 \choose b-1} \cdot (\dim V)^d \\
 & \leq & p(n) \sum_{b=0}^j \sum_{d=0}^{2j} {d-1 \choose b-1} (\max\{C_{\mB}, \dim V\})^{j+b}\\
 & \leq & p(n) (\max\{C_{\mB}, \dim V\})^{2j} \sum_{d=0}^{2j} 2^{d-1} \\
 & \leq & p(n) (4 \max\{C_{\mB}, \dim V\}^2)^{j};
\end{eqnarray*}
in the third line we use the fact that $b \leq j$.

\end{proof}

\section{Bounding the cohomology of a graded algebra} \label{nichols_coh_section}

Let $A$ be a connected, graded $k$-algebra.  Assume further that $A$ is finitely generated in degree 1, that is $Q(A) = A_1$ is finite dimensional.  The quadratic cover $\Ahat$ of $A$ is the algebra generated by $A_1$ subject only to the degree 2 relations in $A$.  Let $S = \Ahat^!$ be the quadratic dual algebra to $\Ahat$; see Definition \ref{quad_dual_defn}.  The purpose of this section is to prove the following:

\begin{thm} \label{main_thm}
The cohomology $\Ext_{A}^{*, *}(k,k)$ is dominated (Definition~\ref{def:dominate}) by the free $S$-module on the superdiagonal subalgebra (Definition~\ref{def:superdiagonal}) of $T (H^*(K(A)^*))$.
\end{thm}

Here $K(A)$ is the Koszul complex of $A$; see section \ref{koszul_section} for the definition.  The grading on $K(A)$ implicit in the statement of Theorem~\ref{main_thm} is the {\em topological degree} (Definition~\ref{def:topdeg}), in which the bigraded piece $K(A)_{p,q}$ defined in Definition~\ref{def:ka} has degree $p-1$.

This result gives us an exponential bound in the growth of the (co)homology of $A$ as an $S$-module in terms of the size of the homology of the Koszul complex.  By the definition of the superdiagonal subalgebra, it would be equivalent to say that $\Ext_{A}^{*, *}(k,k)$ is dominated by a direct sum of $S$-modules $M_m$, where $M_m$ is dominated by $S \otimes H^*(K(A)^*)^{\tensor m}$, and is supported in degrees greater than or equal to $m$. 

Theorem \ref{main_thm} is stated quite generally, but our main interest will be in the case that  $A$ is the Nichols algebra $A = \mB(V^*_{\epsilon})$.  In that case, Lemma \ref{components_lemma} and Proposition \ref{sym_br_prop} identify $S$ with the free braided commutative algebra $R = \Sym_{\br}(V)$.  In section \ref{bound_section} we will apply this by giving bounds on $H^*(K(A)^*)$ when $V = kc$ is a braided vector space of rack type; in order to bound the cohomological growth of Hurwitz spaces.

\subsection{Recollections on the cobar complex of a coalgebra}  \label{index_section}

Recall that the graded dual of a locally finite, graded algebra is a graded coalgebra.  In this section, we'll recall some of the basics for computing the $\Ext$ algebra for a graded algebra; much of this is naturally phrased in terms of the dual coalgebra.  

Let $(C, \Delta)$ be a coaugmented, non-negatively graded coalgebra, and write
$$C_{>0} := \cok(\eta: k \to C)$$
for the cokernel of the coaugmentation\footnote{If $C$ is a Hopf algebra, this is just the unit.}.  We will write $\Omega(C)$ for the \emph{cobar complex} of $C$.  This is a differential bigraded algebra; its underlying bigraded algebra is the tensor algebra
$$\Omega(C) = T(\Sigma^{(1, 0)} C_{>0} )$$
Here $\Sigma^{(1, 0)} C_{>0}$ is the bigraded vector space with
$$(\Sigma^{(1, 0)} C_{>0})_{p, q} = \left\{ \begin{array}{ll} 
C_q, & p=1, \, q>0 \\
0 & \mbox{else}. 
\end{array}\right.$$
This extends to a bigrading of the tensor algebra by summing the bigrading on each tensor factor.  Explicitly, in this bigrading, an element $c_1 \tensor \ldots \tensor c_m$ with $c_i$ lying in $C_{q_i}$ has bidegree $(m, \sum_{i=1}^m q_i)$; we refer to the first coordinate as the {\em tensorial degree} and then second coordinate as the {\em internal degree}.  It will in fact turn out that more important than either of these is the {\em difference} between the internal degree and the tensorial degree.

\begin{defn} The {\em topological degree} is the grading on $\Omega(C)$ which assigns $C_q$ the grade $q-1$; equivalently, the topological degree is the difference between the total internal degree and the tensorial degree.
\label{def:topdeg}
\end{defn}

We equip $\Omega(C)$ with a differential $d^{\Omega}$.  For $c \in C_{q} = (\Sigma^{(1, 0)} C_{>0})_{1, q}$, we may write 
\beqn \label{signs_d_eqn} \Delta(c) = \sum_{n=0}^{q} c_n' \otimes c_n'', \mbox{ where } c_n' \in C_n.  \;\;\; \mbox{ Then }\; d^{\Omega}(c) =  \sum_{n=1}^{q-1} (-1)^{n} c_n' \otimes c_n''.\eeqn
That is, $d^{\Omega}$ is a signed version of reduced coproduct $\overline{\Delta}$, defined on the quotient $C_{>0}$.  Then $d^{\Omega}$ is uniquely extended to $\Omega(C)$ by requiring it to be a graded derivation (with respect to total grading given by the difference of the indices): 
$$d^{\Omega}(x \otimes y) = d^{\Omega}(x) \otimes y + (-1)^{q-p} x \otimes d^{\Omega}(y), \; \mbox{ if } x \in \Omega(C)_{p, q} = (C_{>0}^{\otimes p})_q.$$
The vanishing of $d^{\Omega} \circ d^{\Omega}$ is equivalent to the coassociativity of $\Delta$.

If $A$ is a locally finite, graded algebra and we take $C=A^*$ to be its graded dual, this computes the cohomology of $A$:
$$H^{p, q}(\Omega(A^*)) \cong \Ext_{A}^{p,q}(k, k)$$
since $\Omega(A^*)_{p,q}$ is the internal degree $q$ part of $(A_{>0}^*)^{\otimes p} \cong \Hom(A_{>0}^{\otimes p}, k) = \Hom(B(A)_p, k)$.

\subsection{The Koszul complex} \label{koszul_section} 

We recall the assumptions of Theorem \ref{main_thm}: $A$ is a connected graded algebra, finitely generated in degree 1 by $W := A_1$.  Then there is a surjective map of graded algebras $f:TW \twoheadrightarrow A$.  Write $J = \ker(f)$; this is a homogenous ideal.  Let $J_2 = J \cap (W \otimes W)$ be the degree 2 component.

\begin{defn} \label{quad_dual_defn} The \emph{quadratic cover} of $A$ is
$$\Ahat := T(W) / (J_2)$$
where $(J_2)$ is the two sided ideal generated by $J_2$.  The \emph{quadratic dual} of $\Ahat$ is
$$S := T(W^*) / (J_2^{\perp})$$
where $J_2^\perp =\ker (W^* \otimes W^* = (W \otimes W)^* \to J_2^*)$.  \end{defn}

The ring $S$ is commonly notated as $\Ahat^!$; we use $S$ for notational simplicity.  It is well-known that $S_p \cong \Ext_{\Ahat}^{p,p}(k,k)$.  Indeed, $\Ext_{\Ahat}^{p,p}(k,k)$ is the cokernel of the differential $d^{\Omega}: \Omega^{p-1, p}(A^*) \to \Omega^{p, p}(A^*)$; as in (\ref{cobar_diff_eqn}), one can easily verify that the image of this map is the two-sided ideal generated by $J_2^\perp$.

In the special case that there exist surjections
\beqn \label{double_surj_eqn} \widehat{B(V^*_{\epsilon})} \twoheadrightarrow A \twoheadrightarrow B(V^*_{\epsilon}), \eeqn
(so $W = V^*_{\epsilon}$), Lemma \ref{components_lemma} allows us to identify this quadratic dual: 
$$S_p = \Ext_{\Ahat}^{p,p}(k,k) = \Ext_{A}^{p,p}(k,k) = R_p = \Sym_{\br}(V)_p.$$

Let $\{v_1, \dots, v_n\}$ be a basis for $W$.  We recall that the \emph{Koszul complex} for $A$ is the bigraded vector space $K(A)$ with $K(A)_{p, q} = A_p \otimes S_q^*$.  This becomes a chain complex using the differential
\beqn d(x \otimes \phi)= \sum_{i=1}^n xv_i \otimes \phi v_i^*. \label{diff_eqn} \eeqn
Here $\{v_1^*, \dots, v_n^*\}$ is the dual basis of $W^* =S_1$, and $S^* = \Hom_k(S, k)$ is a right $S$-module via $(\phi a)(b) = \phi(ab)$.  

Although true, it is not obvious that this definition is independent of choice of basis, nor that $d^2 = 0$.  Further, the Koszul complex is normally defined only for quadratic algebras such as $\Ahat$; nonetheless, we may describe $K(A)$ as presented above as 
$$K(A) = A \otimes_{\Ahat} K(\Ahat).$$
Much of our attention will be focused on the dual of this complex:

\begin{defn} The \emph{dual Koszul complex} of $A$ is $K(A)^*_{p, q} := A^*_p \otimes S_q$, with differential
$$d(\phi \otimes r) = \sum_{i=1}^n \partial^L_{v_i}(\phi) \otimes v_i^* r.$$
The first term is the action of $A$ on $A^*$ through left derivatives described in Definition \ref{deriv_defn}; the second term is left multiplication in $S$.

Further, define $K_2(A)^*$ to be the \emph{two-sided dual Koszul complex}:
$$K_2(A)^*_{m, p, q} := S_m \otimes A_p^* \otimes S_q.$$ 
This is equipped with the differential $d = d^R + (-1)^p d^L$, where 
$$d^R(r \otimes \phi \otimes s) = \sum_{i=1}^n r v_i^* \otimes \partial_{v_i}^R(\phi) \otimes s, \; \; \mbox{ and } \; \;d^L(r \otimes \phi \otimes s) = \sum_{i=1}^n r \otimes \partial_{v_i}^L(\phi) \otimes v_i^* s.$$

\label{def:ka}
\end{defn}

Since $d^L$ is the usual differential on $S \otimes K(A)^*$, it's easy to see that $(d^L)^2=0$; similarly $(d^R)^2 = 0$.  Further, $d^L$ and $d^R$ commute, since the action of the left and right partial derivatives commute.  Thus $d^2=0$, being a bicomplex differential.

\subsection{Filtering the cobar complex}

Write $\Omega$ for the cobar complex $\Omega A^* = (T\Sigma^{1, 0} A^*_{>0}, d^{\Omega})$.  There is surjective map of differential graded algebras $\pi: \Omega \to S$ induced by the projection $A^*_{>0} \to A_1^* = W^*$ and imposition of the relations in $S = T(W^*)/(J_2^{\perp})$.  Because 
$$(J_2^\perp) = \im (d^{\Omega}) \cap T(W^*),$$
it is clear that this map is indeed a chain map, where $S$ is equipped with the zero differential.

Define $I := \ker(\pi)$; this is a dg ideal in $\Omega$.  Equip $\Omega$ with the ($I$-adic) filtration by powers of $I$; the associated graded algebra is
$$\Omega^{gr} = \bigoplus_{n=0}^\infty I^n/I^{n+1}.$$
This gives rise to a spectral sequence
\beqn \label{omega_ss_eqn} \bigoplus_{n=0}^{\infty} H^*(I^n/I^{n+1}) = H^* \Omega^{gr} \implies H^* \Omega = \Ext_A(k,k) \eeqn
which we will use to prove Theorem \ref{main_thm}.

We note that the $\Omega$-module structure on $I^n$ descends to an $S = \Omega/I$-module structure on $I^n/I^{n+1}$.  The following is our main tool in understanding $\Omega^{gr}$:

\begin{prop} \label{tensor_prop}

The $n$-fold multiplication map
$$(I/I^2)^{\otimes_S n} \to I^n/I^{n+1}$$
is an isomorphism of chain complexes.

\end{prop}

We need a lemma:

\begin{lem}

The multiplication map
$$\mu: I^n \otimes_{\Omega} I \to I^{n+1}$$
is an isomorphism of chain complexes.

\end{lem}

\begin{proof}

The map is definitionally surjective, and is a map of chain complexes precisely because the differential in $\Omega$ satisfies the Leibniz rule.  So it suffices to show that the map is injective.  To see this, we may ignore the differentials on $I$ and $\Omega$, and just regard them as modules over the underlying ring of $\Omega$, $T(A^*_{>0})$.

There is a short exact sequence of $\Omega$-modules
$$0 \to I \to \Omega \to S \to 0.$$
This induces a long exact sequence in $\Tor_*^{\Omega}(I^n, -)$:
$$0 \to \Tor_1^{\Omega}(I^n, S) \to I^n \otimes_{\Omega} I \to I^n  \to I^n/I^{n+1} \to 0.$$
Here we've identified $I^n \otimes_{\Omega} \Omega$ with $I^n$ and $I^n \otimes_{\Omega} S$ with $I^n/I^{n+1}$.  So it's enough to show $\Tor_1^{\Omega}(I^n, S) = 0$; then $\mu$ will carry $I^n \otimes_{\Omega} I$ injectively into $I^n$, with image $I^{n+1}$.

We recall from \cite{cohn} that $\Omega = TA^*_{>0}$ is a \emph{free ideal ring}: all of its right ideals are free as right $\Omega$-modules.  In particular, this is true for $I^n$; thus $\Tor_1^{\Omega}(I^n, S) = 0$.

\end{proof}

\begin{proof}[Proof of Proposition \ref{tensor_prop}]

As in the previous Lemma, the map is one of chain complexes, and definitionally surjective; we only need to prove that it is injective.  We may do so by proving injectivity on the underlying modules.

Induct on $n$; the claim is tautological in the base case $n=1$.  Assume that it holds for $n-1$.  Consider the following commuting diagram
$$\xymatrix{
 & I^n \otimes_{\Omega} I \ar[r]^-{i \otimes 1} \ar[d]_-{\mu} & I^{n-1} \otimes_{\Omega} I \ar[r] \ar[d]_-{\mu} & (I^{n-1}/I^n) \otimes_{\Omega} I \ar[r] \ar[d]_-{\mu} & 0 \\
0 \ar[r] & I^{n+1} \ar[r]^-{i} & I^{n} \ar[r] & I^{n}/I^{n+1} \ar[r] & 0 \\
}$$
Here $i$ denotes inclusion of higher powers of $I$ into lower powers.  The bottom row is tautologically exact.  The upper row is exact since $- \otimes_{\Omega} I$ is right exact, but it's not clear that $i \otimes 1$ in the upper left is injective.  However, the previous Lemma implies that the two left vertical maps $\mu$ are isomorphisms.  Thus since the lower $i$ is injective, so is $i \otimes 1$, and the upper sequence is in fact short exact.  We conclude that the rightmost vertical map $\mu$ is also an isomorphism:
$$\mu: (I^{n-1}/I^n) \otimes_{\Omega} I \cong I^n/I^{n+1}$$

Notice that the image of $(I^{n-1}/I^n) \otimes_{\Omega} I^2$ in $(I^{n-1}/I^n) \otimes_{\Omega} I$ is zero; these are sums of classes of the form $x \otimes yz$ where $y$ and $z$ are in $I$ and $x$ is represented by a class in $I^{n-1}$.  But
$$x \otimes yz = xy \otimes z = 0$$
since $xy$ represents a class in $I^n$.  Therefore
$$(I^{n-1}/I^n) \otimes_{\Omega} I = (I^{n-1}/I^n) \otimes_{\Omega} (I/I^2)$$
But now the $\Omega$-module structure on both $I^{n-1}/I^n$ and $I/I^2$ factors through $S = \Omega/I$, so this tensor product is equal to $(I^{n-1}/I^n) \otimes_{S} (I/I^2)$.  By induction, this is $(I/I^2)^{\otimes_S n}$.

\end{proof}

\begin{cor} For every $n$, $I^n/I^{n+1}$ is free as a right $S$-module. \end{cor}

\begin{proof} 

By the previous result, it suffices to show this for $n=1$.  Since $I$ is free as a right $\Omega$-module, say on a set $E$, then 
$$I/I^2 = I \otimes_{\Omega} \Omega/I = E \cdot \Omega \otimes_{\Omega} S = E \cdot S$$
is free as an $S$-module on the same set.

\end{proof}

\subsection{Relating $K_2(A)^*$ and $I/I^2$}

In this section, we study the chain complex $I/I^2$, and show that it is quasi-isomorphic to a split subcomplex of $K_2(A)^*$.

We first recall that for a chain complex $C$ of $S$-modules and integer $\ell$, the \emph{truncation} $\tau_{>\ell} C$ is the subcomplex
$$(\tau_{>\ell} C)_p = \left\{\begin{array}{ll} C_p, & p>\ell+1, \\ \ker(d_{\ell+1}: C_{\ell+1} \to C_\ell), & p=\ell+1 \\ 0 & p\leq \ell. \end{array} \right.$$ 
This is functorial in $C$ and has the property that the inclusion $\tau_{>\ell} C \hookrightarrow C$ is an isomorphism in $H_p$ in degrees $p>\ell$, and $H_p(\tau_{>\ell} C) = 0$ if $p\leq \ell$.

\begin{prop} \label{trunc_split_prop}

If $\im(d_{\ell+1}) \leq C_\ell$ is a projective $S$-module, then there is a subcomplex $(\tau_{>\ell}C)^{\perp} \leq C$ such that the sum of inclusions
$$\tau_{>\ell} C \oplus (\tau_{>\ell}C)^{\perp} \to C$$
is an isomorphism of complexes.

\end{prop}

\begin{proof}

Take $(\tau_{>\ell}C)^{\perp}_p$ to be $C_p$ for $p\leq \ell$ and $0$ for $p>\ell+1$.  Let $(\tau_{>\ell}C)^{\perp}_{\ell+1} \leq C_{\ell+1}$ be the image of an $S$-module map $\im(d_{\ell+1}) \to C_{\ell+1}$ splitting $d_{\ell+1}$.  It is easy to see that this construction has the desired properties.

\end{proof}

Let us write $I/I^2$ as a singly-graded chain complex where the degree is topological degree.  The differential on $I/I^2$ (induced from $\Omega$) fixes the internal degree, and increases the tensorial degree by one, so we may write the differential homologically: $d_j: (I/I^2)_j \to (I/I^2)_{j-1}$.

\begin{lem} \label{surj_lem}

The differential $d_1: (I/I^2)_1 \to (I/I^2)_{0}$ is surjective.

\end{lem}

\begin{proof}

If $I_0$ denotes the degree zero (in the above sense) component of $I$, then $(I/I^2)_0 = I_0/I_0^2$.  The degree zero part -- where the internal degree equals the number of tensor factors -- of $\Omega$ is $T(W^*)$. Then $I_0$ is the two-sided ideal in $T(W^*)$ generated by $J_2^{\perp}$.  However, $J_2^{\perp} = d^{\Omega}(A_2^*) = d^{\Omega}(I^{1, 2})$.  Thus, $d_1: I_1 \to I_0$ is surjective, so the same holds for $d_1: (I/I^2)_1 \to (I/I^2)_{0}$.

\end{proof}

\begin{cor} \label{cor_6}

For every $n>0$, the map 
$$(\tau_{>0}(I/I^2))^{\otimes_S n} \to (I/I^2)^{\otimes_S n} \cong I^n/I^{n+1}$$
is a split quasi-isomorphism.

\end{cor}

\begin{proof}

We first establish the case $n=1$ which concerns the inclusion $\tau_{>0}(I/I^2) \hookrightarrow I/I^2$.  Lemma \ref{surj_lem} implies that $H_0(I/I^2) = 0$, so this is an isomorphism in homology in all degrees.

We note that the $S$-module structure on $I/I^2$ is graded, restricting to $S$-module structures on $(I/I^2)_d$ for every degree $d$.  Since $I/I^2$ is a free $S$-module, the same is true for $(I/I^2)_d$.  By Lemma \ref{surj_lem}, $\im(d_1) = (I/I^2)_{0}$ is free, so Proposition \ref{trunc_split_prop} gives the desired splitting  
$$\tau_{>0}(I/I^2) \oplus \tau_{>0}(I/I^2)^{\perp} \cong I/I^2.$$
Further, the vanishing of $H_0$ implies that $ \tau_{>0}(I/I^2)^{\perp}$ is contractible.

Tensoring this isomorphism with itself $n$ times over $S$ shows that $(I/I^2)^{\otimes_S n}$ is isomorphic to the sum of $(\tau_{>0}(I/I^2))^{\otimes_S n}$ with $2^n-1$ other terms.  These other terms all include a tensor factor of $\tau_{>0}(I/I^2)^{\perp}$ and are therefore contractible.

\end{proof}

We now address the relationship between $I/I^2$ and $K_2(A)^*$.  Note, first of all, that $K_2(A)^*$ carries a natural notion of internal degree and tensorial degree; namely, $K_2(A)^*_{m,p,q} = S_m \otimes A_p^* \otimes S_q$ has tensorial degree $m+1+q$ and internal degree $m+p+q$, whence its topological degree $p-1$.  It is topological degree we have in mind when referring to the ``degree" of an element of $K_2(A)^*$.

\begin{lem}

There is a chain map $j:I/I^2 \to K_2(A)^*$ which is an isomorphism in positive degrees and injective in degree 0.

\end{lem}

\begin{proof}

We may write $I = I_0 \oplus I_{>0}$, where $I_{>0}$ is the subspace of degree greater than 0.  Then
\begin{eqnarray*}
I/I^2 & = & (I_0 \oplus I_{>0})/(I_0^2 \oplus (I_0 I_{>0} + I_{>0}  I_0 +  I_{>0}  I_{>0})) \\
 & = & (I_0/I_0^2) \oplus I_{>0}/(I_0 I_{>0} + I_{>0}  I_0 +  I_{>0}  I_{>0}).
\end{eqnarray*}
Further, we may compute the second term in two steps:
$$I_{>0}/(I_0 I_{>0} + I_{>0}  I_0 +  I_{>0}  I_{>0}) = (I_{>0}/I_{>0}^2)/(I_0 I_{>0} + I_{>0}  I_0).$$

We may identify $(I_{>0}/I_{>0}^2)$ with the subspace of $\Omega = T(A^*_{>0})$ consisting of (sums of) tensors with precisely one term in $A^*$ of degree greater than 1:
$$(I_{>0}/I_{>0}^2) = TW^* \otimes A^*_{>1} \otimes TW^* \leq T(A^*_{>0})$$ 
Now quotienting by $I_0 I_{>0} + I_{>0}  I_0$ imposes the defining relations in $S = TW^*/I_0$, yielding an isomorphism of $S$-bimodues
$$(I_{>0}/I_{>0}^2)/(I_0 I_{>0} + I_{>0}  I_0) \cong S \otimes A^*_{>1} \otimes S.$$
Keeping track of degrees, we see that 
$$(I/I^2)_0 = I_0/I_0^2 \;\; \mbox{ and } \;\; (I/I^2)_m = S \otimes A^*_{m+1} \otimes S\, \mbox{ if $m>0$.}$$

The differential on $I/I^2$ is induced from that of $\Omega$.  We note that $d=0$ on terms of degree zero, and that it is a derivation.  Thus for $m>1$,
$$d_m: (I/I^2)_m = S \otimes A^*_{m+1} \otimes S \to S \otimes A^*_{m} \otimes S = (I/I^2)_{m-1}$$
is determined by its value on $A^*_{m+1}$, and extended $S$-bilinearly.  On the representing term in $\Omega^{1, m+1} = A^*_{m+1}$,
$$d: \Omega^{1, m+1} = A^*_{m+1} \to \bigoplus_{i+j = m+1} A^*_i \otimes A^*_j = \Omega^{2, m+1}$$
is computed by a signed version of the reduced diagonal $\overline{\Delta}$ on $A^*_{>0}$, as given in (\ref{signs_d_eqn}).  In the quotient by $I^2$, all of the terms $A^*_i \otimes A^*_j$ with either $i>1$ or $j>1$ vanish.  Thus in $I/I^2$, $d_m: A^*_{m+1} \to S \otimes A^*_{m} \otimes S$ is given by 
$$d_m(\phi) = - \sum_{i=1}^n v_i^* \otimes \partial^R_{v_i}(\phi) \otimes 1 + (-1)^m \sum_{i=1}^n 1 \otimes \partial^L_{v_i}(\phi) \otimes v_i^*$$
by virtue of Lemma \ref{diagonal_lem}, truncated to bidegrees $(1, m)$ and $(m, 1)$ in $A^* \otimes A^*$. 

This is precisely the negative of the differential in $K_2^*(A)$, so we conclude that in degrees $n>0$ there is an isomorphism of chain complexes $(I/I^2)_{>0} \cong K_2^*(A)_{>0}$.  We will extend this to a chain map of the full chain complexes $j: I/I^2 \to K_2(A)^*$:
\beqn \label{chain_eqn} \xymatrix{\cdots \ar[r]^-{d_4} & (I/I^2)_3 \ar[d]^-{j_3}_-{\cong} \ar[r]^-{d_3} & (I/I^2)_2 \ar[d]^-{j_2}_-{\cong} \ar[r]^-{d_2} & (I/I^2)_1 \ar[d]^-{j_1}_-{\cong} \ar[r]^-{d_1} & (I/I^2)_0 \ar[d]^-{j_0} \ar[r]^-{d_0} & 0 \ar[d]^-{j_{-1}}_-{0} \\
\cdots \ar[r]^-{d_4} & K_2(A)^*_3 \ar[r]^-{d_3} & K_2(A)^*_2 \ar[r]^-{d_2} & K_2(A)^*_1 \ar[r]^-{d_1} & K_2(A)^*_0 \ar[r]^-{d_0} & K_2(A)^*_{-1}}\eeqn
Here the vertical maps $j_m$ with $m>0$ are the (chain) isomorphisms that we have constructed above; it remains to define the map $j_0$ in such a way to make the two right squares commute.

We note that $K_2(A)^*_0 = S \otimes W^* \otimes S$.  To define $j_0$, we first define auxiliary maps $s_t: (W^*)^{\otimes m} \to (S \otimes W^* \otimes S)_m$ for $1 \leq t \leq m$ by
$$s_t(\phi_1 \otimes \cdots \otimes \phi_m) = [\phi_1 \cdots \phi_{t-1}] \otimes \phi_t \otimes [\phi_{t+1} \cdots \phi_m]$$
(we interpret the empty product to be $1 \in S$).  Define $j_0: TW^* \to S \otimes W^* \otimes S$ in degree $m$ to be $j_0 = \sum_{t=1}^{m} s_t$.

We claim that the restriction of each $s_t$ to $I_0^2$ is zero: $(I_0^2)_m$ is additively generated by elements of the form 
$$\phi_1 \otimes \cdots \otimes \phi_{i-1} \otimes \Phi_{i,i+1} \otimes \phi_{i+2} \otimes \cdots \otimes \phi_{h-1} \otimes \Phi_{h, h+1} \otimes \phi_{h+2} \otimes \cdots \otimes \phi_{m}$$
where $\Phi_{i, i+1}$ and $\Phi_{h, h+1}$ lie in $J_2^{\perp} \leq W^* \otimes W^*$.  If $i+1 \leq t-1$ or $i \geq t+1$, then $s_t$ kills this element, since $\Phi_{i, i+1}$ is sent to $0 \in S$.  In the intermediate cases $i = t-1, t$ the fact that $h \geq i+2 \geq t+1$ implies that $s_t$ kills the element, since $\Phi_{h, h+1}$ is sent to $0 \in S$.

Restricting to $I_0 \leq TW^*$, then $j_0$ descends to a well-defined map $j_0: I_0/I_0^2 \to S \otimes W^* \otimes S$.  Classes in $(I_0/I_0^2)_m$ are represented by sums of elements of $TW^*$ of the form
$$x:= \phi_1 \otimes \cdots \otimes \phi_{i-1} \otimes \Phi_{i, i+1} \otimes \phi_{i+2} \otimes \cdots \otimes \phi_m$$
where $\Phi_{i, i+1} = \sum_j \phi_i^j \otimes \phi_{i+1}^j \in J_2^{\perp}$.  Then
$$j_0(x) = \sum_j \left([\phi_1 \cdots \phi_{i-1} \phi_i^j] \otimes \phi_{i+1}^j \otimes [\phi_{i+2} \cdots \phi_m] + [\phi_1 \cdots \phi_{i-1}] \otimes \phi_{i}^j \otimes [\phi_{i+1}^j \phi_{i+2} \cdots \phi_m]\right).$$
From this it requires two very short computations to see that this definition of $j_0$ does indeed yield a chain map as in (\ref{chain_eqn}); as previously, this relies on $\Phi_{i, i+1}$ lying in $J_2^{\perp}$.

Finally, we show that $j_0 = \sum s_t$ is injective on $I_0/I_0^2$; equivalently, the kernel of $j_0$ on $TW^*$ is $I_0^2$.  Note that each term $s_t: (W^*)^{\otimes m} \to S_{t-1} \otimes W^* \otimes S_{m-t}$ takes value in a different summand of $(S \otimes W^* \otimes S)_m$, so 
$$\ker(j_0) = \bigcap_{t=1}^m \ker(s_t)$$
Further, $\ker(s_t) = (I_0)_{t-1} \otimes (W^*)^{\otimes m-t+1} +(W^*)^{\otimes t} \otimes (I_0)_{m-t}$.  The first term vanishes if $t=1,2$, and the second if $t=m-1, m$.

Let $x \in \ker(j_0)$.  Then by the previous, $x \in \ker(s_2) = (W^*)^{\otimes 2} \otimes (I_0)_{n-2}$.  Since $x$ is also in $\ker(s_3)$, we may write $x = x'_3 + x''_3$, where $x'_3 \in (W^*)^3 \otimes (I_0)_{n-3}$, and
$$x_3'' \in (W^*)^{\otimes 2} \otimes (I_0)_{n-2} \cap (I_0)_2 \otimes (W^*)^{\otimes n-2}  \subseteq I_0^2.$$
Repeating this argument with $x_3'$ in place of $x$, one may show that $x_3' = x_4'+x_4'',$ where $x_4'' \in I_0^2$ and $x_4' \in (W^*)^4 \otimes (I_0)_{n-4}$.  Inducting in this fashion (and using the fact that $(W^*)^{m-1} \otimes (I_0)_{1} = 0$) we see that $x \in I_0^2$.

\end{proof}

\begin{thm} \label{prop9}

The map
$$\tau_{>0} j: \tau_{>0}(I/I^2) \to \tau_{>0} K_2(A)^*$$
is an isomorphism of chain complexes, and induces a splitting
$$\tau_{>0}(I/I^2) \oplus (\tau_{>0} K_2(A)^*)^{\perp} \cong K_2(A)^*.$$

\end{thm}

\begin{proof}

The previous result clearly shows that $\tau_{>0} j$ is an isomorphism in degrees greater than 1.  It is also an isomorphism in degree 1, since injectivity of $j_0$ ensures that $j_1$ restricts to an isomorphism between $\ker(d_1)$ in the two complexes.  The map is vacuously an isomorphism in non-positive degrees.

Finally, since $j_0$ is injective and $j_1$ is an isomorphism, $j_0$ restricts to an isomorphism between $\im(d_1)$ in the two complexes.  Since $\im(d_1)$ is projective in $I/I^2$, the same holds in $K_2(A)^*$.  So by Proposition \ref{trunc_split_prop}, $K_2(A)^*$ splits as $\tau_{>0}K_2(A)^* \oplus (\tau_{>0} K_2(A)^*)^{\perp}$.

\end{proof} 

\begin{cor} There is a quasi-isomorphism of differential graded $S$-algebras 
$$T_S(\tau_{>0} K_2(A)^*) \to \Omega^{\gr}.$$
\end{cor}

This follows from the previous result and Corollary \ref{cor_6}.

\subsection{The proof of Theorem \ref{main_thm}}

\label{proofmainthm} 

With the previous work in hand, proving Theorem \ref{main_thm} is straightforward.  The spectral sequence (\ref{omega_ss_eqn}) shows that $\Ext_A(k, k)$ is dominated by
$$\bigoplus_{n \geq 0} H_*(I^n/I^{n+1})$$
so it suffices to show that $H_*(I^n/I^{n+1})$ is in turn dominated by the free left $S$-module on $(H^*(K(A)^*))^{\otimes n}$.  Since $ H_*(I^n/I^{n+1})$ vanishes in degrees below $n$, it is therefore dominated by the free left $S$-module on the superdiagonal subalgebra of $T(H^*(K(A)^*))$ in tensorial degree $n$; this will give Theorem \ref{main_thm}.

We now prove that the term $H_*(I^n/I^{n+1})$ is dominated by the free left $S$-module on $H^*(K(A)^*)^{\otimes n}$.   First, Corollary \ref{cor_6} gives 
$$H_*(I^n/I^{n+1}) \cong H_*(\tau_{>0}(I/I^2)^{\otimes_S n}).$$
Then Theorem \ref{prop9} implies that $\tau_{>0}(I/I^2)^{\otimes_S n}$ is a split subcomplex of $(K_2(A)^*)^{\otimes_S n}$, so $H^*(I^n/I^{n+1})$ is a summand of $H^*((K_2(A)^*)^{\otimes_S n})$. Furthermore, $\tau_{>0}(I/I^2)$ is zero in degrees less than $1$, so $\tau_{>0}(I/I^2)^{\otimes_S n}$ (and hence its homology) vanishes in degrees less than $n$.  

We may filter $K_2(A)^*$ with respect to the index $m$ on the first tensor factor of $S$.  In the associated graded chain complex $K_2(A)^{* \gr}$, the differential $d^R$ is suppressed since it increases $m$.  So 
$$K_2(A)^{* \gr} = (K_2(A)^*, d^L).$$
However, since the first factor of $S$ does not participate in the differential, we see that there is an isomorphism of chain complexes $K_2(A)^{* \gr} \cong S \otimes K(A)^*$.  Here the differential on the target is the $S$-linear extension of the differential on $K(A)^*$.

Consider the the $n$-fold tensor power 
\beqn \label{big_tensor_eqn} (K_2(A)^*)^{\otimes_S n} = \bigoplus_{m_i, p_j} S_{m_0} \otimes (A^*_{p_1} \otimes S_{m_1}) \otimes \cdots \otimes (A^*_{p_n} \otimes S_{m_n})\eeqn
This is equipped with the tensor power differential $d = \sum_{i=1}^n \pm d_i^R \pm d_i^L$; here $d_i^R$ increases $m_{i-1}$ by 1 and decreases $p_i$ by 1, whereas $d_i^L$ decreases $p_i$ by 1 and increases $m_{i}$ by 1.

Equip $(K_2(A)^*)^{\otimes_S n}$ with a decreasing filtration $\{F_k\}$, where for an integer $k$, $F_k$ is the sum of the homogenous components of (\ref{big_tensor_eqn}) with
$$nm_0 + (n-1)(p_1+m_1) + (n-2)(p_2+m_2) + \cdots + (p_{n-1} + m_{n-1}) \geq k$$
A short computation shows that each $d_i^L$ exactly preserves the filtration, whereas $d_i^R$ increases it by 1.  Consequently $F_k$ is indeed a filtration by subcomplexes.  Further, in the associated graded, each $d_i^R$ is suppressed, so it is easy to see that there is an isomorphism of chain complexes 
$$((K_2(A)^*)^{\otimes_S n})^{\gr} \cong S \otimes (K(A)^*)^{\otimes n}.$$
The resulting spectral sequence shows that $H_*(I^n/I^{n+1})$ is dominated by the homology of the right-hand side, which is the free left $S$-module on $H^*(K(A)^*)^{\otimes n}$, as desired.

\section{Bounding the cohomology of Nichols algebras} \label{bound_section}

The purpose of this section is to establish tools to bound the growth of $\Ext_{\mB}(k, k)$ as an $R$-module, specifically to prove Hypothesis \ref{e_o_conj} for certain classes of Nichols algebras of rack type.  Theorem \ref{main_thm} is our main tool; to apply it, we need to bound the growth rate of the homology of the Koszul complex $K(\mB)^*$.

In section \ref{koszul_nichols_section} we study the Koszul complex for Nichols algebras in general, and in sections \ref{koszul_rack_section} and \ref{mult_koszul_section} we narrow our focus to Nichols algebras coming from racks, the case relevant to the homology of Hurwitz spaces.  Using a variant on the Conway-Parker/Fried-V\"{o}lklein theorem \cite{fried-volklein}, we show that the Koszul complex for these Nichols algebras do grow at worst exponentially, thereby proving Hypothesis \ref{e_o_conj} for these Nichols algebras.

Because we will never consider the two-sided Koszul complex in this section, we will only ever use the left braided derivatives $\partial^L_v$ which occur in $K(A)^*$ (and not $\partial^R_v$).  This allows us to use the shorthand $\partial_v := \partial^L_v$ for the entirety of this section.

\subsection{The dual Koszul complex for Nichols algebras} \label{koszul_nichols_section}

In this section, we study the Koszul complex and its dual for the Nichols algebra $\mB(V^*_\epsilon)$.  Choose a basis $\{v_1, \dots, v_n\}$ of $V_{\epsilon}$ and let $\{v_1^*, \dots, v_n^*\}$ be the dual basis. Then we may write 
$$K(\mB(V^*_\epsilon))^* = \mB(V^*_\epsilon)^* \otimes R \cong  \mB(V_\epsilon) \otimes R, \;\mbox{ with } \; d(\phi \otimes r) = \sum \partial_{v_i^*}(\phi) \otimes v_i r$$
Notice that this is a right\footnote{Here we recall that the ring $S$ of the previous section is $R = \Sym_{\br}(V)$ in the case that $A=\mB(V^*_{\epsilon})$.} $R$-module, and this action commutes with the differential.  This allows us to define a Koszul complex with coefficients in more general $R$-modules, as follows:

\begin{defn} \label{k_comp_defn}

Let $M$ be a (graded) left $R$-module, and define a bigraded chain complex 
$$\KK(M)^{p, q} := (K(\mB(V^*_\epsilon))^* \otimes_R M)^{p, q} = \mB(V_\epsilon)_p \otimes M_q.$$
Note that the differential is of the form $d: \KK(M)^{p, q} \to \KK(M)^{p-1, q+1}$.

\end{defn}

With this definition, the Koszul complex $K(\mB(V^*_\epsilon))^*$ is now written as $\KK(R)$.  While our primary concern is with $\KK(R)$ itself, monodromy issues require us to control along the way the cohomology of the complex $\KK(R^{(H)})$, where $R^{(H)}$ is an $R$-module corresponding to a proper subgroup $H$ of $G$, defined in section \ref{koszul_rack_section}.

\begin{rem}

It is worth pointing out (although we will not use this fact) that one can in fact form a variant on this complex with $\mA(V^*_\epsilon)$ in place of $\mB(V^*_\epsilon)$, since $R$ is the diagonal cohomology of both of these algebras.  The resulting complex is precisely the ``Koszul-like" complex of section 4 of \cite{evw} which was constructed from the arc complex for braid groups.  One of the main technical advantages of the approach in this paper is that the complex based upon $\mB(V^*_\epsilon)$ is substantially smaller and easier to compute than the one based on $\mA(V^*_\epsilon)$.  In particular, the machinery introduced in section 6 requires the Hopf algebra $A$ to be generated in degree 1, which is true of $\mB(V^*_\epsilon)$ but not of $\mA(V^*_\epsilon)$.

\end{rem}

Since $\KK(R)$ is a complex of free right $R$-modules, for any short exact sequence
$$0 \to M' \to M \to M'' \to 0$$
of $R$-modules, there is a short exact sequence of chain complexes
$$0 \to \KK(M') \to \KK(M) \to \KK(M'') \to 0.$$
This immediately gives us:

\begin{prop} \label{K_les_prop}

There is a long exact sequence
$$\cdots \to  H^{p, q}(\KK(M')) \to H^{p, q}(\KK(M)) \to H^{p, q}(\KK(M'')) \to H^{p-1, q+1}(\KK(M')) \to \cdots$$

\end{prop}

\subsection{The monodromy filtration on Nichols algebras of rack type} \label{koszul_rack_section}

We will now focus on Nichols algebras $\mB(V^*_\epsilon)$ in the case where $V = kc$ is the Yetter-Drinfeld module associated to a conjugation-invariant subset $c \subseteq G$ (i.e., a union of conjugacy classes), as in section \ref{rack_section}.  Notice that $R_n$ may be identified with the free $k$-vector space on the set of orbits of $B_n$ on $c^{\times n}$ under the Hurwitz action.  For a Hurwitz word $w \in R_n$ represented by $(w_1, \dots, w_n) \in c^{\times n}$, there is a well-defined subgroup $H := \langle w_1, \dots, w_n \rangle \leq G$ generated by $w$; we will call $H$ the \emph{monodromy group} associated to $w$. 

We will insist that $c$ generate the group $G$.  Let us write $\Sub_{G, c}$ for the lattice of subgroups $H \leq G$ which may be generated from $c$; that is, $H$ lies in $\Sub_{G, c}$ if and only if $H = \langle c \cap H \rangle$.

For $H \in \Sub_{G, c}$, define $F^{\geq H} R$ to be the span of all Hurwitz words whose monodromy group contains $H$.  Note that this is a two-sided ideal in $R$, and that if $H \leq K$, then $F^{\geq H} R \supseteq F^{\geq K} R$.  Define an $R$-module

$$R^{(H)} := \left. F^{\geq H} R \middle/ \left( \sum_{K > H} F^{\geq K} R \right)\right. .$$

The filtration $F$ on the ring of coefficients $R$ induces a filtration on $\KK(R)$, whose associated $H$-graded piece is precisely $\KK(R^{(H)})$.  Packaging this filtration together with Proposition \ref{K_les_prop}, we have:

\begin{cor} \label{K_ss_cor}

There is a strongly convergent spectral sequence
$$E_1^{p, q} := \bigoplus_{H \in \Sub_{G, c}} H^{p, q}(\KK(R^{(H)})) \implies H^{p, q}(\KK(R))$$
with first differential $d_1$ a sum over minimal $K > H$ of the maps 
$$d_1^{H < K}: H^{p, q}(\KK(R^{(H)})) \to H^{p-1, q+1}(\KK(R^{(K)}))$$
induced by the connecting homomorphism in Proposition \ref{K_les_prop} coming from the extension of $R^{(H)}$ by $R^{(K)}$ implicit in their definition.

\end{cor}

\subsection{Multiple Koszul differentials} \label{mult_koszul_section} 

In this section, we show that the Koszul complex attached to a Nichols algebra of rack type actually carries multiple differentials, and that in particular the graded pieces in the filtration on $\KK(R)$ introduced in the previous section thereby acquire the structure of a multiple complex.  Notice, first, that $c$ serves as a natural basis for $V_\epsilon = kc \otimes \epsilon$.

Let $S$ be a subrack of $c$; that is, $S$ is a subset of $c$ closed under the conjugacy operation.  Let $M$ be a graded left $R$-module and write $d_S$ for the operator on $\KK(M)$ given by 
\beq
d_S(\psi \otimes m) = \sum_{w \in S} \partial_{w^*}(\psi) \otimes w m
\eeq
Note that when $S=c$, the operator $d_S$ is just the differential $d$ of $\KK(M)$.  In fact, as we now show, $d_S$ is a differential for any subrack $S$.

\begin{prop}  The operator $d_S$ is a differential.
\label{dsdiff}
\end{prop}

\begin{proof}
The proof that $d_S^2 = 0$ is more or less the same as the argument that the Koszul complex of $\mB(V)$ is a complex.  We have
\begin{eqnarray*}
d_S^2 (\psi \otimes m) & = & d_S \left( \sum_{w' \in S} \partial_{w'^*}(\psi) \otimes w' m \right) \\
 & = & \sum_{w \in S} \sum_{w' \in S} \partial_{w^*} \partial_{w'^*}(\psi) \otimes ww' m \\
 & = & \sum_{w, w' \in S} \partial_{w^*w'^*}(\psi) \otimes w w' m
\end{eqnarray*}

We note that the product $ww'$ acting on $m$ is the product in the ring $R$, while the product $w^*w'^*$ in $\partial_{w^*w'^*}$ is the product in the Nichols algebra $\mB(V_{\epsilon}^*)$. 
 
For each $r \in R_2$, write $\Sigma(r)$ for the set of $(w,w') \in S \times S$ such that $ww' = r$ . We note that (by definition of $R$) the set $\Sigma(r)$, when it is not empty, consists of a single orbit of the braid group $B_2$ on $S \times S$.  Write $\partial_{\Sigma(r)}$ for the operator on $\mB(V_{\epsilon})$ defined by
\beq
 \partial_{\Sigma(r)} = \sum_{(w,w') \in \Sigma(r)} \partial_{w^*w'^*}.
\eeq
 
We can now write
\beq
d_S^2 (\psi \otimes m) = \sum_{r \in R_2} \partial_{\Sigma(r)}(\psi) \otimes r m.
\eeq
It thus suffices to prove that $\partial_{\Sigma(r)}$ is the zero operator for all $r \in R_2$; because $\partial$ is an action of the algebra $\mB(V^*_{\epsilon})$ on $\mB(V_{\epsilon})$, we only need to show that
\beq
\sum_{(w,w') \in \Sigma(r)} w^*w'^*= 0 \in \mB(V^*_{\epsilon})_2.
\eeq

When $\Sigma(r)$ is empty, this is vacuous.  Otherwise, $\Sigma(r)$ consists of a braid orbit in $S \times S$.  In particular, writing as usual $\sigma: V \tensor V \ra V \tensor V$ for the braiding operation $w \tensor w'  \mapsto w' \tensor w^{w'}$, we have
\beq
\sigma \left( \sum_{(w,w') \in \Sigma(r)} w \tensor w'  \right) = \sum_{(w,w') \in \Sigma(r)} w \tensor w' \in V \tensor V
\eeq
since the action of $\sigma$ just permutes the summands.  Thus on $V^*_{\epsilon} \tensor V^*_{\epsilon}$, 
\beq
\sigma \left( \sum_{(w,w') \in \Sigma(r)} w^* \tensor w'^*  \right) = -\sum_{(w,w') \in \Sigma(r)} w^* \tensor w'^* \in V \tensor V
\eeq

But note that the quantum symmetrizer of part 3 of Proposition~\ref{nichols_characterization_prop} is just $\mS_2 = 1 + \sigma$.  We conclude that
\beq
\sum_{(w,w') \in \Sigma(r)} w^* \tensor w'^* \in \ker \mS_2
\eeq
which means $\sum_{(w,w') \in \Sigma(r)} w^* \tensor w'^*$ is killed by the projection $V_{\epsilon}^* \tensor V_{\epsilon}^* \ra \mB(V_{\epsilon}^*)_2$, and we are done.
\end{proof}

Proposition~\ref{dsdiff} is especially interesting in the case where $M=R^{(H)}$, with $H$ a subgroup of $G$ generated by $c \cap H$.  In this case, we note that the differential $d$ in $\KK(R^{(H)})$ is just $d_{c \cap H}$, since the contribution to $d(\psi \tensor m)$ of any $w$ not in $H$ is $\partial_w(\psi) \tensor wm = 0$.  

If $c_1, \ldots, c_m$ are the $H$-conjugacy classes in $c \cap H$, each $c_i$ is a subrack of $c$, so we can define differentials $d_i = d_{c_i}$ on $\KK(R^{(H)})$.  What's more, $R^{(H)}$ carries $m$ gradings, where the $(q_1, \ldots, q_m)$-graded piece of $R^{(H)}$ is spanned by Hurwitz words containing $q_i$ representatives from $c_i$. If $M$ is an $m$-tuply graded $R^{(H)}$-module, then $\KK(M)$ is $m$-tuply graded as well.   We note that each $d_i$ is a graded operator (in the $q_i$ index), sending $\KK(M)_{p,q_1, \ldots, q_m}$ to $\KK(M)_{p+1,q_1, \ldots, q_i-1, \ldots, q_m}$. It follows that the cohomology of the Koszul complex $\KK(M)$ also acquires an $m$-tuple grading. 

\begin{prop} The differentials $d_1, \ldots, d_m$ in $\KK(M)$ defined above give $\KK(M)$ the structure of an $m$-tuply graded multicomplex.
\end{prop}

\begin{proof} All that has to be checked is that $d_i d_j = - d_j d_i$ for all $i \neq j$.   Since $c_i \sqcup c_j$ is a subrack, Proposition~\ref{dsdiff} guarantees that $d_{c_i \sqcup c_j} = d_i + d_j$ is a differential. So
\beq
d_i^2 = d_j^2 = (d_i + d_j)^2 = 0
\eeq
which implies that $d_i$ and $d_j$ anticommute.
\end{proof}

\begin{rem} The argument of \cite{evw} required a strong hypothesis of {\em non-splitting} on the pair $(G,c)$; in the language of the present paper, this hypothesis was that $m=1$ for every subgroup $H < G$ appearing in the filtration, so that multicomplexes never appeared.  The algebraic methods used in the present paper are more able to handle the complications arising from the multicomplex structure, which is what allows the results of this paper to dispense with the non-splitting condition.
\end{rem}

\subsection{High degree vanishing of cohomology of Koszul complexes}

\label{ss:vanishingkoszul}

We are now ready to show that the multicomplexes above have cohomology that vanishes above a specific multidegree.  We note that if $x$ is an element of $R$, right multiplication $\rho_x: R \to R$ by $x$ is a left $R$-module map, and thus provides a map from $\KK(R)$ to $\KK(R)$, and likewise from $\KK(R^{(H)})$ to $\KK(R^{(H)})$.  We study these maps in particular for elements $g \in c_i \subset R_1$; in this case, right multiplication on coefficients by $g$ increases the $i\nth$ grading by $1$ and leaves the others unchanged.

\begin{lem} \label{cp_lem}

Let $H$ and $c_1, \ldots, c_m$ be as above.  For each $i = 1, \dots, m$, there exists a constant $B_i$ such that for every $g \in c_i$, if $q_i> B_i$ (and for all $p$ and $q_j$ for $j \neq i$), right multiplication by $g$ is an isomorphism
$$\KK(R^{(H)})_{p, q_1, \dots, q_i, \dots, q_m} \to \KK(R^{(H)})_{p, q_1, \dots, q_i+1, \dots, q_m}$$

\end{lem}

\begin{proof}  By definition, every Hurwitz word $w = (g_1, \dots, g_{q})$ in $R_{q_1, \dots, q_m}^{(H)}$ has full monodromy group $H$.  Note that each element of $c_i$ has a common order $d_i$, and define 
$$b_i: = d_i \cdot \#c_i.$$  
Then, if $w \in R_{q_1, \dots, q_m}^{(H)}$, and $q_i>b_i$, there is some element $h \in c_i$ which appears at least $d_i+1$ times in the word $w$.

Following the argument of Proposition 3.4 of \cite{evw} or the appendix to \cite{fried-volklein}, one can use the braid action to move $h^d$ to the end of the word $w$, in such a way that the remainder of the word still generates $H$.  Then one may use braid moves to conjugate $h^d$ to $g^d$ for any $g \in c_i$.  So we can write
$$w = w' \cdot g.$$
That is: right multiplication by $g$ is a surjection.

Therefore, if $q_i>b_i$, and $g \in c_i$, we have a sequence of surjections
$$\xymatrix@1{
R_{q_1, \dots, q_i-1, \dots, q_m}^{(H)} \ar[r]^-{- \cdot g} & R_{q_1, \dots, q_i, \dots, q_m}^{(H)}  \ar[r]^-{- \cdot g} & R_{q_1, \dots, q_i+1, \dots, q_m}^{(H)} \ar[r]^-{- \cdot g} & \cdots}$$
However, since the vector spaces in question have finite dimension, these surjections must eventually be isomorphisms.  Let $B_i$ be chosen so that if $q_i>B_i$, right multiplication by $g$, 
$$- \cdot g: R_{q_1, \dots, q_i, \dots, q_m}^{(H)}  \to R_{q_1, \dots, q_i+1, \dots, q_m}^{(H)}$$
is an isomorphism for every $g \in c_i$.  This induces the claimed isomorphism between the corresponding graded pieces of the Koszul complex.

\end{proof}

Having shown that right multiplication by $g$ is an isomorphism on $\KK(R^{(H)})$ in sufficiently high degree, we now show it is homotopic to zero (in {\em every} degree.)  In fact, the argument below will work for a more general bimodule $M$ in place of $R^{(H)}$, but we do not need this level of generality here. 

\begin{prop} \label{mult_vanish_prop}

Let $S$ be a subrack of $c \cap H$ and write $\KK_S(R^{(H)})$ for the complex $\mB(V_\epsilon) \tensor_R R^{(H)}$ with $d_S$ as differential.  Then, for any $g \in S$, right multiplication by $g$ is nullhomotopic on $\KK_S(R^{(H)})$.

\end{prop} 

\begin{rem}  Note in particular that the statement of the Proposition applies to $\KK(R^{(H)})$ by taking $S = c \cap H$.
\end{rem}

\begin{proof}

Define a map
$$P_g: \KK_S(R^{(H)})^{p, q} \to \KK_S(R^{(H)})^{p+1, q} \; \mbox{ by } \; \psi \otimes r \mapsto g \psi \otimes r.$$
That is, $P$ is given by left multiplication by $g \in c \subseteq V_{\epsilon} = \mB(V_{\epsilon})_1$ in the first tensor factor. 
We compute:
$$P_gd_S(\psi \otimes r) = \sum_{h \in S} g \partial_{h^*}(\psi) \otimes h r$$
and 
\begin{eqnarray*}
d_SP_g(\psi \otimes r) & = & \sum_h \partial_{h^*}(g \psi) \otimes h r \\
 & = & \sum_h \partial_{h^*}(g) \psi^{h^{-1}} \otimes h r + \sum_h g \partial_{h^*}(\psi) \otimes h r \\
 & = & \psi^{g^{-1}} \otimes g r + \left(\sum_h g \partial_{h^*}(\psi) \otimes h r \right)
\end{eqnarray*}
so
$$[d_SP_g-P_gd_S](\psi \otimes r) = \psi^{g^{-1}} \otimes g r = \psi^{g^{-1}} \otimes r^{g^{-1}} g = \psi^{g^{-1}} \otimes (rg)^{g^{-1}}$$
which shows that $\psi \otimes r \mapsto \psi^{g^{-1}} \otimes (rg)^{g^{-1}}$ is nullhomotopic.  However, this map is conjugate (using the right $G$ action on $\KK_S(R^{(H)})$ coming from the Yetter-Drinfeld structure) to the map $\psi \otimes r \mapsto \psi \otimes r g$, so that map is nullhomotopic too, which was the claim to be proved.

\end{proof}

\begin{rem} The same computation shows that $P_g$ and $d_S$ commute when $d_S$ is the differential arising from a subrack $S$ {\em not} containing $g$.
\label{pgcommute}
\end{rem}

\begin{cor}
\label{cohboundrh}

If, for any $i$, $q_i>B_i$, then $H^{p, q_1, \dots, q_m}(\KK(R^{(H)})) = 0$.

\end{cor}

\begin{proof}

Define
$$\HHH^{p, q_1, \dots, q_m}(\KK(R^{(H)})) = H^*(H^*( \cdots(H^*(\KK(R^{(H)}), d_1), \dots,) d_{m-1}), d_m)^{p, q_1, \dots, q_m}$$
to be the iterated homology of the Koszul complex $\KK(R^{(H)})$ with respect to each of the differentials $d_i$.  This is the $E_1$ term of the first of a series of $m-1$ spectral sequences that converge ultimately to $H^{p, q_1, \dots, q_m}(\KK(R^{(H)}))$; this is the $m$-fold analogue of the spectral sequence of a bicomplex.  Thus, it suffices to show that $\HHH^{p, q_1, \dots, q_m}(\KK(R^{(G)}))$ vanishes for $q_i>B_i$.

By  Proposition \ref{mult_vanish_prop}, applied with $S=c_i$, we have that right multiplication by any $g \in c_i$ is a nullhomotopy with respect to the $d_i$ differential.  Moreover, by Remark~\ref{pgcommute},  $d_j$ and $P_g$ commute when $i \neq j$.  Consequently, for each $j<i$, $P_g$ descends through the homology with respect to $d_j$ to give a nullhomotopy of multiplication by $g$.  However, the previous Lemma shows that when $q_i > B_i$, multiplication by $g$ induces a self-isomorphism of the Koszul complex (and hence of the iterated homology).  Thus the isomorphism $g$ is also $0$ whenever $q_i>B_i$.

\end{proof}

\begin{thm} \label{eventually_zero_thm}

For each $H \in \Sub_{G, c}$, there exists a constant $B_H$ so that if $q>B_H$ (and for all $p$), $H^{p, q}(\KK(R^{(H)})) = 0$. Having specified these constants, we have $H^{p, q}(\KK(R)) = 0$ for all $q> \max_H (B_H)$.

\end{thm}

\begin{proof}

The second statement follows from the first via the spectral sequence of Corollary \ref{K_ss_cor}. 

Decompose $c \cap H = c_1 \sqcup \dots \sqcup c_m$ into a disjoint union of conjugacy classes; then Corollary~\ref{cohboundrh} shows that $H^{p, q_1,\dots, q_m} \KK(R^{(H)}) = 0$ whenever any $q_i$ is greater than some constant $B_i$ which depends only upon $H$ and $c_i$.  Taking $B_H = \sum B_i$, if $q = \sum q_i > B_H$, then for some $i$, it must be the case that $q_i>B_i$, so $H^{p, q}(\KK_H(R^{(H)})) = 0$.

\end{proof}

\begin{rem} It might be complicated to write down a good bound for the threshold $B_H$ in general, but for any particular choice of $G$ and $c$, one should be able to compute $B_H$ directly; it depends only on the combinatorics of braid orbits on long tuples of elements of $c \cap H$.  Getting good explicit bounds for $B_H$ would provide explicit bounds for how large the finite field $\F_q$ must be in order for the main counting results of this paper to apply, which would certainly be of interest.
\end{rem}

\subsection{Cohomological growth of Nichols algebras}

\begin{cor} \label{koszul_bound_cor}

There is a constant $D = D(G, c)$ such that
$$\rk(\bigoplus_{q \geq 0} H^{1+j, q}(\KK(R))) \leq D(\# c)^j.$$

\end{cor}

\begin{proof}

Theorem \ref{eventually_zero_thm} shows that the terms in this sum vanish for $q> \mu := \max_H (B_H)$.  The rank of this cohomology is bounded by that of the Koszul complex itself, so:
\begin{eqnarray*} \rk_k \bigoplus_{q = 0}^{\mu} H^{1+j, q}(\KK(R)) & \leq & \rk_k \mB(V_\epsilon)_{1+j} \cdot \left( \sum_{q=0}^{\mu} \rk_k R_q \right) \\
\end{eqnarray*}
Since $\mB(V_\epsilon)$ is a quotient of the tensor algebra on $V_{\epsilon}$, the rank of the first term in the product is bounded by $(\# c)^{1+j}$.  The result follows if we set 
$$D := (\#c) \sum_{q=0}^{\mu} \rk_k R_q.$$

\end{proof}

\begin{cor} \label{we_win_cor}

For arbitrary $(G, c)$, if $V = kc$, Hypothesis \ref{e_o_conj} holds for $\mB(V_\epsilon^*)$: there is a constant $C_{\mB}$ such that the ranks of the cohomology of $\mB(V^*_{\epsilon})$ are bounded as
$$\rk(\Ext_{\mB(V^*_{\epsilon})}^{n-j, n}(k, k)) \leq p(n) C_{\mB}^j,$$
where $p(n) = \max_{m \leq n} \dim_k (R_m).$ 

Consequently Hypothesis \ref{exp_hyp} holds for $\mA(V^*_{\epsilon})$: if $C  = 4\max\{C_{\mB}, \#c\}^2$, then
$$\rk(\Ext_{\mA(V^*_{\epsilon})}^{n-j, n}(k, k)) \leq p(n) C^j.$$

\end{cor}

\begin{rem} It follows from Lemma~\ref{cp_lem} (the Conway-Parker-Fried-V\"{o}lklein argument) that $p(n)$ is a polynomial for $n$ large enough.  In fact, we can describe its degree explicitly; if $m$ is the the maximum over all subgroups $H < G$ of the number of $H$-conjugacy classes making up $c \cap H$, then $\deg p(n) = m-1$.  In particular, $p(n)$ is eventually constant if and only if $c$ consists of a single conjugacy class and $(G,c)$ satisfies the non-splitting condition of \cite{evw}.  We note that the quantity $m$ is also the Gelfand-Kirillov dimension of $R$.   
\label{rem:pnpoly}
\end{rem}

\begin{proof}

The second statement follows from the first by Theorem \ref{alg_bound_thm}. 

Theorem \ref{main_thm} (with $A = \mB(V_{\epsilon}^*$) and $S=R$) gives
$$\dim \Ext_{\mB(V^*_{\epsilon})}^{n-j, n}(k, k)) \leq \sum_{m=0}^j \dim (H^*(\KK(R)))^{\tensor m} \otimes R)_{n-j,n}.$$
This bound\footnote{Given that $\Ext_{\mB(V_\epsilon^*)}(k, k)$ is a braided symmetric algebra, bounding its growth by that of a tensor algebra over $R$ is undoubtedly wildly inefficient.  However, for our purposes, the above suffices.} is none other than the dimension of the bidegree $(n-j,n)$ piece of $U \otimes R$, where $U$ is the superdiagonal subalgebra of the tensor algebra $T(H^*(\KK(R)))$.  

We now apply Proposition \ref{exp_growth_tensor_prop} with $M = H^*(\KK(R))$, thought of as a $k$-vector space graded by topological degree.  Corollary \ref{koszul_bound_cor} says exactly that the $j$th graded piece of $M$ has dimension at most $D(\#c)^j$.  It now follows from Proposition \ref{exp_growth_tensor_prop} that $U_j$ has dimension at most $[(4D)(\#c)]^j$.  Thus $U_j \tensor R$ is a free $R$-module on at most $[(4D)(\#c)]^j$ generators, all of them of nonnegative bidegree; it follows immediately that the bidegree $(n-j,n)$ piece of $U \otimes R$ has dimension at most $p(n) [(4D)(\#c)]^j$.  This is a bound of the form claimed.
\end{proof}

\section{Malle's conjecture over function fields} \label{malle_section}

\subsection{Hurwitz spaces} \label{hur_rack_section}

Let $c \subseteq G$ be a union of conjugacy classes in a finite group $G$.  We will regard $c$ as a rack and obtain a braided vector space $V = kc$ via the definition in Section \ref{rack_section}.  Note that the action of $B_n$ on $V^{\otimes n}$ is by permutation of the basis $c^{\times n}$ of $V^{\otimes n}$.  We recall that $B_n = \pi_1 \Conf_n(\C)$ is the fundamental group of the (unordered) configuration space of the plane.

\begin{defn}

The $n^{\rm th}$ \emph{Hurwitz space} of $\C$ associated to the pair $(G, c)$ is the quotient
$$\Hur_{G, n}^c := \widetilde{\Conf}_n(\C) \times_{B_n} c^{\times n}$$
where $\widetilde{\Conf}_n(\C)$ is the universal cover of $\Conf_n(\C)$, and the action of the braid group is diagonal.  There is an isomorphism
$$H_*(\Hur_{G, n}^c; k) \cong H_*(B_n, V^{\otimes n}).$$

\end{defn}

\begin{rem} The space $\Hur_{G, n}^c$ may be identified with the moduli space of branched covers of the unit disk $D \subseteq \C$ with $n$ branch points in the interior of $D$, Galois group $c$, and a trivialization of the covering over $1 \in \partial D$.  Furthermore, we may define a subspace $\CHur_{G, n}^c$ to be the union of components consisting of branched covers whose domain is connected.  See section 2.3 of \cite{evw} for details.  \end{rem}

We notice that 
$$\Hur_G^c := \coprod_n \Hur_{G, n}^c$$
forms an $A_\infty$ H-space; on the base $\coprod \Conf_n(\C)$ this is the usual $A_\infty$ (in fact, $E_2$) multiplication given by the natural action of the 2-dimensional little disks operad. On the fibres it is the identity map $c^{\times m} \times c^{\times n} \to c^{\times m+n}$, which is equivariant for the homomorphism $B_m \times B_n \to B_{m+n}$.  

Corollary \ref{main_cor} gives us a ring isomorphism
$$H_j \Hur_{G, n}^c \cong \Ext_{\mA(V^*_\epsilon)}^{n-j, n}(k, k)$$
between the homology of the Hurwitz space with this product and the cohomology of the quantum shuffle algebra $\mA = \mA(V^*_\epsilon)$.

\subsection{Counting $G$-covers via the \'{e}tale cohomology of Hurwitz schemes} 
\label{ss:hurwitzschemes}

Malle's conjecture concerns extensions of a global field $K$ with Galois group $G$.  When $K$ is a rational function field $\F_q(t)$, there is a natural bijection between extensions of $K$ and Galois covers of $\P^1$ over $\F_q$; the latter objects can be approached by means of counting $\F_q$-rational points on a certain moduli scheme.  We make this more precise as follows, following the treatment of \cite[\S 7]{evw}.  The following definition is adapted from \cite[\S 2.1]{romagnywewers}.  As is the case throughout this paper, we assume that $|G|$ is prime to the characteristic of $K$, so that all tameness conditions in \cite{romagnywewers} are automatically satisfied.

\begin{defn} Let $G$ be a finite group and $k$ a field.  A {\em tame $G$-cover of $\P^1$ over $k$} is a triple $(X,f,\tau)$ where
\begin{itemize}
\item $X/k$ is a smooth proper geometrically connected curve;
\item $f: X \ra \P^1 / k$ is a tame Galois cover; that is, $f$ is a finite separable map,  \'{e}tale away from a reduced divisor $D \subset \P^1$, such that $\Aut(f)$ acts transitively on every geometric fiber of $f$, and such that the ramification of $f$ over each geometric point of $D$ is nontrivial and prime to the characteristic of $k$ (i.e. $f$ is tamely ramified.)
\item $\tau: G \ra \Aut(f)$ is an isomorphism.
\end{itemize}
\label{def:gcover}
\end{defn}

In fact, we will require from this point on that the characteristic of $k$ is prime to $|G|$, making the tame ramification condition in Definition~\ref{def:gcover} automatic.

We say two $G$-covers $(X,f,\tau),(X',f',\tau')$ are isomorphic if there exists an isomorphism $h: X \ra X'$ over $k$ such that $f' \circ h = f$ and $h \circ \tau(g) = \tau'(g) \circ h$ for all $g \in G$.    

As in \cite{evw}, there is a slight technical annoyance to address arising from the fact that in algebraic geometry it is simpler to talk about covers of projective curves, while our topological story concerns $G$-covers of the disc, which is homeomorphic not to $\P^1(\C)$ but to $\A^1(\C)$.  By a {\em tame $n$-branched $G$-cover of $\A^1$ over $k$} (or, when context makes the parameters clear, a $G$-cover of $\A^1$) we shall mean a tame $G$-cover of $\P^1$ over $k$ as in Definition~\ref{def:gcover} with the property that $D \cap \A^1 \subset \P^1$ is a divisor of degree $n$.

We may also want to specify the local monodromy type of a $G$-cover.  Our covers are all tamely ramified by assumption; the local tame inertia group of $\A^1$ in the neighborhood of a puncture is canonically identified with $\hat{\Z}(1)$, so the local monodromy in a branched $G$-cover can be thought of as a conjugacy class of morphisms $\hat{\Z}(1) \ra G$.  We choose, once and for all, a generator for $\hat{\Z}(1)$; having done so, the local monodromy morphism is simply a conjugacy class in $G$, and so we have attached a multiset of $n$ conjugacy classes to the $n$-branched cover.  We say a multiset of conjugacy classes is {\em $\F_q$-rational} when it is closed under the natural ``Frobenius action" on $G$ which sends $g$ to $g^q$.  It is easy to see that when $X \ra \A^1/\F_q$ is a $G$-cover, its local monodromy must be $\F_q$-rational.

With this in mind, we say a subset of $G$ is $\F_q$-rational if it is closed under the permutation $g \mapsto g^q$ (a permutation because $q$ is always prime to $|G|$), and we suppose from now on that $c$ is a subset of $G$ which is not only closed under conjugacy but is $\F_q$-rational as well.  (In \cite{evw} we imposed a slightly stronger condition, that $c$ was closed under the full action of $\Zhat^\times$ rather than that of the subgroup $q^{\Zhat}$; in retrospect, the weaker condition of $\F_q$-rationality would have sufficed in \cite{evw} as well.)

As mentioned above, if $L/\F_q(t)$ is a Galois $G$-extension of function fields, then the unique smooth proper curve $X/\F_q$ with function field $L$ acquires a natural map $f$ to $\P^1$; if $U$ is the preimage $f^{-1}(\A^1)$, then $f: U \ra \A^1$ carries the structure of a $G$-cover of $\A^1/\F_q$, and this association is a bijection between field extensions and $G$-covers.

It would thus be convenient if there were a moduli scheme $\Hn_{G,n}^c$ such that the set of $k$-points $\Hn_{G,n}^c(S)$ were naturally in bijection with the set of $G$-covers over $k$ with $n$ branch points and local monodromy of type $c$.  If that were the case we could recast Malle's conjecture as a question about counting $\F_q$-rational points on $\Hn_{G,n}^c$.
  
Unfortunately, matters are not quite so simple.  The functor sending $S$ to the category of $G$-covers of $\A^1$ over $S$ is not representable by a scheme in general, but only by an algebraic stack.  The problem is that $G$-covers may have nontrivial automorphisms.  An automorphism of $(X,f,\tau)$ is an isomorphism $h: X \ra X$ such that $h \circ f = f$; in other words, it is an element of $\Aut(f)$, and thus, by definition of $G$-cover, an element $g \in G$.  Moreover, the requirement that $h$ be compatible with $\tau$ means that $h$ has to commute with every element of $G$; that is, $h$ is an element of the center $Z_G$.

When $G$ has trivial center, this problem is absent, and indeed it turns out there is a moduli scheme for $G$-covers.  We explain in \cite[\S 7.4, 7.7-7.8]{evw} how, given this hypothesis, we may use bounds on cohomology arising from topology to control the number of $\F_q$-rational $G$-covers of $\A^1$.  In that paper, our primary application concerned the case where $G$ was a generalized dihedral group, so not much was lost by assuming $G$ was center-free.  In the present paper, more generality is desired, so we will address this issue.  While there is no moduli scheme for $G$-covers when $G$ has nontrivial center, one does have the next best thing, a {\em coarse moduli scheme} over $\Z[1/|G|]$, which is what we denote $\Hn_{G,n}^c$;  this is defined by Romagny and Wewers in \cite[\S 4]{romagnywewers}.  \footnote{Romagny and Wewers don't discuss the specification of local monodromy, but $\Hn_{G,n}^c$ is an open and closed subcheme of their coarse moduli scheme, as explained in the paragraph before Lemma 7.4 of \cite{evw}.  The condition that $c$ is $\F_q$-rational guarantees that $\Hn_{G,n}^c$ is defined over $\F_q$.  See also \cite[\S 7.3]{evw} for discussion of how $\Hn_{G,n}^c$ is obtained from the moduli spaces of covers of $\P^1$ produced in \cite{romagnywewers}.}

\begin{prop}  Let $G$ be a finite group, $\F_q$ a finite field with characteristic prime to $|G|$, and $c$ an $\F_q$-rational conjugacy-invariant subset of $G$.  The number of $G$-covers of $\A^1$ over $\F_q$ with $n$ branch points and local monodromy of type $c$ is  
\beq |Z_G| |\Hn_{G,n}^c(\F_q)|. \eeq
\label{pr:stackycount}
\end{prop}

\begin{rem}  The ``right" proof of Proposition~\ref{pr:stackycount} would be to say: there is a Deligne-Mumford moduli stack $\XX$ for $G$-covers, every point of which has inertia group $Z_G$; for whatever field $k$ of characteristic not dividing $|G|$, the $k$-rational points of $\XX$ are the $G$-covers over $k$.  For any field $k$, the obstruction to a $k$-rational point of the coarse moduli space $\Hn_{G,n}^c$ lifting to a $k$-rational point of $\XX$ lies in $H^2(k,Z_G)$; when $k$ is finite, this cohomology group is trivial, so every $k$-rational point of the coarse moduli space actually parametrizes a $k$-rational $G$-cover.  What's more, the set of $k$-rational points of $\XX$ over a $k$-rational point of $\Hn_{G,n}^c$ is a torsor for the cohomology group $H^1(k,Z_G)$; when $k$ is a finite field, this cohomology group has size $|Z_G|$.  This is essentially the argument of \cite[Cor 2.3.4]{behrend:thesis} (ascribed by Behrend to Serre.)

The reason for not relying on this argument in the present case is merely because the construction of the Deligne-Mumford moduli stack of $G$-covers is carried out in unpublished work of Wewers~\cite{wewers} and because we don't want to give the impression that the results of the present paper depend in a critical way on theorems about stacks.  It is certainly true, however, that the stack point of view is the most illuminating way to think about the technical difficulties arising from the case where $G$ has nontrivial center.
\end{rem}

We now return to the proof of Proposition~\ref{pr:stackycount}.

\begin{proof}

We will need the following properties of the coarse moduli scheme $\Hn_{G,n}^c$:

\begin{itemize}
\item If $(X,f,\tau)$ is a branched $G$-cover of $\A^1$ over an algebraically closed field $k$, parametrized by a point $x$ of $\Hn_{G,n}^c(k)$, and $\sigma: k \ra k$ is a field automorphism, then the branched $G$-cover $(X,f,\tau)^\sigma$ is parametrized by $x^\sigma$. (\cite[4.11 i)]{romagnywewers})
\item There is a finite \'{e}tale morphism $\pi$ from $\Hn_{G,n}^c$ to the configuration space $\UU_n$ of $n$ points on $\A^1$, which sends $(X,f,\tau)$ to the branch locus of $f$. 
\end{itemize}

These facts are essentially drawn from \cite[4.11]{romagnywewers}, where the analogous facts are proved for their moduli space $\mathcal{H}_G^n$, which parametrizes branched $G$-covers of $\P^1$, rather than branched covers of $\A^1$.  We recall again \cite[\S 7.3]{evw}, which explains that $\Hn_{G,n}^c$ is a union of locally closed subschemes of $\mathcal{H}_G^n$ and $\mathcal{H}_G^{n+1}$.  In particular, the results of \cite[4.11]{romagnywewers} apply to $\Hn_{G,n}^c$ as well as to $\mathcal{H}_G^n$; the reader should have this in my mind when we refer to \cite{romagnywewers} again later in the argument.

We will now prove that the set of branched $G$-covers of $\A^1$ parametrized by a point $x \in \Hn_{G,n}^c(\F_q)$ has cardinality $|Z_G|$, which immediately implies Proposition~\ref{pr:stackycount}.

  Let $\Sigma$ be the set of isomorphism classes of $G$-covers of $\A^1$ over $\F_q$ with $n$ branch points and local monodromy of type $c$.  If $(X,f,\tau)$ is such a cover, then the base change $(\bar{X},\bar{f},\bar{\tau})$ of $(X,f,\tau)$ from $\F_q$ to $\Fqbar$ is a $G$-cover of $\A^1$ over $\Fqbar$.  By \cite[4.11 i)]{romagnywewers},  the cover $(\bar{X},\bar{f},\bar{\tau})$ is parametrized by a point of $\Hn_{G,n}^c(\Fqbar)$ which is fixed by the Galois action, and is thus a point of $\Hn_{G,n}^c(\F_q)$.  This provides a map from $\Sigma$ to $\Hn_{G,n}^c(\F_q)$; we will now show that all fibers of this map have size $|Z_G|$.

Let $(X,f,\tau)$ be an element of $\Sigma$, let $D \in \A^1/\F_q$ be its branch locus, and let $U$ be the complement of $D$ in $\A^1$.  Then the \'{e}tale fundamental group $\pi_1^{\et}(U)$ fits into an exact sequence
\beq
1 \ra \pi_1^{\et}(\bar{U})  \ra \pi_1^{\et}(U) \ra \Gal(\Fqbar/\F_q) \ra 1
\eeq
where $\bar{U}$ is the base change of $U$ to $\Fqbar$.  This sequence affords an outer action of $\Gal(\Fqbar/\F_q)$ on $\pi_1^{\et}(\bar{U})$.  Write $F$ for an element of $\pi_1^{\et}(U)$ projecting to Frobenius in $\Gal(\Fqbar/\F_q)$.

Write $\Sigma_U$ for the subset of $\Sigma$ consisting of covers whose branch locus is $D$.  Any $(X,f,\tau) \in \Sigma_U$ is \'{e}tale when restricted to $U$; indeed, $\Sigma_U$ may be identified with the set of \'{e}tale $G$-covers of $U$ with local monodromy of type $c$.  In other words, $\Sigma_U$ is in bijection with the set of conjugacy classes of surjective homomorphisms
\beq
\phi: \pi_1^{\et}(U) \ra G
\eeq
which send the specified tame inertia generators at each geometric point of $D$ (precisely: the generators specified by our choice of a generator of $\Zhat(1)$) to $c$, and which remain surjective after restriction to $\pi_1^{\et}(\bar{U})$.  (This last condition is imposed by the requirement that $X$ is geometrically connected.)

On the other hand, suppose $x$ is a point of $\Hn_{G,n}^c(\F_q)$.  The image of $x$ under the finite morphism $\pi$ from $\Hn_{G,n}^c$ to the configuration space $\UU_n$ of $n$ points on $\A^1$ (as in \cite[4.11 ii)]{romagnywewers}) parametrizes a configuration of $n$ points in $\A^1$; since $\UU_n$ is a fine moduli space, $\pi(x)$ corresponds to a divisor $D \in \A^1/\F_q$ -- or, what is the same, an open subscheme $U=\A^1-D$.  The $\Fqbar$ points of the corresponding fiber $\pi^{-1}(\pi(x))$ parametrize the set of \'{e}tale $G$-covers of $\bar{U}/\Fqbar$ with local monodromy of type $c$.  The $\F_q$-rational points of $\pi^{-1}(\pi(x))$ are just the points of $\pi^{-1}(\pi(x))$ which are fixed by the action of $\Gal(\Fqbar/\F_q)$; that is, they are those conjugacy classes of surjective homomorphisms
\beq
\pi_1^{\et}(\bar{U}) \ra G
\eeq
which are preserved under the outer action of $\Gal(\Fqbar/\F_q)$ on $\pi_1^{\et}(\bar{U})$.  

If $\phi: \pi_1^{\et}(U) \ra G$ is an element of $\Sigma_U$, the corresponding point in $\Hn_{G,n}^c(\F_q)$ is just the restriction $\phi | \pi_1^{\et}(\bar{U})$.  Having expressed the map from $\Sigma$ to $\Hn_{G,n}^c(\F_q)$ in group-theoretic terms, we proceed to the proof of the proposition.

First, we need to know that the map $\Sigma \ra \Hn_{G,n}^c(\F_q)$ is surjective.  Given a point $x$ of $\Hn_{G,n}^c(\F_q)$, let 
\beq
\psi: \pi_1^{\et}(\bar{U}) \ra G
\eeq
be the corresponding surjective homomorphism.  We know $\psi$ is preserved up to conjugacy by the outer action of  $\Gal(\Fqbar/\F_q)$; that is, there exists $\alpha \in  \pi_1^{\et}(\bar{U})$ such that
\beq
\psi(\gamma^F) = \psi(\alpha^{-1} \gamma \alpha)
\eeq
for all $\gamma \in  \pi_1^{\et}(\bar{U})$.  This means we can extend $\psi$ to a homomorphism $\phi: \pi_1^{\et}(U) \ra G$ by defining $\phi(F) = \psi(\alpha)$; thus, there is a cover in $\Sigma$ mapping to $x$.

In other words, we know the fiber $\Sigma_x$ of $\Sigma$ over $x$ is nonempty.  The fiber can be computed immediately from the inflation-restriction sequence in nonabelian group cohomology, but for the reader's convenience we will write out the argument directly.

Let $\phi, \phi':  \pi_1^{\et}(U) \ra G$ be two elements of $\Sigma_x$.  Then $\phi$ and $\phi'$ become conjugate upon restriction to $\pi_1^{\et}(\bar{U})$.  Because $\phi$ and $\phi'$ are only defined up to conjugacy, we may assume $\phi$ and $\phi'$ actually agree after restriction to $\pi_1^{\et}(\bar{U})$.  Now the map
\beq
\zeta: \gamma \mapsto \phi'(\gamma)\phi(\gamma^{-1})
\eeq
is a cocycle in $H^1(\pi_1^{\et}(U),G)$, where the action of $\gamma$ on $G$ is conjugation by $\phi(\gamma^{-1})$.  

We note that $\zeta(n) = 1$ for all $n \in \pi_1^{\et}(\bar{U})$.  By the cocycle relation, we have on the one hand 
\beq
\zeta(n \gamma) = \zeta(\gamma)^{\phi(n)}
\eeq
and on the other hand
\beq
\zeta(n \gamma) = \zeta(\gamma n^{\gamma}) = \zeta(\gamma).
\eeq
It follows that $\zeta(\gamma)$ is invariant under conjugation by $\phi(n)$ for all $n \in \pi_1^{\et}(\bar{U})$.  But $\phi$ is surjective on restriction to $\pi_1^{\et}(\bar{U})$, so $\zeta(\gamma)$ must actually lie in the center $Z_G$.  This implies, in turn, that $\zeta$ is actually a homomorphism from $\pi_1^{\et}(U)$ to $Z_G$ which factors through the quotient  $\Gal(\Fqbar/\F_q)$.

To sum up:  every element of $\phi'$ of $\Sigma_x$ can be written as $\phi \zeta$ where $\zeta$ is a homomorphism from $\pi_1^{\et}(U)$ to $Z_G$ factoring through $\Gal(\Fqbar/\F_q)$.  Conversely, any such $\phi \zeta$ is an element of $\Sigma_x$.  

Suppose $\phi \zeta$ and $\phi \zeta'$ agree in $\Sigma_x$, which is to say they are conjugate by some element $g \in G$.  Both $\phi \zeta$ and $\phi \zeta'$ restrict to the same surjective homomorphism $\psi:  \pi_1^{\et}(\bar{U}) \ra G$; for $g$ to conjugate $\psi$ to itself, $g$ must lie in $Z_G$, and this implies that $\zeta = \zeta'$.

We have now shown that $\Sigma_x$ is in bijection with $\Hom(\Gal(\Fqbar/\F_q),Z_G)$, an abelian group isomorphic to $Z_G$.  This proves the claimed result.
\end{proof}

\begin{prop} The ranks of $H^{2n-j}_{c, \et}(\Hn_{G, n}^c/\Fqbar, \Q_\ell)$ and $H_{j}(\CHur_{G,n}^c/G, \Q_\ell)$ are equal. \label{same_rank_prop}
\end{prop}

\begin{proof}    Until now we have been working entirely in finite characteristic, but we now need to recall that the Hurwitz spaces $\Hn_{G,n}$ defined in \cite{romagnywewers} are defined over $\Spec \Z[1/|G|]$.  Our $\Hn_{G,n}^c$, obtained from $\Hn_{G,n}$ by imposing constraints on local monodromy, is a closed and open subscheme of $\Hn_{G,n}$, defined over a cyclotomic extension of $\Z$ having a residue field contained in $\F_q$.  We bring this up only in order to have access to the manifold $\Hn_{G,n}^c(\C)$, which, as we now explain, is homeomorphic to $\CHur_{G,n}^c/G$.  This is essentially the same as Lemma 7.4 of \cite{evw}, where this statement is proved in the case that $G$ is centerfree.  In turn, that lemma is essentially an explication of \cite[4.11 (iii)]{romagnywewers}.  It turns out that the triviality of $Z_G$ is not really used in \cite[Lemma 7.4]{evw}, whose proof we recapitulate below.

The complex manifold $\CHur_{G,n}^c/G$ is a finite-sheeted cover of the configuration space $\Conf_n(\C)$; any such cover corresponds to a finite set with an action of the $n$-strand braid group.  The finite set can be identified with the set $\Ni(G,n,c) = c^{\times n}/G^{ad}$ of (Nielsen) equivalence classes of $n$-tuples $(g_1, \ldots, g_n)$ where each $g_i$ lies in $c$ and the equivalence relation is simultaneous conjugacy by $G$ on all entries.  We note that $Z_G$ acts as the identity on $\CHur_{G,n}^c$; our quotient here is the usual quotient of topological spaces and not a homotopy quotient, so the quotient by $G$ is just the same as the quotient by $G/Z_G$. 

On the other hand, the complex manifold $\Hn_{G,n}^c(\C)$ is a finite algebraic cover of the complex configuration space $U_n(\C)$.  Such a cover is also described by a finite set with a braid group action; in this case, the finite set is the set of isomorphism classes of $G$-covers of $\A^1$ over $\C$ with a fixed degree-$n$ branch divisor $D$ and local monodromy of type $c$, by \cite[4.11]{romagnywewers}.  Such an isomorphism class is equivalent to a conjugacy class of surjective homomorphisms from $\pi_1^{\et}(\A^1 - D)$ to $G$ sending the puncture classes to $c$, and the set of surjective homomorphisms is again $\Ni(G,n,c)$ with the same braid group action.  This proves the existence of the desired homeomorphism.

The rest of the argument proceeds exactly as in \cite{evw}.  Proposition 7.7 of \cite{evw} shows that $H^{j}_{\et}(\Hn_{G, n}^c/\F_q, \Q_\ell)$ is isomorphic to $H^{j}_{\et}(\Hn_{G, n}^c/\C, \Q_\ell)$; all that's needed is that $\Hn_{G,n}^c$ is an \'{e}tale cover of configuration space, which is true without the assumption that $G$ is centerfree (\cite[4.11]{romagnywewers}).  Comparison of \'{e}tale and analytic cohomology for the complex variety  $\Hn_{G, n}^c/\C$ and Poincar\'{e} duality complete the proof.

\end{proof}

\begin{thm} \label{weak_counting_thm}

Let $p(n) = \max_{m \leq n} \dim_k (R_m)$ as in Corollary~\ref{we_win_cor}.  For every $C > 1$, there exists a constant $Q = Q(G, c,C)$ such that, for all $q>Q$ and prime to $|G|$, and and all positive integers $n$, we have
\beq
|\Hn_{G, n}^c(\F_q)| \leq C p(n) q^n.
\eeq

\label{th:hurcount}

\end{thm}

\begin{proof}

Proposition \ref{same_rank_prop} gives an isomorphism of $\Q_\ell$ vector spaces 
$$H^{2n-j}_{c, \et}(\Hn_{G, n}^c/\Fqbar, \Q_\ell) \cong H_{j}(\CHur_{G,n}^c/G, \Q_\ell).$$
Notice that $\CHur_{G,n}^c/G$ is a finite quotient of a union of components of $\Hur_{G,n}^c$.  Thus the dimension of the homology of the former is bounded by that of the latter.   This, in turn, can be expressed in terms of the dimension of an Ext group:
$$\dim_k H_j \Hur_{G, n}^c = \dim_k \Ext_{\mA(V^*_\epsilon)}^{n-j, n}(k, k)$$ 

We have shown in Corollary \ref{we_win_cor} that there exists a constant $Y = Y(G,c)$ such that 
\beq
\dim \Ext_{\mA(V^*_\epsilon)}^{n-j,n} \leq p(n)Y^j.
\eeq
From those bounds on Betti numbers, we obtain a bound on the number of points in $\Hn_{G, n}^c(\F_q)$ using the Grothendieck-Lefschetz fixed point theorem and Deligne's bounds on the eigenvalues of the geometric Frobenius:
\begin{eqnarray*}
\frac{\# \Hn_{G, n}^c(\F_q)}{ q^n} & = & q^{-n} \sum_{j=0}^{2n} (-1)^j \tr(\Frob | H^{2n-j}_{c, \et}(\Hn_{G, n}^c/\Fqbar, \Q_\ell)) \\
& \leq &  \sum_{j=0}^{2n}   q^{-j/2} \rk H^{2n-j}_{c, \et}(\Hn_{G, n}^c/\Fqbar, \Q_\ell) \\
& \leq &  \sum_{j=0}^{2n}  q^{-j/2} p(n) Y^j 
\end{eqnarray*}
Then, for any $q > Y^2$ we have
\beq
\frac{\# \Hn_{G, n}^c(\F_q)}{ q^n} \leq \sum_{j=0}^{2n}  q^{-j/2} p(n) Y^j \leq p(n) \sum_{j=0}^{\infty}  q^{-j/2} Y^j \leq C p(n)
\eeq
as long as we choose $q$ large enough that $(1-q^{1/2} Y)^{-1} < C$.  This proves the claim.

\end{proof}

We now show how to derive the upper bound in the weak Malle conjecture from Theorem~\ref{th:hurcount}.  In fact, we prove something slightly more general, addressing ``generalized discriminants" as considered in \cite{ev}.  Let $G$ be a finite group, $c$ an $\F_q$-rational conjugacy-invariant subset of $G$, and $w: c \ra \Z_{>0}$ a conjugacy-invariant function on $c$.

 If $\phi: Y \ra \P^1/\F_q$ is a branched $G$-cover with all local monodromy of type $c$, we define the {\em $w$-discriminant} of $\phi$ to be the product over all branch points $y_1, \ldots, y_n$ in $\A^1$ of $q^{w(g_{y_i})}$, where $g_{y_i}$ is the local monodromy at $y_i$ (that is, the image in $G$ of the generator of $\hat{\Z}(1)$ we chose earlier and have left fixed throughout.)  It may seem a little artificial, not to mention non-canonical, that this discriminant receives contributions only from ramification points in $\A^1$, and ignores any ramification at $\infty$.  This choice is not crucial but it is convenient; first of all, it means that the cover $Y$ is represented by a point of $\Hn_{G, n}^c$, which parametrizes covers with $n$ branch points in $\A^1$; second, it works well in the analogy with number fields, where our discriminant does not include any contribution from the archimedean place or places.

When $G$ is a subgroup of $S_m$, there is a natural choice for $w$; send a conjugacy class $c$ (thought of as a permutation $\pi$ on $n$ letters) to its {\em index}, the difference between $m$ and the number of orbits of $\pi$.  In this case, the $w$-discriminant of a cover is simply the usual discriminant of the degree-$m$ cover of $\A^1$ associated to $\phi$ via the embedding $G \ra S_m$.

\begin{thm}
Let $G$ be a finite group, $c$ a conjugacy-invariant subset of $G \bs e$, and $w$ a conjugacy-invariant function on $c$.  Let  $N_{G}^{c,w}(\F_q(t); X)$ be the number of $G$-extensions of $\F_q(t)$ with all local monodromy contained in $c$ and $w$-discriminant at most $X = q^r$.  Suppose further that $q$ is prime to $|G|$ and is at least $Q(G,c)$, the constant appearing in Theorem~\ref{weak_counting_thm}.  Then 
\beq
N_{G}^{c,w}(\F_q(t); X) = O(X^{a(w,G,c)} (\log X)^{d-1})
\eeq
where $a(w,G,c)$ is $(\min_{g \in c} w(g))^{-1}$, and where $d$ is the Gelfand-Kirillov dimension of $R$.  
\label{th:mainmalle}
\end{thm}

\begin{proof}

Note first that we haven't imposed the condition that $c$ be $\F_q$-rational.  The reason is not very deep. Since any $G$-cover of $\A^1/\F_q$ has $\F_q$-invariant monodromy, a $G$-cover of $\A^1/\F_q$ with monodromy contained in $c$ in fact has monodromy contained in an $\F_q$-invariant subset of $c$.  So if $c$ is not $\F_q$-rational, we can simply run the argument below for each $\F_q$-invariant subset of $c$ separately.  Thus, in the proof, we can and do assume that $c$ is $\F_q$-rational. 

Let $c_1, \ldots, c_m$ be the conjugacy classes comprising $c$.  Then the $w$-discriminant of a $G$-cover with $m_i$ branch points of type $c_i$ is $q^{\sum_i m_iw (c_i)}.$  Thus a $G$-cover with local monodromy of type $c$ and $w$-discriminant $q^r$ has at most $r a(w,G,c)$ branch points.  Theorem~\ref{th:hurcount} now tells us that, for all $q > Q(G,c)$, we have
\beq
N_{G}^{c,w}(\F_q(t); X)  \leq  \sum_{n=1}^{r a(w,G,c)}  |\Hn_{G, n}^c(\F_q)| \leq \sum_{n=1}^{r a(w,G,c)} C p(n) q^n.
\eeq

As noted in Remark~\ref{rem:pnpoly}, the function $p(n)$ is a polynomial of degree $d-1$ for $n$ large enough, and in particular is bounded above by $B n^{d-1}$ for some $B$, and thus by $B r^{d-1} a(w,G,c)^{d-1}$ for all $n$ in the range $1 \leq n \leq r a(w,G,c)$.  We then have
\beq
N_{G}^{c,w}(\F_q(t); X) \leq \sum_{n=1}^{r a(w,G,c)} C p(n) q^n \leq BC r^{d-1} a(w,G,c)^{d-1} (1-1/q)^{-1} q^{r a(w,G,c)}.
\eeq
Since $\log X = r \log q$, this upper bound is bounded above by a constant multiple of
\begin{equation}
X^{a(w,G,c)} (\log X)^{d-1}
\label{finalbound}
\end{equation}
as claimed.

It is now now at last time to pay the piper concerning the issue of geometrically connected versus connected.  Fortunately, the fee is modest.  Recall that our ultimate goal is to count extensions of the field $\F_q(t)$ with Galois group $G$.  Such an extension is {\em almost} the same thing as a branched $G$-cover of $\P^1/\F_q$, except for our requirement that a branched $G$-cover be geometrically connected.  What about the $G$-extensions $K/\F_q(t)$ corresponding to covers $Y \ra \P^1$ where $Y$ is connected but not geometrically connected?  In this scenario, there's a subgroup $G_0 < G$ with $G/G_0$ cyclic of order $m$, and $K$ contains $\F_{q^m}$.  So $K/\F_{q^m}(t)$ is a $G$-extension whose discriminant is equal to that of $K/\F_q(t)$, and which is {\em regular} -- that is, it contains no extension of the constant subfield.  The above argument provides an upper bound for the number of regular $G$-extensions of discriminant at most $X$; now apply this with $G_0$ in place of $G$, for each of the finitely many possibilities for $G_0$, and sum.  We note that the exponent $a(w,G_0,c \cap G_0)$ cannot exceed $a(w,G,c)$.  Thus, the bound \eqref{finalbound} holds even when non-regular $G$-extensions are included.  Indeed, one could include not only the field extensions, but the etale $G$-extensions, corresponding to covers of curves which are not even connected, and \eqref{finalbound} would still hold.

\end{proof}

The upper bound in the Weak Malle conjecture in its usual formulation when $G$ is a subgroup of $S_d$ follows immediately, by taking $w$ to be the index, as above.

\begin{cor}
Let $G$ be a transitive subgroup of $S_m$ and let $c$ be a conjugacy-invariant subset of $G \bs e$. Write $N_{G}^c(\F_q(t); X)$ be the number of degree $m$ extensions of $\F_q(t)$ whose Galois closure has Galois group $G$, whose local monodromy all lies in $c$, and whose discriminant is at most $X = q^r$.  Suppose further that $q$ is prime to $|G|$ and is at least $Q(G,c)$, the constant appearing in Theorem~\ref{weak_counting_thm}.  Then 
\beq
N_{G}^c(\F_q(t); X) = O(X^{a(G,c)} (\log X)^{d-1})
\eeq
where $a(G,c) = (\min_{g \in c} \ind(g))^{-1}$ is the constant predicted by the weak Malle conjecture, and where $d$ is the Gelfand-Kirillov dimension of $R$.   \
\label{cor:mainmalle}
\end{cor}

The full Malle conjecture would require that the power of $\log X$ appearing here be the number of conjugacy classes in $c$ realizing the minimal index (or, in the generalized discriminant setting, the minimal value of $w$.)  It seems to us that the methods of the present paper would need substantial refinement in order to get the power of $\log X$ correct in general, as we now explain.  Consider, for example, the case where $G = S_4$ and $c$ is the class of transpositions.  Then $c$ consists of only one conjugacy class, but if $H$ is the subgroup of $G$ generated by $(12)$ and $(34)$, then $c \cap H$ splits into two $H$-conjugacy classes, so the Gelfand-Kirillov dimension of $R$ in this case is not $2$ but $1$ and the upper bound provided by the theorem for quartic extensions of discriminant at most $X$ is on order $X \log X$, while the Malle conjecture prediction (and indeed, in this case, the proven fact) is that the number of $(G,c)$-extensions of $\F_q(t)$ of discriminant at most $X$ is on order $X$.  The problem is that our bound for the number of points on the Hurwitz space parametrizing geometrically connected $G$-covers comes from a bound for the number of points on the moduli space parametrizing {\em all} $G$-covers, connected or not.  In the case under discussion, this means that our bound counts not only quartic extensions with Galois group $S_4$, but also those with Galois group $H$; and in fact, the number of biquadratic extensions with discriminant at most $X$ actually {\em is} on order $X \log X$, and so dominates the term that the Malle conjecture actually asks us to count.  

However, this is in some sense the {\em only} obstacle to getting lower bounds of the same order of growth as our upper bounds, as the following corollary shows.

\begin{cor}  Let $G$ be a finite group, $c$ a conjugacy-invariant subset of $G \bs e$, and $w$ a conjugacy-invariant function on $c$.  Write $d$ for the number of conjugacy classes contained in $c$ which realize the minimal value of $w$ on $c$.  Suppose that, for any proper subgroup $H \subset G$, the number of $H$-conjugacy classes in $c \cap H$ is at most $d$.  Suppose that $q$ is prime to $|G|$ and greater than $Q(G,c)$ and $q-1$ is sufficiently divisible.  Then there exist constants $C_1, C_2$ such that
\beq
C_1 (X^{a(w,G,c)} (\log X)^{d-1}) < N_{G}^{c,w}(\F_q(t); X) < C_2 (X^{a(w,G,c)} (\log X)^{d-1})
\eeq
\label{co:lowerbounds}
\end{cor}

\begin{proof}
We note first that the condition of the corollary implies that $w$ is actually {\em constant} on $c$, since otherwise the number of conjugacy classes making up $c \cap H$ would be greater than $d$ when $H=G$.  So $a(w,G,c) = 1/w$.  We hope that in later work we will be able to relax this condition, e.g. by only needing control over the number of $w$-minimal conjugacy classes in $c \cap H$.  We also note that the condition of the corollary is related to the {\em splitting number} introduced by Seguin in \cite[Def 1.1]{seguin}.

We also note that, when $G$ is endowed with a permutation representation $G \inj S_m$ and $w$ is index, the Malle prediction would be that the power of log in $N_{G}^c(\F_q(t); X)$ is one less than the number of orbits on the $d$ conjugacy classes under the natural action of Frobenius already described; for $q-1$ divisible enough (say, divisible by $|G|$) this action is trivial, so the expected power of log is $d-1$.  In other words, when Corollary~\ref{co:lowerbounds} applies, it shows that the counting function for $G$-extensions is bounded above and below by constant multiples of the Malle prediction.

The notation $d$ for the number of conjugacy classes making up $c$ is justified by the fact that, given the hypothesis on the intersections $c \cap H$, the Gelfand-Kirillov dimension of $R$ is also $d$.  Thus, the upper bound to be proved is just a restatement of Theorem~\ref{th:mainmalle}, and the point of this corollary is to establish the lower bound.  We turn to this now.

First of all, since $w$ is constant, the number of $G$-extensions whose $w$-discriminant is {\em exactly} $q^{wn}$ is $|Z_G| |\Hn_{G,n}^c(\F_q)|$ by Proposition~\ref{pr:stackycount}.  This number is a lower bound for $N_G^{c,w}(\F_q(t); q^{wn})$.  So the lower bound required for the Corollary can be expressed as
\beq
C' q^n n^{d-1} < |\Hn_{G,n}^c(\F_q)|.
\eeq
for some constant $C'$.

Recall from the proof of Theorem~\ref{weak_counting_thm} that we have
\beq
\frac{\# \Hn_{G, n}^c(\F_q)}{ q^n}  =  q^{-n} \sum_{j=0}^{2n} (-1)^j \tr(\Frob | H^{2n-j}_{c, \et}(\Hn_{G, n}^c/\Fqbar, \Q_\ell)) 
\eeq
which we separate into
\begin{equation}
\label{eq:split}
q^{-n}(\tr(\Frob | H^{2n}_{c, \et}(\Hn_{G, n}^c/\Fqbar, \Q_\ell)) + \sum_{j=1}^{2n} (-1)^j \tr(\Frob | H^{2n-j}_{c, \et}(\Hn_{G, n}^c/\Fqbar, \Q_\ell)))
\end{equation}

As in the proof of Theorem~\ref{weak_counting_thm}, the second summand is bounded above by 
\beq
p(n) \sum_{j=1}^{\infty}  q^{-j/2} Y(G,c)^j
\eeq
which now (because we are starting at $j=1$ instead of $j=0$) is bounded above for all $q > Q(G,c)$ by
\beq
C q^{-1/2} p(n)
\eeq
for some constant $C = C(G,c)$.  In particular, since $p(n)$ is a polynomial of degree $d-1$ for large $n$, the second summand satisfies a bound of the form
\begin{equation}
q^{-n} \sum_{j=1}^{2n} (-1)^j \tr(\Frob | H^{2n-j}_{c, \et}(\Hn_{G, n}^c/\Fqbar, \Q_\ell))) \leq C q^{-1/2} n^{d-1}.
\label{eq:secondsummand}
\end{equation}

The new ingredient in this proof is to prove a lower bound for the first summand of \eqref{eq:split}.  We note that
\beq
q^{-n}(\tr(\Frob | H^{2n}_{c, \et}(\Hn_{G, n}^c/\Fqbar, \Q_\ell)))
\eeq
is precisely the trace of Frobenius in its permutation action on the geometrically connected components of $\Hn_{G, n}^c$.  We now show that we can make this trace large by ensuring that many of the components are fixed by Frobenius, as long as we make $q-1$ sufficiently divisible.

First of all, we recall that the geometrically connected components of $\Hn_{G, n}^c$ are precisely the orbits of the braid group on $n$-tuples
\beq
(g_1, \ldots, g_n)
\eeq
with each $g_i$ lying in $c$.  We will restrict our attention to those $n$-tuples satisfying
\begin{equation}
g_1 \cdot \ldots \cdot g_n = 1
\label{prod1}
\end{equation}
 a subset easily seen to be preserved by the braid action.  This set of components will already be enough to obtain a bound of the desired quality.

Write $c_1, \ldots, c_d$ for the conjugacy classes making up $c$.  As explained in \S \ref{ss:hurwitzschemes}, any $n$-branched $G$-cover of $\A^1$ has a local monodromy type, which is the $d$-tuple of natural numbers $\mathbf{n} = n_1, \ldots, n_d$ where $n_i$ is the number of branch points of type $c_i$.  Wood, in \cite{mmw:liftinginvariant}, defines a {\em lifting invariant} of $n$-branched covers satisfying \eqref{prod1}, valued in a Galois module $U'(G,c)$, which is constant on any geometrically connected component of $\Hn_{G, n}^c$ by \cite[\S 6.2]{mmw:liftinginvariant} and which is equivariant with the action of Galois on geometrically connected components~\cite[\S 6.1]{mmw:liftinginvariant}.  (In the notation of \cite{mmw:liftinginvariant}, $U'(G,c)$ is the kernel of the homomorphism $U(G,c) \ra G$.)

 More precisely:  $\Gal(\bar{\Q}/\Q)$ acts on the set $c_1, \ldots, c_d$ through the cyclotomic character $\chi: \Gal(\bar{\Q}/\Q) \ra \Zhat^*$:  a Galois element $\sigma$ sends $c_i$ to $c_i^{\chi(\omega)}$.  The free abelian group generated by the $c_i$ thus acquires the structure of a Galois module, which we call $A$.  There is a Galois-equivariant homomorphism $\pi: U'(G,c) \ra A$ with finite-index image $A_0$, whose kernel is a finite Galois module~\cite[\S 2, \S 5]{mmw:liftinginvariant}.  (As you might expect, if $u$ is the lifting invariant of a branched cover, $\pi(u)$ is its local monodromy type.)

The Galois action on $U'(G,c)$ also factors through the cyclotomic character. More specifically, if $\sigma$ is a Galois element fixing the $|G|^2$ roots of unity, $\sigma$ acts trivially on $U'(G,c)$~\cite[Remark 4.1]{mmw:liftinginvariant}.  In particular, if $q-1$ is a multiple of $|G|^2$, Frobenius acts trivially on $U'(G,c)$.

Now Theorem 3.1 of \cite{mmw:liftinginvariant} tells us that there is a constant $M(G,c)$ such that, for any element $a$ of $A_0$ in which each $c_i$ appears with multiplicity at least $M(G,c)$, the geometrically connected components of Hurwitz space with local monodromy type $a$ are actually in bijection with the elements of $U(G,c)$ lying over $a$.  In particular, once $|G|^2 | q-1$, this means that the geometrically connected components of Hurwitz space with local monodromy type $a$ are all defined over $\F_q$.  Thus, the number of $\F_q$-rational geometrically connected components of $\Hn_{G, n}^c$ is at least as large as the number of elements of $A_0$ with coordinates summing to $n$ all coordinates at least $M(G,c)$; this number, in turn, is bounded below by a constant multiple of $n^{d-1}$ for $n$ large enough.

We conclude that
\beq
q^{-n}(\tr(\Frob | H^{2n}_{c, \et}(\Hn_{G, n}^c/\Fqbar, \Q_\ell))) \geq C'' n^{d-1}
\eeq
for some positive constant $C''$ and all sufficiently large $n$. Combining this with \eqref{eq:secondsummand} and making $q$ large enough that $C q^{-1/2} n^{d-1} < C''$, we arrive at the desired result
\beq
\frac{\# \Hn_{G, n}^c(\F_q)}{ q^n} \geq C' n^{d-1}.
\eeq
\end{proof}

The conditions of Corollary~\ref{co:lowerbounds} are somewhat restrictive, but the result does provide lower bounds for some natural counting problems in arithmetic statistics over function fields.  In particular, we are able to give a growth rate for the number of quartics with Galois group $A_4$ and discriminant at most $X$, over function fields containing a cube root of unity.  Malle's conjecture would predict that the number of such quartics would be asymptotic to a multiple of $X^{1/2} (\log X)^2$.  The best known lower bound for number fields is on order $X^{1/2}$, due to Baily and to Alberts~\cite{alberts:solvable}, while the best known upper bound (proved for $A_4$-extensions of $\Q$ only) is on order $X^{0.7784 \ldots}$.\cite{bstttz}.  Over $\F_q(t)$, when $q$ is $1$ mod $3$, we can come within a constant of the Malle prediction.

\begin{prop} Let $N_{A_4}(\F_q(t);X)$ be the number of quartic extensions of $\F_q(t)$ whose Galois group is contained in $A_4$ and whose discriminant is at most $X$.  Then, for all sufficiently large odd $q$ congruent to $1$ mod $3$,
\beq
C_1 X^{1/2} (\log X)^2 \leq N_{A_4}(\F_q(t);X) \leq C_2 X^{1/2} (\log X)^2.
\eeq
\end{prop}

\begin{proof}
We note first of all that all three nontrivial conjugacy classes in $A_4$ have index $2$.  To show that the conditions of Corollary~\ref{co:lowerbounds} are satisfied, we need to check that no subgroup of $A_4$ has more than three nontrivial conjugacy classes; this is easy to check exhaustively, with the Klein $4$-group being the only proper subgroup that even has as many as three nontrivial conjugacy classes.  So Corollary~\ref{co:lowerbounds} applies with $d=3$ and $w$ the constant function $2$.  

All that remains is to check that the condition ``$q-1$ sufficiently divisible" can be made explicit as ``$q-1$ is a multiple of $3$."  Divisibility conditions on $q-1$ appear in two places in the proof of Corollary~\ref{co:lowerbounds}.  First, we need Frobenius to act trivially on the conjugacy classes making up $c$; to be precise, these are the conjugacy classes of $(123)$, of $(132)$, and of $(12)(34)$.  The Frobenius action of raising to the $q$ power always fixes the double-flip class; when $3|q-1$ it also fixes the two $3$-cycle classes.

What remains is to show that when $3|q-1$, Frobenius acts trivially on $U'(G,c)$.  For this, we need to look at the construction of $U'(G,c)$ a little more closely, following \S 2 of \cite{mmw:liftinginvariant}.  The Schur multiplier $H_2(A_4,\Z)$ of $A_4$ is a group of order $2$, and a central extension of $A_4$ realizing that Schur multiplier is $\SL_2(\F_3) \ra \PSL_2(\F_3) \cong A_4$.  One checks that the commutator of
\beq
\mat{1}{0}{0}{-1}, \mat{0}{1}{1}{0}
\eeq
in $\SL_2(\F_3)$ is $-I$, and thus generates the whole Schur multiplier; moreover, both of those involutions map to elements of $c$ in $A_4$; thus, the reduced Schur cover $S_c$ defined before Lemma 2.2 of \cite{mmw:liftinginvariant} is isomorphic to $G$.  Given that fact, \cite[Thm 2.5]{mmw:liftinginvariant} gives $U'(G,c)$ a very explicit description: the map $U'(G,c) \ra A$ is an isomorphism onto its image, which in this case is just the kernel of the homomorphism from $A$ to $\Z/3\Z$ which sends $(123)$ to $1$, $(132)$ to $2$, and $(12)(34)$ to $0$.  Combinatorially speaking, we are just saying that the braid group acts {\em transitively} on $n$-tuples $(g_1, \ldots, g_n)$ with a given local monodromy type and satisfying $g_1 \cdot \ldots \cdot g_n = 1$, so long as the local monodromy type contains each nontrivial conjugacy class of $A_4$ sufficiently many times.

In particular, the action of Frobenius on $U'(G,c)$ is trivial once $3 | q-1$.   This completes the proof.

\end{proof}

\bibliography{biblio}

\end{document}